\documentclass[aos]{imsart}
\usepackage{amsthm,amsmath,amssymb}
\usepackage{natbib}
\usepackage{multirow}
\usepackage[pdftex]{graphicx}
\usepackage{caption}
\usepackage{subcaption}
\usepackage{makecell}
\usepackage{booktabs}
\usepackage{array}
\usepackage{nameref}
\usepackage{url}
\usepackage{algorithm}
\usepackage{mathtools}
\usepackage{bm}
\usepackage{bbm}
\usepackage{float}
\usepackage{smile}
\usepackage{lipsum}
\usepackage{mathrsfs}
\usepackage{dsfont}
\usepackage{epstopdf}
\usepackage{caption}
\usepackage{marvosym}
\usepackage{diagbox}
\renewcommand*{\liminf}{\displaystyle \operatornamewithlimits{liminf}}
\renewcommand*{\limsup}{\displaystyle \operatornamewithlimits{limsup}}

\usepackage{natbib}
\usepackage{multirow}

\usepackage[usenames,dvipsnames,svgnames,table]{xcolor}
\usepackage[colorlinks=true,
            linkcolor=blue,
            urlcolor=blue,
            citecolor=blue]{hyperref}

\usepackage[utf8]{inputenc}
\usepackage[usenames,dvipsnames]{xcolor}
\usepackage{tkz-berge}
\definecolor{fondpaille}{cmyk}{0,0,0.1,0}

\makeatletter
\newcommand*{\independent}{
  \mathbin{
    \mathpalette{\@indep}{}
  }
}
\newcommand*{\@indep}[2]{
  \sbox0{$#1\perp\m@th$}
  \sbox2{$#1=$}
  \sbox4{$#1\vcenter{}$}
  \rlap{\copy0}
  \dimen@=\dimexpr\ht2-\ht4-.2pt\relax
  \kern\dimen@
  {#2}
  \kern\dimen@
  \copy0
} 
\makeatother

\newcommand*{\TV}{\operatorname{TV}}
\newcommand*{\Tr}{\operatorname{Tr}}

\newcommand*{\fm}{\mathfrak{m}}

\newcommand*{\supp}{\mathrm{supp}}

\arxiv{arXiv:1608.03045}

\startlocaldefs
\endlocaldefs

\begin{document}

\begin{frontmatter}

\title{Combinatorial Inference for Graphical Models}
\runtitle{Combinatorial Inference for Graphical Models}

\begin{aug}
 \author{\fnms{Matey} \snm{Neykov}\corref{}\ead[label=e1]{mneykov@stat.cmu.edu}}\thanksref{t1},
 \author{\fnms{Junwei} \snm{Lu}\corref{}\ead[label=e2]{junweil@princeton.edu}}\thanksref{t1} \and
 \author{\fnms{Han} \snm{Liu}\corref{}\ead[label=e3]{hanliu@northwestern.edu}}\thanksref{t2},
   \address[a]{Department of Statistics,\\ Carnegie Mellon University, \\ Pittsburgh, PA 15213, USA  \\ \printead{e1}\\}
  \address[b]{Department of Operations Research\\and Financial Engineering,\\ Princeton University, \\ Princeton, NJ 08540, USA  \\\printead{e2}\\}
  \address[c]{Department of Electrical Engineering \\ and Computer Science\\ Department of Statistics \\ Northwestern University\\ Evanston, IL 60208, USA\\\printead{e3}}
\end{aug}
  \affiliation{Carnegie Mellon University\\ Princeton University\\ Northwestern University}

 \thankstext{t1}{These authors contributed equally to this work.}
 \thankstext{t2}{Research supported by NSF DMS1454377-CAREER; NSF IIS 1546482-BIGDATA; NIH R01MH102339; NSF IIS1408910; NIH R01GM083084.}

\runauthor{Neykov, Lu and Liu}

\begin{abstract}
We propose a new family of combinatorial inference  problems for graphical models. Unlike classical statistical inference where the main interest is point estimation or parameter testing, combinatorial inference aims at testing the global structure of the underlying graph. Examples include testing the graph  connectivity, the presence of a cycle of certain size, or the maximum degree of the graph. To begin with, we study the information-theoretic limits of a large family of combinatorial inference problems. We propose new concepts including structural packing and buffer entropies to characterize how the complexity of combinatorial graph structures impacts the corresponding minimax lower bounds. On the other hand, we propose a family of novel and practical structural testing algorithms to match the lower bounds. We provide numerical results on both synthetic graphical models and brain networks to illustrate the usefulness of these proposed methods.
\end{abstract}
 
\begin{keyword}
\kwd{Graph structural inference}
\kwd{minimax testing}
\kwd{uncertainty assessment}
\kwd{multiple hypothesis testing}
\kwd{post-regularization inference}
\end{keyword}

\end{frontmatter}

\section{Introduction} \label{intro}

Graphical model is a powerful tool for modeling complex relationships among many random variables. A central theme of graphical model research is to infer the structure of the underlying graph based on observational  data. Though significant progress has been made, existing works mainly focus  on estimating the graph \citep{Meinshausen2006High,Liu2009Nonparanormal:,Ravikumar2011High,Cai2011Constrained} or testing the existence of a single edge \citep{jankova2014confidence,ren2015asymptotic, neykov2015aunified, gu2015local}. 

In this paper we consider a new inferential problem: testing the combinatorial structure of the underlying graph. Examples include testing the graph  connectivity, cycle presence, or assessing the maximum degree of the graph. Unlike classical inference which aims at testing a set of Euclidean parameters, combinatorial inference aims to test some global structural properties and requires the development of new methodology. As for methodological development, this paper mainly considers the Gaussian graphical model (though our method is applicable to the more general semiparametric exponential family graphical models and elliptical copula graphical models): Let  $\bX = (X_1,\ldots, X_d)^T \sim N_d\bigl(\bm{0}, (\bTheta^*)^{-1}\bigr)$ be a $d$-dimensional Gaussian random vector with  precision matrix $\bTheta^*=(\Theta^*_{jk})$. Let $G^*= G(\bTheta^*) := (\overline V, E^*)$ be an undirected graph, where  $\overline V = \{1,\ldots, d\}$ and an edge $(j,k) \in E^*$ if and only if $\Theta^*_{jk}\neq 0$. It is well known that $G^*$ has the pairwise Markov property, i.e.,    $(j,k) \not \in E^*$ if and only if $X_j$ and $X_k$ are conditionally independent given the remaining variables. In a combinatorial inference problem, our goal is to test whether $G^*$  has certain global structural properties based on $n$ random samples $\bX_1,\ldots, \bX_n$. Specifically,  let $\cG$ be the set of all graphs over the vertex set $\overline V$ and $(\cG_0, \cG_1)$ be a pair of non-overlapping subsets of $\cG$. We assume all the graphs in $\cG_1$ have a  property (e.g., connectivity) while the graphs in $\cG_0$ do not have this property.  Such a pair $(\cG_0, \cG_1)$ is called a sub-decomposition of $\cG$.  
Our goal is to test the hypothesis $\Hb_0: G^*\in \cG_0$ versus $\Hb_1: G^*\in \cG_1$.
We provide several concrete examples below. \vspace{.2cm}

\noindent{\bf Connectivity.}  A graph is connected if and only if there exists a path connecting each pair of its vertices. To test connectivity, we set $\cG_0= \{G \in \cG ~|~ G \text{ is disconnected} \}$ and $\cG_1= \{G \in \cG ~|~ G \text{ is connected} \}$. Under the Gaussian graphical model, this is equivalent to testing whether the variables can be partitioned into at least two  independent sets. \vspace*{.2 cm}

\noindent{\bf Cycle presence.} Sometimes it is of interest to test whether the underlying graph is a forest. In this example we let $\cG_0 = \{G \in \cG ~|~ G \mbox{  is a forest}\}$ and $\cG_1 = \{G \in \cG ~|~ G \mbox{  contains a cycle}\}$. If a graph is a forest, it can be easily visualized on a two dimensional plane. \vspace{.2 cm}

\noindent{\bf Maximum degree.} Another relevant question is to test whether the maximum degree of the graph is less than or equal to some integer $s_0 \in \NN$ versus the alternative that the maximum degree is at least $s_1 \in \NN$, where $s_0 < s_1$. Define the sub-decomposition $\cG_0 = \{G ~|~ d_{\max}(G) \leq s_0\}$ and $\cG_1 = \{G~|~ d_{\max}(G) \geq s_1\}$ respectively.\vspace{.2 cm}

While our ultimate goal is to test whether $G^* \in \cG_0$ versus $G^* \in \cG_1$, our access to $G^*$ is only through the random samples $\{\bX_i\}_{i = 1}^n$. Under Gaussian models,  we can translate the original problem of testing  graphs to  testing the precision matrix:  
\begin{align}\label{prec:matrix:prob}
\Hb_0: \bTheta^* \in \cS_0 \mbox{ vs } \Hb_1: \bTheta^* \in \cS_1.
\end{align}
In (\ref{prec:matrix:prob}), $\cS_0, \cS_1 \subset \cM(s)$ are two sets of precision matrices such that for all $\bTheta \in \cS_0, \cS_1$ we have $G(\bTheta) \in \cG_0, \cG_1$ respectively, and $\cM(s)$ is defined as
\begin{align}
\cM(s) = \Big\{\bTheta ~\Big |~ \bTheta = \bTheta^T, C^{-1} \leq \bTheta \leq C, \|\bTheta\|_1 \leq L, \max_{j \in [d]}\|\bTheta_{*j}\|_0 \leq s \Big\}\label{MS:def}
\end{align} 
for some constants $1 \leq C \leq L$. The inequalities $C^{-1} \leq \bTheta \leq C$ in (\ref{MS:def}) are meant in a ``positive-semidefinite sense'', i.e., the minimum and maximum eigenvalues of $\bTheta$ are assumed to be bounded by $C^{-1}$ and $C$ from below and above respectively, and $\|\bTheta_{*j}\|_0$ is the cardinality of the non-zero entries of the $j$\textsuperscript{th} column of $\bTheta$ (see Section \ref{sec:notation} for precise notation). The set $\cM(s)$ restricts our attention to well conditioned symmetric matrices $\bTheta$, whose induced graphs $G(\bTheta)$ have maximum degree of at most $s$. Given this setup, we aim to characterize necessary conditions on the pair $\cS_0, \cS_1$ under which the combinatorial inference problem  in (\ref{prec:matrix:prob}) is \textit{testable}. Specifically, recall that a test is any measurable function $\psi: \{\bX_i\}_{i = 1}^n \mapsto \{0,1\}$. Define the minimax risk of testing $\cS_0$ against $\cS_1$ as:
\begin{align}
\gamma(\cS_0, \cS_1) = \inf_{\psi} \Bigl[ \max_{\bTheta \in \cS_0} \PP_{\bTheta}(\psi = 1) + \max_{\bTheta \in \cS_1} \PP_{\bTheta}(\psi = 0) \Bigr]. \label{risk:def}
\end{align}
If $\liminf_{n \rightarrow \infty} \gamma(\cS_0, \cS_1) = 1$, we say that the problem (\ref{prec:matrix:prob}) is \textit{untestable} since any test fails to distinguish between $\cS_0$ and $\cS_1$ in the asymptotic minimax sense. We are specifically interested in an asymptotic setting where the dimension $d$ is a function of the sample size, i.e., $d = d(n)$ so that $d \rightarrow \infty$ as $n \rightarrow \infty$. This setting will be implicitly understood throughout the paper. Due to the close relationship between the sets of precision matrices $(\cS_0, \cS_1)$ and the sub-decomposition $(\cG_0,\cG_1)$ (recall that for all $\bTheta \in \cS_0, \cS_1$ we have $G(\bTheta) \in \cG_0, \cG_1$ resp.), we anticipate that the sub-decomposition can capture the intrinsic challenge of the test in (\ref{prec:matrix:prob}). Indeed, in Sections \ref{sec:generic:lower} and \ref{general:lower:bound} we develop a framework capable of capturing the impact of the combinatorial structures of $\cG_0$ and $\cG_1$ to the lower bound $\gamma(\cS_0, \cS_1)$. Such lower bounds provide necessary conditions for any valid test. We then develop practical procedures that match the obtained lower bounds. To understand how the sub-decomposition $(\cG_0,\cG_1)$ affects the intrinsic difficulty of the problem in (\ref{prec:matrix:prob}), we consider the three examples given before. Our lower bound framework distinguishes between two types of sub-decompositions --- in the first type one can find graphs belonging to $\cG_0$ and $\cG_1$ differing in only one single edge, while in the second type all graphs belonging to $\cG_0$ must differ on multiple edge sets from the graphs belonging to $\cG_1$.

One can check that in the first two examples (connectivity and cycle presence testing) there exist graphs belonging to $\cG_0$ and $\cG_1$ differing in only one single edge. For instance, when testing connectivity, consider a tree  with a single edge removed (thus it becomes a forest) versus a connected tree. Extending this intuition, for a fixed graph $G_0 = (\overline V, E_0) \in \cG_0$, we call the edge set $\cC = \{e_1, \ldots, e_m\}$ a single-edge \textit{null-alternative divider}, or simply a divider for short, if for all edges $e \in \cC$ the graphs $(\overline V, E_0 \cup \{e\}) \in \cG_1$. Intuitively the bigger the cardinality of a divider is, the harder it is to tell the null from the alternatives. In Section \ref{sec:generic:lower} we detail that $\gamma(\cS_0, \cS_1)$ is asymptotically $1$, when the signal strength of separation between $\cS_0$ and $\cS_1$ is low (see   (\ref{generic:S0:def}) and (\ref{generic:S1:def}) for a formal definition) and there exists a divider with sufficiently large packing number. The packing number, formally defined in Definition \ref{structure:entropy:packing}, represents the cardinality of a subset of edges in $\cC$ which are ``far'' apart, where the proximity measure of two edges is a predistance (compared to distance, a predistance does not have to satisfy the triangle inequality) based on the graph $G_0$. Recall that for a graph $G$ and two vertices $u$ and $v$, a graph geodesic distance is defined by:
$$d_G(u, v) := \mbox{ length of the shortest path between } u \mbox{ and } v \mbox{ within } G.$$
Using the notion of geodesic distance, one can define a predistance between two edges, by taking the minimum over the geodesic distances of their corresponding nodes. 

If the  difference between null and alternative is more than one edge, as in the maximum degree testing $\cG_0 = \{G~|~ d_{\max}(G) \leq s_0\}$ vs $\cG_1 = \{G~|~ d_{\max}(G) \geq s_1\}$ for example, the packing number does not always capture the lower bound of the tests. In Section \ref{general:lower:bound}, we develop a novel mechanism to handle this more sophisticated case. We introduce a concept called ``buffer entropy" which can overcome the disadvantages of the  packing number and produce sharper lower bounds. 

On the other hand, to match the lower bound, we propose the alternative witness test as a general algorithm for combinatorial testing. Our algorithm identifies a critical structure and proceeds to test whether this structure indeed belongs to the true graph.  We prove that alternative witness tests can control both the type I and type II errors asymptotically. 

\subsection{Contributions}

There are three major contributions of this paper.

Our first contribution is to develop a novel strategy for obtaining minimax lower bounds on the signal strength required to distinguish combinatorial graph structures which are separable via a single-edge divider. In particular, we relate the information-theoretic lower bounds to the packing number of the divider, which is an intuitive combinatorial quantity. To obtain this connection, we relate the chi-square divergence between two probability measures to the number of ``closed walks" on their corresponding Markov graphs. Our analysis hinges on several technical tools including Le Cam's Lemma, matrix perturbation inequalities and spectral graph theory. The usefulness of the approach is demonstrated by obtaining generic and interpretable lower bounds in numerous examples such as testing connectivity, connected components, self-avoiding paths, and cycles.

Our second contribution is to provide a device for proving lower bounds under the settings where the null and alternative graphs differ in multiple edges. Under such case, the packing number does not always provide tight lower bounds. In order  to overcome this issue we formalize a graph quantity called buffer entropy. The buffer entropy is a complexity measure of the structural tests and provides lower bounds. We apply buffer entropy to derive lower bounds for testing the maximum degree and detecting a sparse clique and cycles. 

Our third contribution is to propose an alternative witness test (\ref{prec:matrix:prob}), which matches the lower bounds on the signal strength. Our algorithm works on sub-decompositions which are stable with respect to addition of edges, i.e., given a graph $G \in \cG_1$ adding edges to $G$ yields graphs which belong to $\cG_1$. The alternative witness test is a two step procedure --- in the first step it identifies a minimal structure ``witnessing'' the alternative hypothesis, and in the second step it attempts to certify the presence of this structure in the graph. The alternative witness test utilizes recent advances in high-dimensional inference and provides honest tests for combinatorial inference problems. It has two advantages compared to the support recovery procedures in \cite{Meinshausen2006High, Ravikumar2011High,Cai2011Constrained}: First, it allows us to control the type I error at any given level; second, it does not require perfect recovery of the underlying graph to conduct valid inference.

\subsection{Related Work}

Graphical model inference is relatively straightforward when $d<n$, but becomes notoriously challenging when $d \gg n$. In high-dimensions, estimation procedures were studied by \cite{Yuan2006Model, friedman2008sparse, Lam2009Sparsistency, Cai2011Constrained} among others, while for variable selection procedures see \cite{Meinshausen2006High, raskutti2008model,Liu2009Nonparanormal:, Ravikumar2011High, Cai2011Constrained} and references therein. Recently, motivated by \cite{zhang2014confidence}, various  inferential methods for high-dimensional graphical models were suggested \citep[e.g.]{liu2013, jankova2014confidence,chen2015asymptotically,ren2015asymptotic, neykov2015aunified, gu2015local}, most of which focus on testing the presence of a single edge (except \cite{liu2013} who took the FDR approach \citep{benjamini1995controlling} to conduct multiple tests and \cite{gu2015local} who developed procedures of edge testing in Gaussian copula models). None of the aforementioned works address the problem of combinatorial structure testing.

In addition to estimation and model selection procedures, efforts have been made to understand the fundamental limits of these problems. Lower bounds on estimation were obtained by \cite{ren2015asymptotic}, where the authors show that the parametric estimation rate $n^{-1/2}$ is unattainable unless $s \log d/\sqrt{n} = o(1)$. Lower bounds on the minimal sample size required for model selection in Ising models were established by \cite{santhanam2012information}, where it is shown that support recovery is unattainable when $n \ll s^2 \log d$. In a follow up work, \cite{wang2010information} studied model selection limits on the sample size in Gaussian graphical models. The latter two works are remotely related to ours, in that both works exploit  graph properties to obtain information-theoretic lower bounds. However, our problem differs significantly from theirs since we focus on developing lower bounds for testing graph structure, which is a fundamentally different problem from estimating the whole graph. 

Our problem is most closely related to those in \cite{addario2010comb, arias2012detection, arias2015detectingpositive, arias2015detecting}, which are inspired by the large body of research on minimax hypothesis testing \citep[e.g.]{ingster1982minimax, Ingster2010Detection, arias2011detection, Arias-Castro2011Global} among many others. \cite{addario2010comb} quantify the signal strength as the mean parameter of a standard Gaussian distribution, while \cite{arias2012detection, arias2015detectingpositive} impose models on the covariance matrix of a multivariate Gaussian distribution. In our setup the parameter spaces of interest are designed to reflect the graphical model structure, and hence the signal strength is naturally imposed on the precision matrix. \cite{Arias-Castro2011Global} provide detection bounds for the linear model. This is related to our work since one can view a linear model with Gaussian design as a Gaussian graphical model. \cite{arias2015detecting} address testing on a lattice based Gaussian Markov random field. For specific problems they establish lower bounds on the signal strength required to test the empty graph versus an alternative hypothesis. This is different from the setting of our problems, where the null hypothesis is usually not the empty graph. 

\subsection{Notation}\label{sec:notation}

The following notation is used throughout the paper. For a vector $\vb = (v_1, \ldots, v_d)^T\in \RR^d$, let $\|\vb\|_q = (\sum_{i = 1}^d v_i^q)^{1/q},  1 \leq q < \infty$, $\|\vb\|_0 = | \supp(\vb)|$, where $\supp(\vb) = \{j~|~ v_j \neq 0\}$, and $|A|$ denotes the cardinality of a set $A$. Furthermore, let $\|\vb\|_{\infty} = \max_{i} |v_i|$ and $\vb^{\otimes 2} = \vb \vb^T$. For a matrix $\Ab$, we  denote  $\Ab_{*j}$ and $\Ab_{j*}$ to be the $j$\textsuperscript{th} column and row of $\Ab$ respectively. For any $n \in \NN$ we use the shorthand notation $[n] = \{1,\ldots, n\}$. For two integer sets $S_1, S_2 \subseteq [d]$, we denote  $\Ab_{S_1S_2}$ to be the sub-matrix of $\Ab$ with elements $\{A_{jk}\}_{j \in S_1, k \in S_2}$. Moreover, we denote $\|\Ab\|_{\max} = \max_{jk} |A_{jk}|$, $\|\Ab\|_p = \max_{\|\vb\|_p = 1} \|\Ab \vb\|_p$ for $p \geq 1$. For a symmetric matrix $\Ab \in \RR^{d\times d}$ and a constants $c,C$, with a slight abuse of notation we write $c \leq \Ab \leq C$ to mean that the matrices $\Ab - c \Ib_d$ and $C \Ib_d - \Ab$ are positive semi-definite, where $\Ib_d$ denotes the $d \times d$ identity matrix.

For a graph $G$ we use $V(G)$, $E(G), d_{\max}(G)$ to refer to the vertex set, edge set and maximum degree of $G$ respectively. We also denote $V(E)$ as the vertex set of the edge set $E$.  We reserve special notation for the complete vertex set $\overline V := [d]$, the complete edge set $\overline E :=  \{e \in 2^{[d]} ~|~ |e| = 2\}$ and the complete graph $\overline G:= (\overline V, \overline E)$. For two integers $j, k \in [d]$ we use unordered pairs $(j, k) = (k, j)$ to denote undirected edges between vertex $k$ and vertex $j$. Any symmetric matrix $\Ab \in \RR^{d\times d}$ naturally induces an undirected graph $G(\Ab)$, with vertices in the set $\overline V$ and edge set $E(G(\Ab)) = \{(j,k ) ~|~ A_{jk} \neq 0, j \neq k\}$. Additionally if $E$ is an arbitrary edge set (i.e., $E \subseteq \overline E$) for $e := (j,k) \in E$ we use the notation $A_e = A_{jk} =A_{kj}$ interchangeably to denote the element $e$ of the matrix $\Ab$.

Given two sequences $\{a_n\}, \{b_n\}$ we write $a_n = O(b_n)$ if there exists a constant $C < \infty$ such that $a_n \leq C b_n$; $a_n = o(b_n)$ if $a_n/b_n \rightarrow 0$, and $a_n \asymp b_n$ if there exists positive constants $c$ and $C$ such that $c < a_n/b_n < C$. 
Finally we use the shorthand notation $\wedge$ and $\vee$ for $\min$ and $\max$ of two numbers respectively.
\subsection{Organization of the Paper}

The paper is structured as follows. A lower bound on single edge dividers along with applications to several examples is presented in Section \ref{sec:generic:lower}. In Section \ref{sparse:graph:models:inf} we outline the alternative witness test, and illustrate how to apply it to the examples considered in Section \ref{sec:generic:lower}. In Section \ref{general:lower:bound} we generalize the lower bounds strategies from the single edge divider to multiple edge divider stetting. A brief discussion is provided in Section \ref{disc:sec}. Full proof of the main result of Section \ref{sec:generic:lower} is presented in Section \ref{proof:of:suff:lower:bound:sec}. Numerical studies, real data analysis and all remaining proofs are deferred to the supplement.

\section{Single-Edge Null-Alternative Dividers} \label{sec:generic:lower}

In this section we derive a novel and generic lower bound strategy, applicable to null and alternative hypotheses which differ in one single edge: i.e., under the Gaussian model, there exist two matrices $\bTheta_0 \in \cS_0$ and $\bTheta_1 \in \cS_1$ whose induced graphs $G_0 := G(\bTheta_0)$ and $G_1 := G(\bTheta_1)$ differ in a single edge. We formalize this concept in the definition below.

\begin{definition}[Single-Edge Null-Alternative Divider]\label{single:edge:null:alt:sep} For a sub-decom-position  $(\cG_0, \cG_1)$ of $\cG$, let $G_0 = (\overline V, E_0) \in \cG_0$ be a graph under the null. We refer to an edge set $
\cC = \{e_1, \ldots, e_m\},$ as a (single-edge) \textit{null-alternative divider} with the null base $G_0$ if for any $e \in \cC$ the graphs $G_e := (\overline V, E_0 \cup \{e\}) \in \cG_1$.
\end{definition}

As remarked in the introduction, if a large divider exists, it is expected that differentiating $G_0$ from an alternative graph $G_e\in\cG_1$ is more challenging. Indeed, our main result of this section confirms this intuition. We proceed to define a predistance for a graph $G$ and two edges $e, e'$ (which need not belong to $E(G)$) which plays a key role in our lower bound result. 

\begin{definition}[Edge Geodesic Predistance]\label{closed:walks:def} Let $G = (\overline V, E)$ and $\{e, e'\}$ be a pair of edges ($e$ and $e'$ may or may not belong to $E$). We define
$$
d_G(e, e') := \min_{u \in e, v \in e'} d_{G}(u, v),
$$
where $d_{G}(u, v)$ denotes the geodesic distance between vertices $u$ and $v$ on the augmented graph $G$. If such a path does not exist $d_G(e,e') = \infty$.
\end{definition}
By definition $d_G(e,e')$ is a predistance, i.e., $d_G(e,e) = 0$ and $d_G(e, e') \geq 0$. Moreover, $d_G(e, e')$ has the same value regardless of whether $e, e' \in E(G)$. See Fig \ref{conn:graph:ex2} for an illustration of  $d_G(e, e')$.
Inspired by the classical concept of packing entropy on metric spaces \citep[e.g.,][]{yang1999information} we propose the structural packing entropy for graphs in an attempt to characterize information-theoretic lower bounds for combinatorial inference.

\begin{definition}[Structural Packing Entropy]\label{structure:entropy:packing}
Let $\cC$ be a non-empty edge set and $G$ be a graph. For any $r \geq 0$ we call the edge set $N_r \subset \cC$ an $r$-packing of $\cC$ if for any $e, e' \in N_r$ we have $d_{G}(e, e') \geq r$. Define the structural $r$-packing entropy as:
\begin{align}
M(\cC, d_{G}, r) := \log \max \big\{ |N_r| \mid {N_r \subset \cC, N_r \mbox{ is a $r$-packing of $\cC$} }\big\}. \label{struc:entropy}
\end{align}
\end{definition}

The packing entropy in Definition \ref{structure:entropy:packing} is an analog to the classical  packing entropy on metric spaces in the sense that it is defined over an edge set $\cC$ equipped with a predistance $d_G(e,e')$ based on the graph $G$. 

To study minimax lower bounds, we only need to focus on the Gaussian graphical model whose structural properties are completely characterized by the precision matrices.  We now formally define the sets of precision matrices $\cS_0$ and $\cS_1$ used in this section. Let:
\begin{align}
\cS_0(\theta, s) &:= \Big \{\bTheta\in  \cM(s) ~|~ G(\bTheta) \in \cG_{0}, \min_{e \in E(G(\bTheta))} |\Theta_e| \geq \theta \Big \} \mbox{ and, }\label{generic:S0:def} \\
\cS_1(\theta, s) &:= \Big \{\bTheta\in  \cM(s)  ~|~ G(\bTheta) \in \cG_{1}, \min_{e \in E(G(\bTheta))} |\Theta_e| \geq \theta \Big \}, \label{generic:S1:def}
\end{align}
where $\cM(s)$ is defined in (\ref{MS:def}). The parameter $\theta$ in the definitions of $\cS_0(\theta, s)$ and $\cS_{1}(\theta, s)$ denotes the \textit{signal strength}, and as we show below, its magnitude plays an important role in determining whether one can distinguish between graphical models  in $\cS_{0}(\theta, s)$ and $\cS_{1}(\theta, s)$. 

\begin{theorem}[Necessary Signal Strength] \label{suff:cond:lower:bound} Let $D$ be a fixed integer. Suppose that 
\begin{align}\label{scaling:conditions}
\theta \leq \max_{\substack{G_0 \in \cG_0~:~ d_{\max}(G_0) \leq D, \\ \cC ~\mbox{\tiny divider with null base} ~ G_0}} \kappa \sqrt{\frac{M(\cC, d_{G_0}, \log |\cC|)}{n}} \wedge \frac{(1 - C^{-1})\wedge e^{-\frac{1}{2}}}{\sqrt{2}(D + 2)},
\end{align}
where $C$ is defined in \eqref{MS:def}.
Then if $M(\cC, d_{G_0}, \log |\cC|) \rightarrow \infty$ as $n \rightarrow \infty$, there exists a sufficiently small constant $\kappa$ in \eqref{scaling:conditions} (depending on $D, C, L$) such that $\liminf_{n \rightarrow \infty} \gamma(\cS_0(\theta, s), \cS_1(\theta, s)) = 1$.
\end{theorem}
Theorem \ref{suff:cond:lower:bound} allows us to quantify the signal strength necessary for combinatorial inference via combinatorial constructions. The radius $\log|\cC|$ of the packing entropy in (\ref{scaling:conditions}) ensures that the pairs of distinct edges are sufficiently far apart. The constant term $\frac{(1 - C^{-1})\wedge e^{-\frac{1}{2}}}{\sqrt{2}(D + 2)}$ in (\ref{scaling:conditions}) ensures that precision matrices with signal strength $\theta$ indeed belong to  $\cM(s)$. 

In Section \ref{edge:removal:NAS:extension} of the supplement we also provide a \textit{deletion-edge} version of Theorem \ref{suff:cond:lower:bound}, which proceeds in the opposite direction, i.e., it starts from an alternative graph $G_1$ and deletes edges from the divider $\cC$ to produce graphs under the null hypothesis. This strategy can yield sharper results than Theorem \ref{suff:cond:lower:bound} in certain situations, and we illustrate this with two examples in Section \ref{edge:removal:NAS:extension}.

\begin{proof}[Proof Sketch] The proof of Theorem \ref{suff:cond:lower:bound} can roughly be divided into four steps. Full details of the proof will be provided in Section \ref{proof:of:suff:lower:bound:sec}. 

\vspace{0.1in}
{\bf Step 1} ({\it Connect the structural parameters to geometric parameters}).  Given the adjacency matrices of the null and alternative graphs $G_0$ and $\{G_e\}_{e \in \cC}$, we construct the corresponding precision matrices and make sure that they belong to $\cS_0(\theta,s)$ and $\cS_1(\theta,s)$.
\vspace{0.1in}

{\bf Step 2} ({\it Construct minimax risk lower bound via Le Cam's method}). The second step uses Le Cam's method to lower bound $\gamma(\cS_0(\theta, s), \cS_1(\theta, s))$. This requires us to evaluate the chi-square divergence between a normal and a mixture normal distribution.  The chi-square divergence can be expressed via ratios of determinants. In particular, we show that the log chi-square divergence can be equivalently re-expressed via an infinite sum of differences among trace operators of adjacency matrix powers.

\vspace{0.1in}

{\bf Step 3} ({\it Represent the lower bound by the number of shortest closed walks in the graph}). 
 In this step we control the deviations of the differences of the trace operators. Since the trace of the power of an adjacency matrix equals the number of closed walks within the corresponding graph, we eliminate the trace powers which are smaller than the shortest closed walks. The traces of the higher powers are handled via matrix perturbation bounds.
 
 \vspace{0.1in}
 {\bf Step 4} ({\it Characterize the smallest magnitude of the geometric parameter using the packing entropy}). 
Lastly, we show that condition (\ref{scaling:conditions}) ensures that the closed walks on the packing of the divider are sufficiently lengthy, which implies that the chi-square divergence vanishes when the signal strength $\theta$ is small. 
\end{proof}

A typical application of Theorem \ref{suff:cond:lower:bound} proceeds by constructing a graph $G_0$ under the null hypothesis, which is one edge apart from the alternative. Next, one builds a divider $\cC$ with as large as possible packing number, so that \textit{adding} any edge from $\cC$ to $G_0$ results in an alternative graph. Clearly choosing the graph $G_0$ is crucial for this strategy to work. Below we give several examples of explicit constructions of $G_0$ and divider. At the end of the section we also provide somewhat general guidance how to select $G_0$.
\subsection{Some Applications} \label{some:applications}

In this section we give several examples of combinatorial testing, which readily fall into the framework developed in Section \ref{sec:generic:lower}. Although some of the examples We show one more additional example on self-avoiding paths in Section \ref{edge:removal:NAS:extension} of the supplement.

\begin{example}[Connectivity Testing]\label{conn:testing} 
Consider the sub-decomposition $\cG_0 = \{G \in \cG~|~ G \mbox{ disconnected} \}$ vs $\cG_1 = \{G \in \cG~|~ G \mbox{ connected}\}$. We construct a base graph $G_0 := (\overline V, E_0)$ where 
$$E_0 := \{(j, j+1)_{j = 1}^{\lfloor d/2\rfloor - 1}, (\lfloor d/2\rfloor,1), (j, j +1)_{j = \lfloor d/2\rfloor + 1}^{d}, (\lfloor d/2\rfloor + 1, d)\},$$
and let $\cC := \{(j, \lfloor d/2\rfloor + j)_{j = 1}^{\lfloor d/2\rfloor}\} \mbox{ (see Fig \ref{conn:graph:ex2}).}$
Clearly adding any edge from $\cC$ to $G_0$ connects the graph, so $\cC$ is a single edge divider with a null base $G_0$. Furthermore, the maximum degree of $G_0$ equals $2$ by construction. To construct a packing set of $\cC$, we collect all edges $(j, \lfloor d/2\rfloor + j)$ satisfying $\lceil \log |\cC|\rceil$ divides $j$ except if $j > \lfloor d/2\rfloor - \lceil \log |\cC|\rceil$. This procedure results in a packing set with radius at least $\lceil \log |\cC|\rceil $ which has cardinality of at least $\Big \lfloor \frac{|\cC|}{\lceil \log |\cC|\rceil}\Big\rfloor -1$. Therefore 
$$M(\cC, d_{G_0}, \log |\cC|) \geq \log \Big[\Big \lfloor \frac{|\cC|}{\lceil \log |\cC|\rceil}\Big\rfloor -1\Big] \asymp \log |\cC| \asymp \log d.$$ 
Theorem \ref{suff:cond:lower:bound} implies that the asymptotic minimax risk is $1$ if $\theta < \kappa \sqrt{\log d/n} \wedge \frac{(1 - C^{-1})\wedge e^{-\frac{1}{2}}}{4\sqrt{2}}$. 
\begin{figure}
\centering
\begin{tikzpicture}[scale=.6]
\SetVertexNormal[Shape  = circle,
                  FillColor  = cyan!50,
                  MinSize = 11pt,
                  InnerSep=0pt,
                  LineWidth = .5pt]
   \SetVertexNoLabel
   \tikzset{EdgeStyle/.style= {thin,
                                double          = red!50,
                                double distance = 1pt}}
                                \begin{scope}[rotate=90]
                                \grEmptyCycle[prefix=a,RA=3]{5}{1}%
                                \begin{scope}\grEmptyCycle[prefix=b,RA=1.5]{5}{1}\end{scope} \end{scope}
       \tikzset{EdgeStyle/.style= {dashed,thin,
                                double          = red!50,
                                double distance = 1pt}}
                                                                       \tikzset{LabelStyle/.style = {below, fill = white, text = black, fill opacity=0, text opacity = 1}}
                                    \Edge[label=$e\protect\vphantom{'}$](a1)(b1)
    \Edge[label=$e'$](a4)(b4)
    \Edge(a2)(b2)
    \Edge(a3)(b3)
    \Edge(a0)(b0)
            \tikzset{EdgeStyle/.style= {thin,
                                double          = red!50,
                                double distance = 1pt}}
     \Edge(b0)(b1)
     \Edge(b1)(b2)
     \Edge(b2)(b3)
     \Edge(b3)(b4)
     \Edge(b0)(b4)
     \Edge(a0)(a1)
     \Edge(a1)(a2)
     \Edge(a2)(a3)
     \Edge(a3)(a4)
     \Edge(a0)(a4)
\end{tikzpicture}
\vspace{-1em}
\caption{The graph $G_0$ with two edges $e,e' \in \cC: d_{G_0}(e, e') = 2$, $d = 10$.}\label{conn:graph:ex2}\vspace{-8pt}
\end{figure}
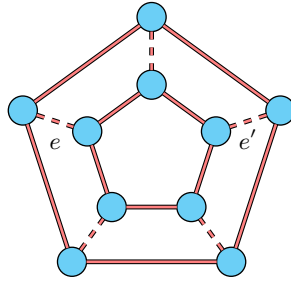
\end{example}

\begin{example}[$\fm+1$ vs $\fm$ Connected Components, $\fm \geq \sqrt{d} $] \label{conn:comp} 
Let $\fm \geq \sqrt{d}$ be an integer.
In this example we are interested in testing whether the graph contains $\fm + 1$ connected components vs $\fm$ connected components. The reason to assume $\fm \geq \sqrt{d}$ is to make sure there are sufficiently many edges for constructing a single edge divider in order to obtain sharp bounds. The case when $\fm < \sqrt{d}$ is treated in Example \ref{conn:comp:new} via a different divider construction (In fact, the case $\fm < \sqrt{d}$  requires deleting edges from the alternative rather than adding edges to the null base. See Section \ref{edge:removal:NAS:extension} of the supplement for more details). Formally we have the sub-decomposition $\cG_0 = \{G \in \cG ~|~ G \mbox{ has } \geq \fm + 1 \mbox{ connected components} \}$ vs $\cG_1 = \{G \in \cG ~|~ G \mbox{ has} \leq \fm \mbox{ connected components}\}$. Construct the null base graph $G_0 = (\overline V, E_0)$, where $E_0 := \{(j, j+1)_{j = 1}^{d - \fm - 1}\},$ and we let $\cC := \{(j, j+1)_{j = d - \fm}^{d - 1}\} \mbox{  (see Fig \ref{path:graph:conn:comp:ex2})}.$
Adding an edge $e \in \cC$ to $G_0$ converts the base graph $G_0$ into a graph with $\fm$ connected components and therefore $\cC$ is a single edge divider with a null base $G_0$. Additionally, the maximum degree of $G_0$ is $2$ by construction. Note that the distance between any two edges in $\cC$ is $0$ if and only if they share a common vertex, and $\infty$ in all other cases. This implies that we can construct a packing set by taking every other edge in the set $\cC$. We conclude that $M(\cC, d_{G_0}, \log |\cC|) \asymp \log (|\cC|/2) \asymp \log d$. Hence, by Theorem \ref{suff:cond:lower:bound} the minimax risk goes to $1$ when $\theta < \kappa \sqrt{\log d/n} \wedge \frac{(1 - C^{-1})\wedge e^{-\frac{1}{2}}}{4\sqrt{2}}$. 
\vspace{-0em}
\begin{figure}
\centering
\begin{tikzpicture}[scale=.7]
\SetVertexNormal[Shape      = circle,
                  FillColor  = cyan!50,
                  MinSize = 11pt,
                  InnerSep=0pt,
                  LineWidth = .5pt]
   \SetVertexNoLabel
   \tikzset{LabelStyle/.style = {below, fill = white, text = black, fill opacity=0, text opacity = 1}}
   \tikzset{EdgeStyle/.style= {thin,
                                double          = red!50,
                                double distance = 1pt}}
    \begin{scope}\grPath[prefix=a,RA=2]{4}\end{scope}
                                        \Edge(a2)(a3)
       \tikzset{EdgeStyle/.style= {dashed,thin,
                                double          = red!50,
                                double distance = 1pt}}
     \begin{scope}[shift={(6,0)}]\grPath[prefix=b,RA=2]{4}\end{scope}
     \Edge[label=$e\protect\vphantom{'}$](a3)(b1)
          \Edge[label=$e'$](b2)(b3)
\end{tikzpicture}
\vspace{-1em}
\caption{Null base graph $G_0$ with $d - \fm - 1$ edges, divider $\cC$ (dashed), $d_{G_0}(e, e') = \infty$, $d$ = 7.} \label{path:graph:conn:comp:ex2}\vspace{-8pt}
\end{figure}
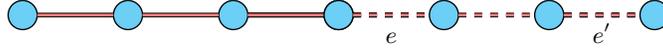
\end{example} 
\begin{example}[Cycle Testing] \label{cycle:example} 
Consider testing whether the graph is a forest vs the graph contains a cycle. Let $\cG_0 = \{G \in \cG ~|~ G \mbox{  is cycle-free}\}$ and $\cG_1 = \{G \in \cG ~|~ G \mbox{  contains a cycle}\}$. Define the null base graph $G_0 = (\overline V, E_0)$, where $E_0 := \{(j,j+1)_{j = 1}^{d-1}\}$. Let the edge set $\cC := \{(j, j + 2)_{j = 1}^d\},$ where the addition is taken modulo $d$ (refer to Fig \ref{ham:cycle:graph} for a visualization). By construction we have $G_0 \in \cG_0$ and $|\cC| = d$. Adding any edge from $\cC$ to $G_0$ results in a graph with a cycle, and hence the edge set $\cC$ is a single edge divider with a null base $G_0$. The maximum degree of $G_0$ equals $2$, and is thus bounded. Moreover, there exists a $(\log |\cC|)$-packing set of $\cC$ of cardinality at least $\frac{|\cC| - 2}{\lceil \log |\cC| \rceil + 2}$ which can be produced by collecting the edges $(j, j + 2)$ for $j = k(\log|\cC| + 2) + 1$ for $k = 0, 1,\ldots$ and $j \leq d - 2$. The last observation implies that $M(\cC, d_{G_0}, \log |\cC|) \asymp \log \frac{|\cC| - 2}{\lceil \log |\cC| \rceil + 2} \asymp \log d$. Hence by Theorem \ref{suff:cond:lower:bound} we conclude that the minimax risk goes to $1$ when $\theta < \kappa \sqrt{\log d/n} \wedge \frac{(1 - C^{-1})\wedge e^{-\frac{1}{2}}}{4\sqrt{2}}$.
\vspace{-1em}

\begin{figure}[H] 
\centering
\begin{tikzpicture}[scale=.7]
\SetVertexNormal[Shape      = circle,
                  FillColor  = cyan!50,
                  MinSize = 11pt,
                  InnerSep=0pt,
                  LineWidth = .5pt]
   \SetVertexNoLabel
   \tikzset{EdgeStyle/.style= {thin,
                                double          = red!50,
                                double distance = 1pt}}
                                                                \begin{scope}[rotate=90]\grEmptyCycle[prefix=a,RA=2]{7}{1}\end{scope}
                                                                    \Edge(a0)(a1)
                                                                    \Edge(a1)(a2)
                                                                    \Edge(a2)(a3)
                                                                    \Edge(a4)(a5)
                                                                    \Edge(a5)(a6)
                                                                    \Edge(a6)(a0)
                                                                    \tikzset{LabelStyle/.style = {right, fill = white, text = black, fill opacity=0, text opacity = 1}}
								\tikzset{EdgeStyle/.append style = {dashed, thin, bend right}}
                                                                    \Edge[label=$e\protect\vphantom{'}$](a3)(a1)
                                                                    \Edge(a1)(a6)
                                                                    								\tikzset{LabelStyle/.style = {left, fill = white, text = black, fill opacity=0, text opacity = 1}}
                                                                    \Edge[label=$e'$](a6)(a4)
                                                                    \Edge(a2)(a0)      
								 \Edge(a0)(a5)         
								 \Edge(a4)(a2)    
								 \Edge(a5)(a3)      
\end{tikzpicture}
\vspace{-1em}
\caption{The graph $G_0$ with two (dashed) edges $e,e' \in \cC$ such that $d_{G_0}(e, e') = 2$, $d = 7$.}\label{ham:cycle:graph}\vspace{-8pt}
\end{figure}
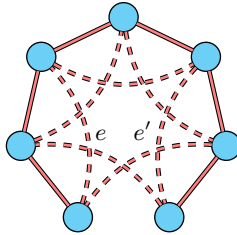

\end{example}

\begin{example}[Tree vs Connected Graph with Cycles]
The construction in Example \ref{cycle:example}  also shows that we have the same signal strength limitation to test for cycles, even if we restrict to the subclass of connected graphs, i.e., the class of graphs under the null hypothesis is the class of all trees $\cG_0 = \{G \in \cG~|~ G \mbox{ is a tree}\}$, and the alternative is the class of all connected graphs which contain a cycle --- $\cG_1 \in \{G \in \cG~|~ G \mbox{ is connected but is not a tree}\}$.
\end{example}

\begin{example}[Triangle-Free Graph] Consider testing whether the graph contains a triangle (i.e., $3$-clique). More formally let the decomposition of $\cG$ be $\cG_0 = \{G \in \cG~|~ G \mbox{ is triangle-free}\}$ and $\cG_1 = \{G \in \cG~|~ \exists \mbox{ $3$-clique in } G\}$. It is clear that in this case we can reuse the set $\cC$ and its null base $\tilde G_0 = (\overline V, E_0 \cup \{(1,d)\})$, where $E_0$ and $\cC$ are taken as in Example \ref{cycle:example}.
\end{example}

\subsection{General remarks on choosing a null base $G_0$} The following result sheds some light on reasonable choices of $G_0$.

\begin{proposition}\label{smart:bound} Let the graph $G_0 = (\overline V, E_0) \in \cG_0$ have bounded maximum degree. Suppose there exist constants $0 < c, \gamma \le1$ so that for each vertex $v \in \overline V$, one can find a set of vertices $W_v$ satisfying $|W_v| \geq c d^\gamma$ and for all $w \in W_v$, we have $(\overline V, E_0 \cup\{(v,w)\}) \in \cG_1$. Then there exists a divider with null base $G_0$ satisfying $M(\cC, d_{G_0}, \log |\cC|) \asymp \log d$.
\end{proposition}

Of note, for any edge set $\cC$ one has $M(\cC, d_{G_0}, \log |\cC|) \leq \log {d \choose 2} \asymp \log d$, which implies that graphs $G_0 \in \cG_0$ as in Proposition \ref{smart:bound} give scalar optimal bounds. The existence of such graphs is dependent on the sub-decomposition $(\cG_0, \cG_1)$. Notably, all examples in Section \ref{some:applications} fall under the framework of Proposition \ref{smart:bound}. Its proof can be found in the supplement.

\begin{remark}\label{supp:recovery:remark} When $\bTheta \in \cM(s)$, the results in Theorem \ref{suff:cond:lower:bound} (and Theorem \ref{suff:cond:lower:bound:deletion:NAS}) suggest that a signal strength of order $\sqrt{\log d/n}$ is necessary for controlling the minimax risk (\ref{risk:def}). In fact, Theorem 7 of \cite{Cai2011Constrained} shows that under such signal strength condition, support recovery of $\bTheta$ is indeed achievable, which  further implies that controlling the minimax risk (\ref{risk:def}) is possible.  A naive procedure for matching the lower bound is to first perfectly recover the graph structure. Then construct a test based on examining whether the graph has the desired combinatorial structure. Though such an approach is theoretically feasible, it is not practical. First, such an approach is overly conservative and does not allow us to tightly control the type I error at a desired level. Second, such an approach crucially depends on having a suitable thresholding parameter to estimate the graph, which is in general not realistic.
\end{remark}

In the next section, we present a family of testing procedures, which do not require perfect support recovery of the full graph. Compared to the naive approach described in Remark \ref{supp:recovery:remark}, our tests explicitly exploit the combinatorial structure of the targeted hypotheses and can control the type I error at any desired level.

\section{Alternative Witness Test}\label{sparse:graph:models:inf}

We start with a high-level outline of a new combinatorial inference approach for graphical models. For clarity we mainly present using the case of Gaussian graphical models, and comment on extensions to other graphical models in Section \ref{extensions:GM:OTHERS} of the supplement.  

Importantly the algorithms we develop in this section apply to alternative graph classes which are stable under edge addition. Formally, for any $G = (\overline V, E) \in \cG_1$ and any edge $e$, we require that the graph $(\overline V, E \cup \{e\}) \in \cG_1$. Note that due to edge addition stability under the alternative, the full graph $\overline G = (\overline V,\overline E)$ belongs to $\cG_1$. For a graph $G \in \cG_1$, define the following class of edge sets
\begin{align*}
\cW_1(G) = \{E ~|~ \forall E, ~ E \subseteq E(G) \Rightarrow (\overline V, E) \in \cG_1\}.
\end{align*}
The set $\cW_1(G)$ collects all edge sets forming graphs in $\cG_1$, which can be obtained by the graph $G$ via iteratively pruning one edge at a time. We use the shorthand notation
$$\overline \cW_1 = \cW_1(\overline G).$$
Consider the following parameter sets:
\begin{align}
\cS_0(s) &:= \Big \{\bTheta\in  \cM(s) ~|~ G(\bTheta) \in \cG_{0}\Big \}, \mbox{ and }\label{generic:S0:def:new} \\
\cS_1(\theta, s) &:= \Big \{\bTheta\in  \cM(s) ~|~ G(\bTheta) \in \cG_{1}, \max_{E \in \overline \cW_1}\min_{e \in E} |\Theta_e| \geq \theta \Big \}, \label{generic:S1:def:new}
\end{align}

The parameter set (\ref{generic:S0:def:new}) does not impose any assumption on the minimum signal strength, thus is  broader than the one defined in (\ref{generic:S0:def}). In Definition (\ref{generic:S1:def:new}), the signal strength is not imposed on all edges of the alternative graphs. In fact, we only need to impose the signal strength assumption on a subset of edges which can be obtained by pruning the complete graph. Such a condition is much weaker than the usual condition needed for perfect graph recovery. We note that for sub-decompositions $(\cG_0, \cG_1)$ satisfying $\cW_1(G) \subseteq \overline \cW_1$ for all $G \in \cG_1$, the parameter set (\ref{generic:S1:def:new}) is strictly larger than the parameter set (\ref{generic:S1:def}). 

Given $n$ independent samples $\bX_i \sim N(0, (\bTheta^*)^{-1})$ and a sub-decomposition $(\cG_0,\cG_1)$ we formulate a procedure for testing $\Hb_0: \bTheta^* \in \cS_0(s)$ vs $\Hb_1: \bTheta^* \in \cS_1(\theta, s)$. Let $\hat \bSigma := n^{-1}\sum_{i = 1}^n \bX_i^{\otimes 2}$ be the empirical covariance matrix. Let  $\hat \bTheta$ be any estimator of the precision matrix $\bTheta^*$ satisfying for some fixed constant $K > 0$:
\begin{align}
&~~~ ~~~ ~~~ ~~~ ~~~ ~~~ ~~~  \|\hat \bTheta - \bTheta^*\|_{\max} \leq K \sqrt{\log d/n}, \label{maxnorm:prec:ass} \\
& \|\hat \bTheta - \bTheta^*\|_{1}  \leq K s \sqrt{\log d/n}, ~~~ \|\hat  \bSigma\hat \bTheta - \Ib_d \|_{\max}  \leq K\sqrt{\log d/n}, \label{other:norms:prec:ass}
\end{align}
with probability at least $1 - d^{-1}$ uniformly over the parameter space $\cM(s)$ (recall definition (\ref{MS:def})). An example of an estimator of $\bTheta^*$ with this properties is the CLIME procedure introduced by \cite{Cai2011Constrained} (see also (\ref{generalCLIMEopt})). An overview of the \textit{alternative witness test} is sketched below:
\begin{itemize}
\item[i.] In the first step the alternative witness test identifies a \textit{minimal structure witnessing the alternative};
\item[ii.] In the second step the alternative witness test attempts to \textit{certify} that the minimal structure identified by the first step is indeed present in the graph.
\end{itemize}

Split the data $\cD = \{\bX_i\}_{i = 1}^n$ in two approximately equal-sized sets $\cD_1 = \{\bX_i\}_{i = 1}^{\lfloor n/2\rfloor}, \cD_2 = \{\bX_i\}_{i = \lfloor n/2\rfloor + 1}^{n}$ and obtain estimates $\hat \bTheta^{(1)}, \hat \bTheta^{(2)}$ on $\cD_1$ and $\cD_2$ correspondingly. For the first step, we exploit $\hat \bTheta^{(1)}$ to solve the following max-min combinatorial optimization problem
\begin{align}\label{min:struct:cert}
\widehat E :=  \argmax_{E \in \overline \cW_1}\min_{e \in E} |\hat \Theta^{(1)}_e|,
\end{align}
where edge sets with the smallest cardinality are prefered, and further ties are broken arbitrarily. Program (\ref{min:struct:cert}) aims at identifying the smallest edge set in $\overline \cW_1$ whose minimal signal is as large as possible. Given a consistent estimator $\hat \bTheta^{(1)}$ and sufficiently strong signal strength, the solution of (\ref{min:struct:cert}) identifies a minimal substructure of $G(\bTheta^*)$ belonging to $\cG_1$. This strategy is motivated by the definition of the alternative parameter set $\cS_1(\theta,s)$ (\ref{generic:S1:def:new}). We remark that solving  program (\ref{min:struct:cert}) could be computationally challenging for some combinatorial tests. However, for all examples considered in this paper, simple and efficient polynomial time algorithms are available. We refer to the graph $(\overline V, \widehat E)$ as the \textit{minimal structure witnessing the alternative}. Although the minimal structure witness is defined in full generality, for the ease of presentation we justify its validity on a case by case basis.

In the second step, the alternative witness test attempts to certify the witness structure using the estimate $\hat \bTheta^{(2)}$. Formally, we aim to test the hypothesis
\begin{align}\label{NAWT:cert:test}
\Hb_0: \exists ~ e \in \widehat E: \Theta^*_e = 0 \mbox{ vs } \Hb_1: \forall ~ e \in \widehat E: \Theta^*_e \neq 0.
\end{align}
A rejection of the null hypothesis in (\ref{NAWT:cert:test}) certifies the presence of the alternative witness structure. If the test fails to reject, the alternative witness test cannot reject the null structure hypothesis. In Section \ref{min:struct:cert:sec} we give a detailed description on how the second step of the test works.

\subsection{Minimal Structure Certification}\label{min:struct:cert:sec}

In this section we detail an algorithm for testing (\ref{NAWT:cert:test}). In fact, we present a general test for the following multiple testing problem
$$\Hb_0:  \exists ~e \in E \mbox{ s.t. } \Theta^*_{e} = 0 \mbox{   vs   }\Hb_1: \forall ~ e \in E, ~ \Theta^*_{e} \neq 0,$$
using the data $\cX = \{\bX_i\}_{i = 1}^n$, $\bX_i \sim N(0,(\bTheta^*)^{-1})$, where $E$ is a pre-given edge set. Following \cite{neykov2015aunified}, for any $j,k \in [d]$ we define the bias corrected estimate 
\begin{equation}
 \tilde \Theta_{jk} := \hat \Theta_{jk} - \frac{\hat \bTheta^T_{*j} (\hat \bSigma \hat \bTheta_{*k} - \eb_k)}{\hat \bTheta^T_{*j} \hat \bSigma_{*j}},
\end{equation}
where $\eb_k$ is a canonical unit vector with $1$ at its $k$\textsuperscript{th} entry. Under mild regularity conditions, we  show that if $\hat \bTheta$ satisfies (\ref{maxnorm:prec:ass}) and (\ref{other:norms:prec:ass}), $\tilde \Theta_{jk}$ admits the following Bahadur representation:
\begin{align}\label{bahadur:repr}
\sqrt{n}\tilde \Theta_{jk} = n^{-1/2} \sum_{i = 1}^n(\bTheta^{*T}_{*j}\bX_i^{\otimes 2}  \bTheta^*_{*k} - \Theta^*_{jk}) + o_p(1).
\end{align}
This motivates a mutiplier bootstrap scheme for approximating the distribution of the statistic  $\sqrt{n}\tilde \Theta_{jk}$ under the null hypothesis:
$$\hat W_{jk} = n^{-1/2} \sum_{i = 1}^n(\hat \bTheta^{T}_{*j}\bX_i^{\otimes 2}  \hat \bTheta_{*k} - \hat \Theta_{jk})\zeta_i,$$
where $\zeta_i \sim N(0,1), i \in [n]$ are independent and identically distributed. To approximate the null distribution of the statistic $\max_{(j,k) \in S} \sqrt{n}|\tilde \Theta_{jk}|$ over a subset $S \subseteq E$, let $c_{1-\alpha, S}$ denote the $(1-\alpha)$-quantile of the statistic $\max_{(j,k) \in S} |\hat W_{jk}|$ (conditioning on the dataset $\cX$). Formally, we let
\begin{align}\label{quantile:def}
c_{1-\alpha, S} := \inf_{t \in \RR} \Big\{t ~\Big |~ \PP\Big(\max_{(j,k) \in S}  |\hat W_{jk}| \leq t  ~\Big|~ \cX\Big) \geq 1- \alpha \Big\}.
\end{align}
As defined $c_{1-\alpha,S}$ is a population quantity (conditioning on $\cX$). In practice an arbitrarily accurate estimate $\hat c_{1-\alpha, S}$ of $c_{1-\alpha, S}$ can be obtained via Monte Carlo simulations. Below we describe a multiple edge testing procedure returning a subset $\hat{E^{\rm nc}} \subseteq E$ of rejected edges (hypotheses). Our procedure is based on the step-down construction of \cite{Chernozhukov2013Gaussian}, which is inspired by the multiple  testing method of \cite{romano2005exact}. 
\begin{algorithm}[t]
\normalsize
\caption{Multiple Edge Testing}\label{step:down}
\begin{algorithmic}
\STATE Initialize $\widehat{E^{\rm n}} \leftarrow E$;
\REPEAT
\STATE Reject $R \leftarrow \{e \in \widehat{E^{\rm n}}: \sqrt{n}|\tilde \Theta_{e}| \geq c_{1 - \alpha, \widehat{E^{\rm n}}}\}$
\STATE{Update $\widehat{E^{\rm n}} \leftarrow \widehat{E^{\rm n}} \setminus R$}
\UNTIL{$R = \varnothing$ or $\widehat{E^{\rm n}} = \varnothing$}
\RETURN{$\hat{E^{\rm nc}} \leftarrow E \setminus \widehat{E^{\rm n}}$}
\end{algorithmic}
\end{algorithm}
Decompose the edge set $E = E^{\rm n} \cup {E^{\rm nc}}$, where $E^{\rm n} \cap E^{\rm nc} = \varnothing$, $E^{\rm n}$ is the subset of true null edges and $E^{\rm nc}$ is the set of non-null edges. Define the parameter set
\begin{align}\label{def::kappa}
\cM_{E^{\rm n},E^{\rm nc}}(s,\kappa) = \bigg\{\bTheta \in \cM(s) \Big| \max_{e \in E^{\rm n}}|\Theta_{e}| = 0, \min_{e \in E^{\rm nc}}|\Theta_{e}| \geq \kappa \sqrt{\frac{\log d}{n}}\bigg\}.
\end{align} 
For a fixed edge set $E$, we say that an edge set $\widehat{E^{\rm r}}$ has strong control of the family-wise error rate if
\begin{align}
\limsup_{n \rightarrow \infty} \sup_{E^{\rm n} \subset E} \sup_{\bTheta^* \in \cM_{E^{\rm n},E^{\rm nc}}(s,\kappa)} \PP(E^{\rm n} \cap \widehat{E^{\rm r}} \neq \varnothing) \leq \alpha, \label{strong:FWER:control}
\end{align}
for some pre-specified size $\alpha > 0$. Our next result shows that Algorithm \ref{step:down} returns an edge set $\widehat{E^{\rm nc}}$ with strong control of the family-wise error rate.

\begin{proposition}[Strong Family-Wise Error Rate Test] \label{multiple:edge:testing:validity} 
Let $\bTheta^* \in \cM_{E^{\rm n},E^{\rm nc}}(s,\kappa)$ and furthermore:
\begin{align}\label{bootstrap:rates}
s\log(nd)\sqrt{\log d \log(nd)}/\sqrt{n} = o(1), ~~~~ (\log(dn))^6/n = o(1).
\end{align}
Then for any fixed edge set $E$, the output $\widehat{E^{\rm nc}}$ of Algorithm \ref{step:down}  satisfies (\ref{strong:FWER:control}). In addition, if the constant $\kappa$ in \eqref{def::kappa} satisfies $\kappa \geq C' C^4$ for a fixed absolute constant $C' > 0$, we have $
\liminf_{n \rightarrow \infty} \inf_{E^{\rm n} \subseteq E} \inf_{\bTheta^* \in \cM_{E^{\rm n},E^{\rm nc}}(s,\kappa)} \PP(E^{\rm nc} =\widehat{E^{\rm nc}}) = 1.$
\end{proposition}
The first condition in (\ref{bootstrap:rates}) ensures the validity of the Bahadur representation in (\ref{bahadur:repr}).  The second condition in (\ref{bootstrap:rates}) is to guarantee validity of the high-dimensional bootstrap, and a similar condition is required by \cite{Chernozhukov2013Gaussian}. While the first condition of (\ref{bootstrap:rates}) is not necessarily sharp, it is nearly optimal by ignoring logarithmic terms of dimension and sample size comparing to the minimax rate established in \cite{ren2015asymptotic}.

Of note, when $\kappa$ is sufficiently large, Algorithm \ref{step:down} achieves exact control of the family-wise error rate, i.e., we have equality in (\ref{strong:FWER:control}). This happens since all edges in $E^{\rm nc}$ will be rejected with overwhelming probability, while the bootstrap comparison is asymptotically exact for the remaining edges $E^{\rm n}$. As a consequence of this result, if the null hypothesis set $\cS_0(s)$ considered in (\ref{generic:S0:def:new}) exhibits signal strength as in definition (\ref{generic:S0:def}), the alternative witness tests are exact. 

\begin{definition} For an edge set $E$ let $\hat{E^{\rm nc}}$ be the output of Algorithm \ref{step:down} on data $\cX$ with level $1-\alpha$. Define the following test function:
$$
\psi^B_{\alpha, E}(\cX) := \mathds{1}(E = \widehat{E^{\rm nc}}).
$$
$\psi^B_{\alpha, E}(\cX)$ tests whether the set $E$ is comprised only of non-null edges.
\end{definition}
\subsection{Examples} \label{struct:test:section}

In this section we describe practical algorithms based on the alternative witness test for testing problems outlined in Section \ref{some:applications}. Our tests can distinguish the null from the alternative hypotheses when the minimum signal strength is sufficiently large. As we shall see, the magnitude of the required signal strength  is precisely of  order $\sqrt{\log d/n}$ and therefore in view of Section \ref{sec:generic:lower} these tests are minimax optimal. Recall that we observe $n$ i.i.d. samples $\{\bX_{i}\}_{i = 1}^n$ from $\bX_{i} \sim N_d(0, (\bTheta^*)^{-1})$. We split the data into $\cD_1 = \{\bX_i\}_{i = 1}^{\lfloor n/2\rfloor}, \cD_2 = \{\bX_i\}_{i = \lfloor n/2\rfloor + 1}^{n}$ and obtain estimates $\hat \bTheta^{(1)}, \hat \bTheta^{(2)}$ on $\cD_1$ and $\cD_2$ correspondingly, for which (\ref{maxnorm:prec:ass}) and (\ref{other:norms:prec:ass}) hold. For space considerations, we present only the connectivity test in full details, and we elaborate on the minimal structures for the remaining tests. Full details can be found in Section \ref{detailed:upperbound:algos} of the supplement.

\subsubsection{Connectivity Testing}\label{conn:testing:sec}

This example proposes a new procedure for honestly testing whether $G(\bTheta^*)$ is a connected graph. Accordingly, the sub-decomposition is $\cG_0 := \{G \in \cG~|~ G \mbox{ disconnected} \}$ and $\cG_1 := \{G \in \cG~|~ G \mbox{ connected}\}$. The pair $(\cG_0, \cG_1)$ determines the parameter sets definitions (\ref{generic:S0:def:new}) and (\ref{generic:S1:def:new}). 

Finding the minimal structure witness (\ref{min:struct:cert}) reduces to finding a maximum spanning tree (MST) $\hat T$ on the full graph with edge weights $|\hat \Theta^{(1)}_e|$. The complexity of finding a MST is $O(d^2\log d)$, where $d$ is the number of vertices. We summarize the procedure below: 
\begin{algorithm}[H]
\normalsize
\caption{Connectivity Test}\label{al:conn}
\begin{algorithmic}
\STATE \textbf{Input:} $\cD=\{\bX_i\}_{i=1}^n$, level $0 < \alpha < 1$.
\STATE Split the data  $\cD_1 = \{\bX_i\}_{i = 1}^{\lfloor n/2\rfloor }, \cD_2 = \{\bX_i\}_{i = \lfloor n/2\rfloor + 1}^{n}$
\STATE Using $\cD_1$ obtain estimate $\hat \bTheta^{(1)}$ satisfying (\ref{maxnorm:prec:ass})
\STATE Find MST $\hat T$ on the full graph with weights $|\hat \Theta_e^{(1)}|$.
\STATE \textbf{Output:} $\psi^B_{\alpha, \widehat T}(\cD_2)$.
\end{algorithmic}
\end{algorithm}
\vspace{-10pt}
The results on connectivity testing are summarized in the following
\begin{corollary} \label{conn:test:prop} Let $\theta = (2K + C' C^4)\sqrt{\log d/\lfloor n/2\rfloor}$ for an absolute constant $C' > 0$ and  assume that (\ref{bootstrap:rates}) holds. Then for any fixed level $\alpha$, the test $\psi^B_{\alpha, \widehat T}(\cD_2)$ from the output of Algorithm \ref{al:conn} satisfies:
\begin{align*}
\limsup_{n \rightarrow \infty} \sup_{\bTheta^* \in \cS_0(s)} \PP(\mbox{reject } \Hb_0) \leq \alpha, ~~~~ \liminf_{n \rightarrow \infty} \inf_{\bTheta^* \in \cS_{1}(\theta, s)} \PP(\mbox{reject } \Hb_0) = 1.
\end{align*}
\end{corollary}
\vspace{-10pt}
\subsubsection{Connected Components Testing}

Connected component testing is more general compared to connectivity testing. For $\fm \in [d - 1]$ let $\Hb_0: \mbox{\# connected components} \geq \fm + 1$ vs $\Hb_1: \mbox{\# connected components} \leq \fm$. Testing connectivity is a special case when $\fm = 1$. 

For a fixed $\fm \in [d-1]$ define the sub-decomposition $\cG_0 = \cF_{\fm + 1}$ and $\cG_1 = \cup_{j \leq \fm}\cF_{j}$ 
 where 
$$\cF_{j} := \{G \in \cG~|~ G \mbox{ has exactly $j$ connected components}\} \mbox{ for all } j \in [d].$$
The sub-decomposition $(\cG_0, \cG_1)$ also defines the parameter sets (\ref{generic:S0:def:new}) and (\ref{generic:S1:def:new}). Recall that a sub-graph $F$ of $G$, i.e., $E(F) \subset E(G)$ and $V(F) = V(G) = \overline V$, is called a \textit{spanning forest} of $G$, if $F$ contains no cycles, adding any edge $e \in E(G) \setminus E(F)$ to $E(F)$ creates a cycle, and $|E(F)|$ is maximal. This definition extends naturally to graphs with positive weights on their edges. It is easy to check that the minimal structure witness (\ref{min:struct:cert}) is the maximal spanning forrest with $\fm$ connected components, and can be found efficiently via a greedy algorithm. For more details see Section \ref{detailed:upperbound:algos} of the supplement. 

\subsubsection{Cycle Testing}\label{ham:cycle:det:sec}

In this example we sketch how to test whether the graph is a forest. Recall the sub-decomposition $\cG_0 = \{G \in \cG~|~ G \mbox{ is a forest}\}$ and $\cG_1 = \{G \in \cG~|~ G \mbox{ contains a cycle}\}$. The pair $(\cG_0, \cG_1)$ also defines the parameter sets (\ref{generic:S0:def:new}) and (\ref{generic:S1:def:new}). 
The minimal structure witness (\ref{min:struct:cert}) is a cycle, and can be found via greedily adding edges until a cycle is formed. For more details see Section \ref{detailed:upperbound:algos} of the supplement.

\subsubsection{Triangle-Free Graph Testing}\label{triangle:free:testing:sec}

In this example we discuss an algorithm for testing whether the graph is triangle-free. The corresponding sub-decomposition is $\cG_0 = \{G \in \cG~|~ G \mbox{ is triangle-free}\}$ and $\cG_1 = \{G \in \cG~|~ \exists \mbox{ $3$-clique subgraph of } G\}$. 
The pair $(\cG_0, \cG_1)$ also defines the parameter sets (\ref{generic:S0:def:new}) and (\ref{generic:S1:def:new}). The minimal structure witness is a triangle, and can be greedily identified.

\section{Multi-Edge Dividers}\label{general:lower:bound}

The results of Section \ref{sec:generic:lower} (and Section \ref{edge:removal:NAS:extension} of the supplement) have two major limitations. First, the null base $G_0$ is assumed to be of bounded degree. Second, our results cover only tests for which there exists a singe-edge divider. In this section we relax both of these conditions. The following motivating example illustrates a relevant testing problem which does not fall into the framework of Section \ref{sec:generic:lower}.\vspace{.2cm}

\noindent{\bf Maximum Degree Testing.} Consider testing whether the maximum degree of the graph $d_{\max}$ satisfies $d_{\max} \leq s_0$ vs $d_{\max} \geq s_1$, where $s_0 < s_1 \leq s$ are integers which are allowed to scale with $n$. In this case it is impossible to simultaneously construct a null base graph $G_0$ of bounded degree and a single-edge divider $\cC$. \vspace{.2cm}

To handle multiple edge dividers, we first extend Definitions \ref{single:edge:null:alt:sep} and \ref{closed:walks:def} to allow for the above examples.

\begin{definition}[Null-Alternative Divider]\label{multi:edge:null:alt:sep} Let $G_0 = (\overline V, E_0) \in \cG_0$ be a fixed graph under the null with adjacency matrix $\Ab_0$. We call a collection of edge sets $\bC$ a (multi-edge) divider with null base $G_0$, if for all edge sets $S \in \bC$ we have $S \cap E_0 = \varnothing$ and $(\overline V, E_0 \cup S) \in \cG_1$. 
For any edge set $S \in \bC$, we denote the adjacency matrix of the graph $(\overline V, S)$ with $\Ab_S$. 
\end{definition}

\begin{definition}[Edge Set Geodesic Predistance]\label{S:leftright:Sprime:def} For two edge sets $S$ and $S'$ and a given graph $G$ let $d_{G}(S, S') = \min_{e \in S, e' \in S'} d_G(e,e').$ 
\end{definition}

We provide two generic strategies for obtaining combinatorial inference lower bounds on the signal strength. The first strategy, described in Section \ref{bounded:edges}, assumes that all $S \in \bC$ satisfy $|S| \leq U$ for some fixed constant $U$. The second strategy, presented in Section \ref{scaling:edges}, does not require bounded cardinality of the edge sets $S$, but requires that the null bases and dividers have some special combinatorial properties.

\subsection{Bounded Edge Sets} \label{bounded:edges}

Below we consider an extension of Theorem \ref{suff:cond:lower:bound} for multi-edge dividers, where the number of edges in each set $S \in\bC$ satisfy $|S| \leq U$ for some fixed integer $U \in \NN$. In contrast to Section \ref{sec:generic:lower}, here the graph $G_0$ is allowed to have unbounded degree.

\begin{theorem}\label{suff:cond:lower:bound:many} 
Let $G_0 \in \cG_0$ be a graph under the null, and let $\bC$ be a multi-edge divider with null base $G_0$. Suppose that for some sufficiently small absolute constant $\kappa > 0$:
\begin{align}\label{scaling:conditions:sec4}
\theta \leq \kappa \sqrt{\frac{M(\bC, d_{G_0}, \log |\bC|)}{n U}} \wedge \frac{\kappa}{U(\|\Ab_0\|_2 + 2U)} \wedge \frac{1 - C^{-1}}{4(\|\Ab_0\|_1 + 2U)}.
\end{align}
If $M(\bC, d_{G_0}, \log |\bC|) \rightarrow \infty$ we have ${\displaystyle \liminf_{n \rightarrow \infty}} \gamma(\cS_0(\theta, s), \cS_1(\theta, s)) = 1$.
\end{theorem}

Theorem \ref{suff:cond:lower:bound:many} is an extension of Theorem \ref{suff:cond:lower:bound}. Specifically, Theorem \ref{suff:cond:lower:bound} corresponds the setting  where $U = 1$, and $\|\Ab_0\|_2 \leq \|\Ab_0\|_1 \leq D$ (recall that $D$ is an upper bound of the graph degree). Even though by assumption $U = \max_{S \in \bC} |S|$ is bounded, we  explicitly keep the dependency on $U$  in (\ref{scaling:conditions:sec4}) to reflect how the bound changes if $U$ is allowed to scale. The first term on the right hand side of (\ref{scaling:conditions:sec4}) is the structural packing entropy, while the remaining two terms ensure the parameter $\theta$ is small enough to construct a valid packing set (More details are provided in the proof). 

We illustrate the usefulness of Theorem \ref{suff:cond:lower:bound:many} by an example similar to the ones in Section \ref{some:applications}. Consider testing whether the maximum degree of the graph $G(\bTheta^*)$ is at most $s_0$ vs it is at least $s_1$, where $s_0 < s_1 \leq s$ can increase with $n$  but the null-alternative gap $s_1 - s_0$ remains bounded. Therefore we cannot apply Theorem \ref{suff:cond:lower:bound} but should use Theorem \ref{suff:cond:lower:bound:many} instead. Define the sub-decomposition $\cG_0 = \{G~|~ d_{\max}(G) \leq s_0\}$ and $\cG_1 = \{G~|~ d_{\max}(G) \geq s_1\}$ respectively.

\begin{example}[Maximum Degree Test with Bounded Null-Alternative Gap] \label{max:deg:1:edge} Let $\cS_0(\theta, s)$ and $\cS_1(\theta, s)$ be defined in (\ref{generic:S0:def}) and (\ref{generic:S1:def}). Assume that $s \log d/n = o(1), s \sqrt{\log d/n} = O(1)$ and $s = O(d^{\gamma})$ for some $\gamma < 1$. Then if $\kappa$ is small enough and $\theta < \kappa \sqrt{\log d/n}$, we have $\liminf_{n \rightarrow \infty} \gamma(\cS_0(\theta,s), \cS_1(\theta, s)) = 1$.
\end{example}

The proof of Example \ref{max:deg:1:edge} is deferred to the supplement.

\subsection{Scaling Edge Sets}\label{scaling:edges}

Theorem \ref{suff:cond:lower:bound:many} requires the cardinalities of the edge sets in the divider $\bC$ to be bounded. 
In this section, we consider multi-edge dividers $\bC$ allowing the sizes of $S \in \bC$ to increase with $n$. For this case, the previous notion of packing entropy based on geodesic predistence is no longer effective. Instead, we introduce a new mechanism called buffer entropy to quantify the lower bound under scaling multi-edge dividers.

We first intuitively explain why the structural entropy in Theorem \ref{suff:cond:lower:bound:many} may not be sufficient for handling dividers with scaling edge sets. Recall that Theorem \ref{suff:cond:lower:bound:many} uses the structural
entropy $M(\bC, d_{G_0}, \log |\bC|)$ to characterize the lower bound. In turn, the structural entropy is calculated based on the edge set geodesic predistance $d_{G_0}$ in Definition \ref{S:leftright:Sprime:def}. One difference between fixed and scaling edge sets sizes is that, one can only pack a limited number of edge sets or large size which are sufficiently far apart (and hence do not overlap). A less wasteful strategy would be to allow for the edge sets to overlap. However, in general, different edge sets $S, S' \in \bC$ may have multiple overlapping vertices and the notion of geodesic predistance is no longer precise enough to reflect the closeness between $S$ and $S'$. 
 
Below we introduce a concept called vertex buffer, which helps to measure the closeness between edge sets $S$ and $S'$ more precisely than the geodesic predistance.

\begin{definition}[Vertex Buffer] \label{non:isolated:vertices} Let $G_0 = (\overline V,E_0)$ be a given graph and $S,S'$ be two edge sets. The vertex buffer of $S,S'$ under $G_0$ is defined as 
$$\cV_{S,S'} := \{V(E_0 \cup S) \cap V(S')\} \cup \{V(E_0 \cup S') \cap V(S)\}.\footnotemark$$
\end{definition}
\footnotetext{We suppress the dependence of $\cV_{S,S'}$ on $G_0$ to ease the notation.}
An important property of the set $\cV_{S,S'}$ is that all paths passing through at least one edge in both $S$ and $S'$ must contain at least one vertex in $\cV_{S,S'}$. In that sense, a large buffer size $|\cV_{S,S'}|$ indicates that the edge sets $S$ and $S'$ are close to each other. We visualize an example of a vertex buffer in Fig \ref{fig:buffer-vis}.

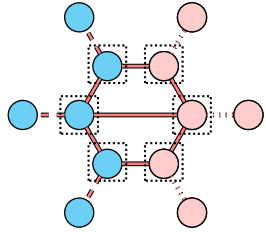
\begin{figure}[h]
\begin{center}
\begin{tikzpicture}[scale=.5]
\SetVertexNormal[Shape      = circle,
                  FillColor  = cyan!50,
                  MinSize = 11pt,
                  InnerSep=0pt,
                  LineWidth = .5pt]
   \SetVertexNoLabel
   \tikzset{EdgeStyle/.style= {thin,
                                double          = red!50,
                                double distance = 1pt}}
                                                                \begin{scope}[rotate=120]\grEmptyCycle[prefix=a,RA=1.5]{6}{1}\end{scope}
                                                                \Edge(a0)(a5)
                                                                \Edge(a1)(a4)
                                                                \Edge(a2)(a3)
                            								\tikzset{EdgeStyle/.append style = { thin}}

                                                                \Edge(a0)(a1)
                                                                \Edge(a1)(a2)
                                                                    \tikzset{LabelStyle/.style = {right, fill = white, text = black, fill opacity=0, text opacity = 1}}

								\tikzset{EdgeStyle/.append style = { thin}}
                                                                    								\tikzset{LabelStyle/.style = {left, fill = white, text = black, fill opacity=0, text opacity = 1}}                                                                                                                                    \Edge(a3)(a4)
\Edge(a4)(a5)
\begin{scope}[rotate=120]\grEmptyCycle[prefix=b,RA=3]{6}{1}\end{scope}
\AddVertexColor{red!20}{a3,a4,a5}
\AddVertexColor{red!20}{b3,b4,b5}
   \tikzset{EdgeStyle/.style= { thin,dashed,
                                double          = red!50,
                                double distance = 1pt,
                                }}
                                \Edge(a0)(b0)                                                                                                                                                                                 			       \Edge(a1)(b1)
			       \Edge(a2)(b2)
			          \tikzset{EdgeStyle/.style= { thin,dotted,
                                double          = red!50,
                                double distance = 1pt}}
                                \Edge(a3)(b3)                                                                                                                                                                                 			       \Edge(a4)(b4)
			       \Edge(a5)(b5)
			       \draw[thick, densely dotted] (-0.5-.75,1-.15) rectangle (.5-.75,2-.15);
			       \draw[thick, densely dotted] (-0.5+.75,1-.15) rectangle (.5+.75,2-.15);
			       \draw[thick, densely dotted] (-0.5-1.5,1-1.5) rectangle (.5-1.5,2-1.5);
			       \draw[thick, densely dotted] (-0.5+1.5,1-1.5) rectangle (.5+1.5,2-1.5);
			       \draw[thick, densely dotted] (-0.5-.75,1-2.75) rectangle (.5-.75,2-2.75);
			       \draw[thick, densely dotted] (-0.5+.75,1-2.75) rectangle (.5+.75,2-2.75);

\end{tikzpicture}
\vspace{-1em}
\caption{Visualization of the vertex buffer in $\cV_{S,S'}$. Here $S$, $S'$ are plotted with dashed and dotted edges respectively and $G_0$ is in solid edges. The vertices in the buffer are marked in the dashed squares.}\vspace{-18pt}
\label{fig:buffer-vis}
\end{center}
\end{figure}

In contrast to the bounded edge sets case, when the edge sets in $\bC$ are allowed to scale in size, it is not effective to build packing sets based on the predistance, since this strategy limits the number of edge sets we can build. One way to increase the cardinality of $\bC$ is to consider a larger number of potentially overlapping structures, and use the buffer size as a more precise closeness measure between these structures. Below we formalize the concept of buffer entropy  which quantifies this intuition.

\begin{definition}[Buffer Entropy]\label{buffer:entropy:def} Let $\bC$ be a multi-edge divider with a base graph $G_0$. The buffer entropy is defined as:
\begin{align}
M_{\text{B}}(\bC, G_0) := \log\Big(\big[\max_{S \in \bC}\EE_{S'} |\cV_{S,S'}| \big]^{-1}\Big), \label{buffer:entropy:def:eq}
\end{align}
where the expectation $\EE_{S'}$ is taken from uniformly sampling $S'$ from $\bC$.
\end{definition}
We want the buffer entropy to be as large as possible to achieve sharp lower bounds. Note the following trivial bound on the size $|\cV_{S,S'}|$
\[
|\cV_{S,S'}| \leq \sum_{v \in V(S)} \mathds{1}(v \in \cV_{S,S'}) + \sum_{v \in V(S')} \mathds{1}(v \in \cV_{S,S'}) .
\]
An important condition allowing us to relate the signal strength lower bounds to buffer entropy requires that the divider is such that the variables $\{ \mathds{1}(v \in \cV_{S,S'})\}_{v \in V(S)}$ are negatively associated.

\begin{definition}[Incoherent Divider] \label{asmp:na} The collection of edge sets $\bC$ is called an incoherent divider with a null base $G_0$, if for any fixed $S \in \bC$, the random variables $\{ \mathds{1}(v \in \cV_{S,S'})\}_{v \in V(S)}$ with respect to a uniformly sampled $S'$ from $\bC$ are negatively associated. In other words, for any pair of disjoint sets $I, J \subseteq V(S)$ and any pair of coordinate-wise nondecreasing functions $f, g$  we have:
$$
\Cov\big(f(\{\mathds{1}(v \in \cV_{S,S'})\}_{v \in I}), g(\{\mathds{1}(v \in \cV_{S,S'})\}_{v \in J})\big) \leq 0.
$$
\end{definition}
We show concrete constructions of incoherent dividers in Examples \ref{unfixed:sparse:star}, \ref{clique:detection:example} and \ref{cycle:detection:example}. As a remark, negative association is satisfied by a variety of classical discrete distributions such as the multinomial and hypergeometric, and even more generally by the class of permutation distributions \citep[e.g.]{joag1983negative, dubhashi1996balls}. It is a standard assumption that has been exploited in other works \citep[e.g.]{addario2010comb} for obtaining lower bounds.

Besides the packing entropy, the lower bound in Theorem \ref{suff:cond:lower:bound:many} involves the maximum degree $\|\Ab_0\|_1$ and the spectral norm $\|\Ab_0\|_2$. We define similar quantities for the scaling edge sets case. For a divider $\bC$ with null base $G_0$ and any two edge sets $S,S' \in \bC$ define the notation:
\begin{align}
\Ab_{S,S'} := \Ab_{0} +  \Ab_{S} +  \Ab_{S'}.\label{matrix:Ass}
\end{align}
As the sizes of $S, S' \in \bC$ are no longer ignorable, we need to consider the matrix $\Ab_{S,S'}$ (\ref{matrix:Ass}) instead. Denote the uniform maximum degree as $\Gamma :=  \max_{S,S' \in \bC} \|\Ab_{S,S'}\|_1$ and  uniform spectral norm as $\Lambda :=  \max_{S,S' \in \bC} \|\Ab_{S,S'}\|_2$. We define
\begin{align*}	
\cR & := \max_{S,S' \in \bC} \frac{|S \cap S'|}{|\cV_{S,S'}|},  ~~~~~~  \cB :=   \Lambda^4 \wedge \max_{S,S' \in \bC}(\Gamma^2 |\cV_{S, S'}|) .
\end{align*}
$\cR$ is an edge-node ratio measuring how dense the edge set $S \cap S'$ is compared to the vertex buffers. The quantity $\cB$ is an auxiliary quantity which assembles maximum degrees, spectral norms and buffer sizes and helps to obtain a compact lower bound formulation.

Below we connect the structural features we defined above to the lower bound. 
Recall definitions (\ref{generic:S0:def}) and (\ref{generic:S1:def}) on $\cS_0(\theta, s)$ and $\cS_1(\theta, s)$. We have the following theorem.
\begin{theorem} \label{comb:NAS} Let $\bC$ be an incoherent divider with a null base $G_0$. Then if $M_{\text{B}}(\bC, G_0) \rightarrow \infty$ and
\begin{align}
\theta \leq \sqrt{ \frac{M_{\text{B}}(\bC, G_0)}{4 n \cR}} \wedge \sqrt{\frac{\cR}{\cB}} \wedge \frac{1 - C^{-1}}{2 \sqrt{2} \Gamma}, \label{signal:strength}
\end{align}
the minimax risk satisfies
$
{\displaystyle \liminf_{n \rightarrow \infty}} \gamma(\cS_0(\theta, s), \cS_1(\theta, s)) = 1.
$
\end{theorem}
When the sample size $n$ is sufficiently large, the buffer entropy term on the right hand side of (\ref{signal:strength}) is the smallest term and drives the bound which bares similarity to Theorem \ref{suff:cond:lower:bound:many}. 

To better illustrate the usage of Theorem \ref{comb:NAS} we consider three examples.  First we focus on the problem of testing whether the maximum degree in the graph is $\leq s_0$ vs $\geq s_1$. When $s_0 = 0$, this problem is related to the problem of detecting a set of $s_1$ signals in the normal means model \cite[e.g.]{ingster1982minimax, baraud02non-asymptotic, donoho2004higher, addario2010comb, verzelen2010goodness, Arias-Castro2011Global}. However the two problems are distinct, since we are studying structural testing in the graphical model setting. Given $s_0 < s_1 \leq s$, we let the sub-decomposition be $\cG_0 = \{G~|~ d_{\max}(G) \leq s_0\}$ and $\cG_1 = \{G~|~ d_{\max}(G) \geq s_1\}$.  
We summarize our results in the following

\begin{example}[Maximum Degree Test with Scaling Divider]\label{unfixed:sparse:star} Let $\cS_0(\theta, s)$ and $\cS_1(\theta, s)$ be defined   in (\ref{generic:S0:def}) and (\ref{generic:S1:def}). Assume that $s \sqrt{\log d/n} = O(1)$ and $s = O(d^{\gamma})$ for some $\gamma < 1/2$. Then for a small enough absolute constant $\kappa$ if $\theta < \kappa \sqrt{\log d/n}$ we have $\liminf_{n \rightarrow \infty} \gamma(\cS_0(\theta,s), \cS_1(\theta, s)) = 1.$
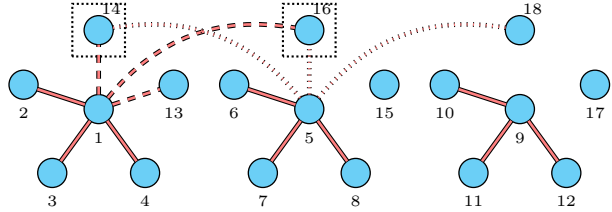
\begin{figure}[t]
\centering
\begin{tikzpicture}[scale=.7]
\tikzset{LabelStyle/.style = {above, fill = white, text = black, fill opacity=0, text opacity = 1}}
\draw[thick, densely dotted] (-0.5,1) rectangle (.5,2);
\draw[thick, densely dotted] (4-0.5,1) rectangle (4+.5,2);
\SetVertexNormal[Shape      = circle,
                  FillColor  = cyan!50,
                  MinSize = 11pt,
                  InnerSep=0pt,
                  LineWidth = .5pt]
   \SetVertexNoLabel
      \tikzset{EdgeStyle/.style= {thin,
                                double          = red!50,
                                double distance = 1pt}}
\begin{scope}[rotate=90]
 \grEmptyStar[prefix=a,RA=1.5]{6}
       \end{scope}
 \Edge(a5)(a1)
  \Edge(a5)(a2)
    \Edge(a5)(a3)
   \begin{scope}[shift={(4cm, 0cm)},rotate=90]
    \grEmptyStar[prefix=b,RA=1.5]{6}
 \Edge(b5)(b1)
  \Edge(b5)(b2)
    \Edge(b5)(b3)
   \end{scope}
               \node[below] at (a5.-90) {\tiny $1$};
        \node[below] at (a1.-90) {\tiny $2$};
    \node[below] at (a2.-90) {\tiny $3$};
        \node[below] at (a3.-90) {\tiny $4$};
       \node[below] at (a4.-90) {\tiny $13$};
        \node[above] at (a0.+30) {\tiny $14$};
      \begin{scope}[shift={(8cm, 0cm)},rotate=90]
    \grEmptyStar[prefix=c,RA=1.5]{6}
     \Edge(c5)(c1)
  \Edge(c5)(c2)
    \Edge(c5)(c3)
   \end{scope}
        \tikzset{EdgeStyle/.style= {thin,dashed,
                                double          = red!50,
                                double distance = 1pt}}
        \Edge(a5)(a0)
        \Edge(a5)(a4)
                \tikzset{EdgeStyle/.style= {thin,dashed, bend left,
                                double          = red!50,
                                double distance = 1pt}}
                        \Edge(a5)(b0)  
          \tikzset{EdgeStyle/.style= {thin,dotted,
                                double          = red!50,
                                double distance = 1pt}}
          \Edge(b5)(b0)
              \tikzset{EdgeStyle/.style= {thin,dotted,bend right,
                                double          = red!50,
                                double distance = 1pt}}
                \Edge(c0)(b5)  
                \Edge(b5)(a0)  
               \node[below] at (b5.-90) {\tiny $5$};
        \node[below] at (b1.-90) {\tiny $6$};
    \node[below] at (b2.-90) {\tiny $7$};
        \node[below] at (b3.-90) {\tiny $8$};
                \node[below] at (b4.-90) {\tiny $15$};
                        \node[above] at (b0.+30) {\tiny $16$};
                                       \node[below] at (c5.-90) {\tiny $9$};
        \node[below] at (c1.-90) {\tiny $10$};
    \node[below] at (c2.-90) {\tiny $11$};
        \node[below] at (c3.-90) {\tiny $12$};
                \node[below] at (c4.-90) {\tiny $17$};
                        \node[above] at (c0.+30) {\tiny $18$};
\end{tikzpicture}
\vspace{-1em}
 \caption{Test for maximum degree  $\cG_0 = \{G~|~ d_{\max}(G) \leq s_0\}$ and $\cG_1 = \{G~|~ d_{\max}(G) \geq s_1\}$ with $s_0 = 3$ and $s_1 = 6$. We split the vertices into  two parts $\{1, \ldots,  \lfloor \sqrt{d} \rfloor\}$ and $\{ \lfloor \sqrt{d} \rfloor+1, \ldots, d\}$.  We use the first part of vertices to construct $s_0$-star graphs as $G_0$ (visualized with solid edges). To construct the divider $\bC$, we select any $s_1-s_0$ vertices (e.g., vertices $13, 14, 16$) from the second vertices part $\{ \lfloor \sqrt{d} \rfloor+1, \ldots, d\}$ and connect them to any center of the $s_0$-star graphs in $G_0$ (e.g., vertex 1). This gives us one of the edge sets $S \in \bC$ (e.g., $S = \{(1,13), (1,14), (1,16)\}$ depicted in dashed edges in the figure). $\bC$ is comprised by all such edge sets. We depicted the vertex buffer $\cV_{S,S'} = \{14, 16\}$ for $S = \{(1,13), (1,14), (1,16)\}$ and $S' = \{(5,14), (5,16), (5,18)\}$ and $S \cap S' = \varnothing$.}\label{sparse:star:prior:multi:edge}\vspace{-18pt}
\end{figure}

\end{example}
Due to space limitations, we show how  this example follows from Theorem \ref{comb:NAS} in Section \ref{supp:bounded:edges} of the supplement. Here, we simply sketch the construction of the divider in Fig \ref{sparse:star:prior:multi:edge}. On an important note, the negative association of the random variables $\{\mathds{1}(v \in \cV_{S,S'})\}_{v \in V(S)}$ can be easily deduced by a result of \cite{joag1983negative}. Our second example further illustrates the usage of Theorem \ref{comb:NAS} with a clique detection problem. Define the null and alternative parameter spaces: $\cS_0 := \{\Ib_d\}$ and
\begin{align*}
 \cS_1(\theta,s) := \{\Ib_d + \theta (\vb \vb^T - \Ib_d) ~|~  \theta \in (0,1), \forall j: v_j \in \{\pm 1,0\}, \|\vb\|_2^2 = s\}.
\end{align*}
This setup is related to that in \cite{berthet2013optimal, johnstone2009consistency}. Our case is different from previous works because we parametrize the precision matrix rather than the covariance matrix, and the parametrization is distinct. Under our parametrization, the graph in the alternative hypothesis consists of a single $s$-clique.
\begin{example}[Sparse Clique Detection]\label{clique:detection:example} Suppose $s = O(d^{\gamma})$ for a $\gamma < 1/2$. For values of $\theta < \frac{1}{4\sqrt{2}s} \wedge \sqrt{\frac{\log (d/s^2)}{4ns}}$ we have $
\liminf_{n\rightarrow \infty} \gamma(\cS_0, \cS_1(\theta,s)) = 1.$
\end{example}

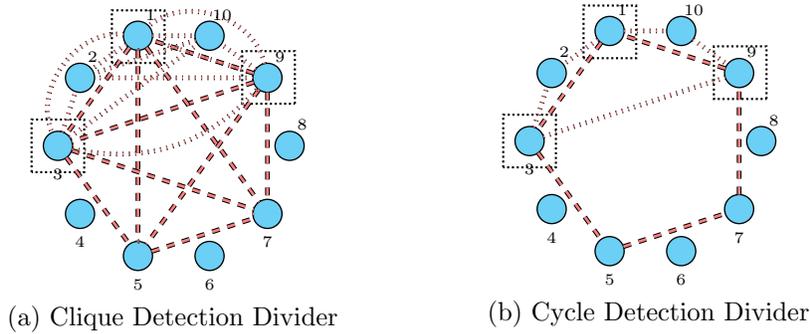
\begin{figure}
\centering
\begin{subfigure}{.5\textwidth}
\centering
\begin{tikzpicture}[scale=.7]
\draw[thick, densely dotted] (-1.175,1.575) rectangle (-0.175,2.575);
\draw[thick, densely dotted] (2.3,0.8) rectangle (1.3,1.8);
\draw[thick, densely dotted] (-2.7,-0.5) rectangle (-1.7,.5);
\SetVertexNormal[Shape      = circle,
                  FillColor  = cyan!50,
                  MinSize = 11pt,
                  InnerSep=0pt,
                  LineWidth = .5pt]
   \SetVertexNoLabel
   \tikzset{LabelStyle/.style = {above, fill = white, text = black, fill opacity=0, text opacity = 1}}
         \tikzset{EdgeStyle/.style= {thin,dotted,
                                double          = red!50,
                                double distance = 1pt}}
  \begin{scope}[shift={(0cm, 0cm)}]
\grEmptyCycle[prefix=c,RA=2.2]{5}%
 \begin{scope}[rotate=36]
         \tikzset{EdgeStyle/.style= {thin,dashed,
                                double          = red!50,
                                double distance = 1pt}}
\grComplete[prefix=d,RA=2.2]{5}%
  \end{scope}
  \Edge(d1)(c1)
    \Edge(d1)(c2)
    \Edge(d2)(c2)
    \Edge(d2)(c1)
    \Edge(c1)(c2)
   \Edge(d0)(c2)
      \Edge(d0)(c1)
            \Edge(d0)(d1)
                     \tikzset{EdgeStyle/.style= {thin,dotted, bend right = 60,
                                double          = red!50,
                                double distance = 1pt}}
    \Edge(d0)(d1)
    \Edge(d1)(d2)
                         \tikzset{EdgeStyle/.style= {thin,dotted, bend right = 30,
                                double          = red!50,
                                double distance = 1pt}}
                                    \Edge(d2)(d0)
        \node[above] at (c1.+30) {\tiny $10$};
    \node[above] at (c2.+30) {\tiny $2$};
        \node[below] at (c3.-90) {\tiny $4$};
       \node[below] at (c4.-90) {\tiny $6$};
        \node[above] at (c0.+30) {\tiny $8$};
        \node[above] at (d1.+30) {\tiny $1$};
    \node[below] at (d2.-90) {\tiny $3$};
        \node[below] at (d3.-90) {\tiny $5$};
       \node[below] at (d4.-90) {\tiny $7$};
        \node[above] at (d0.+30) {\tiny $9$};
\end{scope}
\end{tikzpicture}\vspace{-4pt}\caption{Clique Detection Divider}\label{sparse:clique:subfig}\vspace{-8pt}
\end{subfigure}\begin{subfigure}{.5\textwidth}
\centering
\begin{tikzpicture}[scale=.7]
\draw[thick, densely dotted] (-1.175,1.575) rectangle (-0.175,2.575);
\draw[thick, densely dotted] (2.3,0.8) rectangle (1.3,1.8);
\draw[thick, densely dotted] (-2.7,-0.5) rectangle (-1.7,.5);
\SetVertexNormal[Shape      = circle,
                  FillColor  = cyan!50,
                  MinSize = 11pt,
                  InnerSep=0pt,
                  LineWidth = .5pt]
   \SetVertexNoLabel
   \tikzset{LabelStyle/.style = {above, fill = white, text = black, fill opacity=0, text opacity = 1}}
         \tikzset{EdgeStyle/.style= {thin,dotted,
                                double          = red!50,
                                double distance = 1pt}}
  \begin{scope}[shift={(0cm, 0cm)}]
\grEmptyCycle[prefix=c,RA=2.2]{5}
 \begin{scope}[rotate=36]
         \tikzset{EdgeStyle/.style= {thin,dashed,
                                double          = red!50,
                                double distance = 1pt}}
\grCycle[prefix=d,RA=2.2]{5}
  \end{scope}
  \Edge(d1)(c1)
    \Edge(d1)(c2)
    \Edge(d2)(c2)
    \Edge(d2)(d0)
        \Edge(c1)(d0)
                     \tikzset{EdgeStyle/.style= {thin,dotted, bend right = 60,
                                double          = red!50,
                                double distance = 1pt}}
                         \tikzset{EdgeStyle/.style= {thin,dotted, bend right = 30,
                                double          = red!50,
                                double distance = 1pt}}
        \node[above] at (c1.+30) {\tiny $10$};
    \node[above] at (c2.+30) {\tiny $2$};
        \node[below] at (c3.-90) {\tiny $4$};
       \node[below] at (c4.-90) {\tiny $6$};
        \node[above] at (c0.+30) {\tiny $8$};
        \node[above] at (d1.+30) {\tiny $1$};
    \node[below] at (d2.-90) {\tiny $3$};
        \node[below] at (d3.-90) {\tiny $5$};
       \node[below] at (d4.-90) {\tiny $7$};
        \node[above] at (d0.+30) {\tiny $9$};
\end{scope}
\end{tikzpicture}
\vspace{-4pt}\caption{Cycle Detection Divider}\vspace{-8pt}\label{cycle:detection:subfig}
\end{subfigure}

\caption{Sparse clique and cycle detection with $s=5$. In both cases we set $G_0 = (\overline V, \varnothing)$ and visualize two intersecting $S, S' \in \bC$: in Fig \ref{sparse:clique:subfig} $S$ is the $5$-clique $K_{\{1,3,5,7,9\}}$, $S'$ is the $5$-clique $K_{\{1,2,3,9,10\}}$ and $\cV_{S,S'} = \{1,3,9\}$; in Fig \ref{cycle:detection:subfig} $S$ is the $5$-cycle $C_{\{1,3,5,7,9\}}$, $S'$ is the $5$-cycle $C_{\{1,2,3,9,10\}}$ and $\cV_{S,S'} = \{1,3,9\}$.}\label{joint:cycle:clique} \vspace{-18pt}
\end{figure}

We show how Example \ref{clique:detection:example} follows from Theorem \ref{comb:NAS} in Section \ref{supp:bounded:edges} of the supplement. The divider construction we use is simply drawing $s$ vertices and connecting them to form a $s$-clique. Fig \ref{sparse:clique:subfig} illustrates two sets from the divider along with their vertex buffer.

We conclude this Section by a final example on cycle detection. In this example the sub-decomposition is $\cG_0 = \{(\overline V, \varnothing)\}$, and $\cG_1 = \{(\overline V, C) | C = \{(v_1,v_2), \ldots, (v_{s-1}, v_s), (v_s, v_1)\}, v_i \in \overline V\}$ for an integer $s \in \NN$. We have the following example, whose proof can be found in Section \ref{supp:bounded:edges} of the supplement. We show two sets from the divider on Figure \ref{cycle:detection:subfig}.

\begin{example}[Sparse Cycle Detection]\label{cycle:detection:example} Suppose $s = O(d^{\gamma})$ for a $\gamma < 1/2$. Then for a small enough absolute constant $\kappa$ if $\theta < \kappa \sqrt{\log d/n}$ we have \\
$\liminf_{n \rightarrow \infty} \gamma(\cS_0(\theta,s), \cS_1(\theta, s)) = 1.$
\end{example}

\subsection{Upper Bounds} Algorithms matching the lower bounds developed in Sections \ref{bounded:edges} and \ref{scaling:edges} are discussed in Section \ref{upper:bounds:multi:NAS} of the supplement. 

\section{Discussion}\label{disc:sec}
{ In this manuscript we provide general results for upper and lower bounds of testing graph properties. 
There is still room to improve the proof techniques for lower bounding. Our arguments rely only on ``one-sided'' alternatives, and it is possible to obtain sharper bounds by additional randomization such as in \cite{baraud02non-asymptotic}.  Additionally, we use the Gaussian distribution to quantify the lower bounds. We are further  interested in generalizing our results to other important graphical models, such as the Ising model, in our future studies.}

\section{Proof of Theorem \ref{suff:cond:lower:bound}} \label{proof:of:suff:lower:bound:sec}

In this section we prove the main result of Section \ref{sec:generic:lower}. To begin with, we give a high level picture of our proof. The argument consists of four major steps. 
Our first three steps will show that, there exists a constant $R$ such that if $\theta \leq \frac{1 - C^{-1}}{\sqrt{2}(\|\Ab_0\|_1 + 2)}$ we have:
\begin{align}
\gamma(\cS_0(\theta, s), \cS_1(\theta, s)) \geq 1 - \frac{1}{2}\sqrt{\frac{1}{|\cC|^2} \sum_{e,e' \in \cC} \exp\left(n  \frac{(R\theta)^{2 d_{G_0}(e,e') + 2}}{d_{G_0}(e,e') + 1}\right) - 1}. \label{useful:bound}
\end{align}
To establish this result, in the first step, we select one precision matrix from the null $\cS_0(\theta, s)$ and  a set of precision matrices from the alternative $\cS_1(\theta, s)$.
In the second step, we apply Le Cam's Lemma to the precision matrices constructed above to get a lower bound of $\gamma(\cS_0(\theta, s), \cS_1(\theta, s))$. In the third step, we establish trace perturbation inequalities to further connect the lower bound achieved in the second step to the geometric quantities of the graphs. In the fourth step, we prove the theorem by showing that the right hand side of (\ref{useful:bound}) goes to 1 if \eqref{scaling:conditions} is satisfied.

\vspace{.3cm}
\noindent {\large{\bf Step 1}\textit{(Matrix Construction).}}\\
In this step we construct a class of precision matrices based on the null base graph $G_0$ and the divider set $\cC$ and verify that these matrices indeed belong to the sets $\cS_0(\theta,s)$ and $\cS_1(\theta,s)$. We begin with giving the upper bound of matrix norms of adjacent inequalities.  Let $\Ab_0$ be the adjacency matrix of the graph $G_0$. Observe that since $\Ab_0$ is symmetric, by H\"{o}lder's inequality $\|\Ab_0\|_2 \leq \sqrt{\|\Ab_0\|_1 \|\Ab_0\|_{\infty}} = \|\Ab_0\|_1 \leq D$.

Similarly, denote with $ \Ab_{e}$ the adjacency matrix of the graph $(\overline V, \{e\})$ for $e \in \cC$. Under our assumptions it follows that $\Ab_0 +  \Ab_{e}$ is the adjacency matrix of the graph $G_{e} = (\overline V, E_0\cup \{e\})$. For brevity, for any two edges $e,e' \in \cC$ we define the shorthand notation $\Ab_{e,e'} := \Ab_0 + \Ab_e + \Ab_{e'}$. Take $\bTheta_0 = \Ib + \theta \Ab_0$, $\bTheta_e = \Ib + \theta (\Ab_0 + \Ab_{e})$, $ \bTheta_{e,e'} = \Ib + \theta \Ab_{e,e'}$, for $e,e' \in \cC$, $\theta > 0$. By the triangle inequality for any $e, e' \in \cC$ we have:
\begin{align*}
\max(\|\Ab_0\|_2, \|\Ab_0 + \Ab_{e}\|_2, \|\Ab_{e,e'}\|_2) &\leq \|\Ab_0\|_2 + 2,\\
\max(\|\Ab_0\|_1, \|\Ab_0 + \Ab_{e}\|_1, \|\Ab_{e,e'}\|_1) & \leq \|\Ab_0\|_1 + 2.
\end{align*}
Recall definition (\ref{MS:def}) of the set $\cM(s)$. Next we make sure that the matrices $\bTheta_0$ and $\bTheta_e$ fall into $\cM(s)$ and in addition the matrix $\bTheta_{e,e'} > 0$. For the upper bounds, it suffices to choose $\eta$ satisfying:
\begin{align*}
\max(\|\bTheta_0\|_2, \|\bTheta_e\|_2) & \leq 1 + (\|\Ab_0\|_2 + 2)\theta \leq C,\\
\max(\| \bTheta_0\|_1, \| \bTheta_e\|_1)  & \leq 1 + (\|\Ab_0\|_1 + 2)\theta \leq L.
\end{align*} 

Recall that $\|\Ab_0\|_2 \leq \|\Ab_0\|_1$, and $C \leq L$ hence both inequalities are implied if $1  + (\|\Ab_0\|_1 + 2)\theta \leq C$. This inequality holds since
$$
\theta < \frac{(1 - C^{-1}) \wedge e^{-1/2}}{\sqrt{2}(D + 2)} \leq \frac{C - 1}{(\|\Ab_0\|_1 + 2)},
$$
where the last inequality is true since $D = \|\Ab_0\|_1$, and $C \geq 1$ and therefore $C - 1\geq 1 - C^{-1}$. Furthermore, by Weyl's inequality:
\begin{align}\label{lower:bound:eigen}
\lambda_d( \bTheta_0), \lambda_d( \bTheta_e), \lambda_d(\bTheta_{e,e'}) \geq 1 - \theta(\|\Ab_0\|_2 + 2) \geq 1 - \theta(\|\Ab_0\|_1 + 2),
\end{align}
where $\lambda_d$ denotes the smallest eigenvalue of the corresponding matrix. We want to ensure that the last term is at least $C^{-1}$. Since by assumption $\theta < \frac{1 - C^{-1}}{\sqrt{2}(\|\Ab_0\|_1 + 2)}$ the above inequalities are satisfied. Furthermore, we have $G_0\in \cG_0, G_e \in \cG_1$ for all $e \in \cC$ and hence the induced graphs $G(\bTheta_0) \in \cG_0$ and $G(\bTheta_e) \in \cG_1$ for all $e \in \cC$. This shows that $\bTheta_0 \in \cS_0(\theta,s)$ and $\bTheta_e \in \cS_1(\theta,s)$. We also obtain as a by-product that $\bTheta_{e,e'} \geq 0$. 

\vspace{.3cm}
\noindent {\large{\bf Step 2}\textit{(Minimax Risk Lower Bound).}}\\
In this step we obtain a lower bound on the minimax risk driven by Le Cam's Lemma \citep{lecam1973convergence}. Using a determinant identity we control the chi-square divergence by the traces of adjacency matrices' powers. Put the uniform prior on $\cC$ and consider the models generated by $N(0, (\bTheta_e)^{-1})$ where $e \in \cC$. Define:
$$
\overline \PP_{} = \frac{1}{|\cC|} \sum_{e \in \cC} \PP_{\bTheta_e},
$$
where $\PP_{\bTheta_e}$ we define the probability measure when the data is i.i.d. $\bX_i \sim N(0, (\bTheta_e)^{-1})$, and let $\PP_{\bTheta_0}$ be the probability measure when the data is i.i.d. $\bX_i \sim N(0, (\bTheta_0)^{-1})$.
Next, by Neyman-Pearson's lemma we have:
\begin{align}
\gamma(\cS_0, \cS_1) \geq \inf_{\psi}\Bigl[\PP_{\bTheta_0}(\psi = 1) + \overline \PP_{} (\psi = 0)\Bigr] = 1 - \TV(\overline \PP_{}, \PP_{\bTheta_0}), \label{ave:risk}
\end{align}
where for two probability measures $P,Q \ll \lambda$ on a measurable space $(\Omega, \cA)$, $\TV$ stands for \textit{total variation distance}, and is defined as 
$$
\TV(P, Q) = \sup_{A \in \cA}|P(A) - Q(A)| = \frac{1}{2} \int \Big|\frac{d P}{d \lambda}(\omega) - \frac{d Q}{d \lambda}(\omega) \Big| d \lambda(\omega).
$$ 
By Cauchy-Schwartz one has:
\begin{align}\label{lik:ratio:bound:chi:sq}
1 - \TV(\overline \PP_{\pi}, \PP_{\bTheta_0}) \geq 1 - \frac{1}{2} \sqrt{D_{\chi^2}(\overline \PP_{\pi}, \PP_{\bTheta_0})},
\end{align}
where $D_{\chi^2}(P, Q)$ is the chi-square divergence between the measures $P, Q$ and is defined as:
$$
D_{\chi^2}(P, Q) = \int \left(\frac{d P}{d Q}(\omega) - 1\right)^2 d Q(\omega) = \int \left(\frac{d P}{d Q}(\omega)\right)^2 d Q(\omega)  - 1,
$$
assuming that $P \ll Q$. Observe that $D_{\chi^2}(\overline \PP_{}, \PP_{\bTheta_0})$ can be equivalently expressed as:
\begin{align}
D_{\chi^2}(\overline \PP_{}, \PP_{\bTheta_0}) = \EE_{\bTheta_0} L^2_{\bTheta_0} - 1,\label{lik:ratio:bound}
\end{align}
where $L_{\bTheta_0} = \frac{1}{|\cC|} \sum_{e \in \cC} \frac{d \PP_{\bTheta_e}}{d \PP_{\bTheta_0}}$ is the integrated likelihood ratio, and $\EE_{\bTheta_0}$ denotes the expectation under $\bX_i \sim N(0, (\bTheta_0)^{-1})$. Hence by (\ref{ave:risk}) and (\ref{lik:ratio:bound:chi:sq}), it suffices to obtain upper bounds on the integrated likelihood ratio in order to lower bound the minimax risk (\ref{risk:def}). Writing out the likelihood ratio comparing the normal distribution with precision matrix $ \bTheta_0$ to the uniform mixture of normal distribution with precision matrix $ \bTheta_e$ for $e \in \cC$ we get:
\begin{equation*}
 L_{\bTheta_0} = \frac{1}{|\cC|} \sum_{e \in \cC} \bigg(\frac{\det(\bTheta_e)}{\det(\bTheta_0)}\bigg)^{n/2} \prod_{i = 1}^n \exp(-\bX_i^T \theta \Ab_{e} \bX_i/2).
\end{equation*}
To calculate the chi-square distance in \eqref{lik:ratio:bound}, next we square this expression and take its expectation under $\PP_{\bTheta_0}$ to obtain:
\begin{align}\label{eq:ee1}
\MoveEqLeft \EE_{\bTheta_0} L^2_{\bTheta_0} = \nonumber \\
& = \frac{1}{|\cC|^2} \sum_{e,e' \in \cC} \frac{(\det( \bTheta_e)\det( \bTheta_{e'}))^{n/2}}{(\det( \bTheta_0))^{n}} \EE_{ \bTheta_0}  \exp\Big(-\frac{\sum_{i = 1}^n\bX_i^T \theta (\Ab_{e} + \Ab_{e'}) \bX_i}{2}\Big) \nonumber\\
& = \frac{1}{|\cC|^2} \sum_{e,e' \in \cC} \bigg(\frac{\det(\bTheta_{e})}{\det(\bTheta_0)}\bigg)^{n/2}\bigg(\frac{\det(\bTheta_{e'})}{\det(\bTheta_{e,e'})}\bigg)^{n/2}. 
\end{align}

Next, we will expand the determinants above. 
Recall that we have ensured that $1 - \theta (\|\Ab_0\|_2 + 2) > 0$ (see \ref{lower:bound:eigen}). This implies 
\[
\theta \max(\|\Ab_0\|_2, \|\Ab_0 + \Ab_e\|_2, \|\Ab_0 + \Ab_{e'}\|_2, \|\Ab_{e,e'}\|_2) \leq 1.
\]
 For what follows for a symmetric matrix $\Ab_{d \times d}$ we denote its ordered eigenvalues with $\lambda_1(\Ab) \geq \lambda_2(\Ab) \geq \ldots \geq \lambda_d(\Ab)$. Let $\Ab \in \RR^{d \times d}$ be a symmetric matrix  such that $\|\Ab\|_2 \leq 1$. Then we have:
\begin{align*}
\MoveEqLeft \log \det (\Ib + \Ab)  = \sum_{j = 1}^d \log \lambda_{j}(\Ib + \Ab) =  \sum_{j = 1}^d \log  (1 + \lambda_{j}(\Ab)) \\
&=  \sum_{j = 1}^d \sum_{k = 1}^{\infty} (-1)^{k + 1} \frac{\lambda^k_{j}(\Ab)}{k}  =  \sum_{k = 1}^{\infty} (-1)^{k + 1} \frac{\Tr(\Ab^k)}{k}.
\end{align*}
Using the form $\det (\Ib + \Ab) = \exp( \log \det (\Ib + \Ab))$ and plugging the above equation into \eqref{eq:ee1}, we conclude that:
\begin{align*}
\EE_{\bTheta_0} L^2_{\bTheta_0}& =\frac{1}{|\cC|^2} \sum_{e,e' \in \cC}  \exp\bigg(\frac{n}{2}\sum_{k = 1}^{\infty} \frac{(-\theta)^{k}}{k} \bigg(T^k_1 + T^k_2\bigg) \bigg), \text{ where }
\end{align*}
\begin{align*}
T^k_1 := \Tr[\Ab_0^k -(\Ab_0 + \Ab_{e})^k] ~~~~ T^k_2 := \Tr[(\Ab_{e,e'})^k - (\Ab_0 + \Ab_{e'})^k].
\end{align*}

\vspace{.3cm}
\noindent {\large {\bf Step 3} \textit{(Trace Perturbation Inequalities).}}\\
In this step, we control $T_1^k + T_2^k$ in terms of $k$ and link it with the geometric quantities of the graph. We view $T_1^k$ as the perturbation difference between $ \Tr[\Ab_0^k]$ and $\Tr[(\Ab_0 + \Ab_{e})^k]$ and we treat $T_2^k$ similarly. In the following step, we aim to develop the perturbation inequalities for the trace of matrix powers.

 First we will argue that $T^k_1 + T^k_2 \geq 0$ for all $k \in \NN$. To see this recall that the trace operator of an adjacency matrix $\Mb$ satisfies 
$$
\Tr(\Mb^k) = \mbox{ number of all closed walks of length } k.
$$
First we consider case $e \neq e'$. Notice that all closed walks in $G(\Ab_0 + \Ab_{e})$ that do not belong to $G(\Ab_0)$ have to pass through the edge $e$ at least once. Similarly  all closed walks in $G(\Ab_0 + \Ab_{e'})$ that do not belong to $G(\Ab_0)$ have to pass through the edge $e'$ at least once. Furthermore, all closed walks of length $k$ passing through either $e$ or $e'$ belong to $G(\Ab_{e,e'})$. In addition $G(\Ab_{e,e'})$ might contain extra closed walks passing through both $e$ and $e'$. This shows:
$$
T^k_1 + T^k_2 \geq 0,
$$
for all $k$. This shows that when $k$ is odd we have $(- \theta)^k (T^k_1 + T^k_2) \leq 0$, and thus to control $\EE_{\bTheta_0} L_{\bTheta_0}^2$ it suffices to focus only on even $k$.

Next we prove that for $k <2 d_{G_0}(e,e') + 2$, we have $T^k_1 + T^k_2 \equiv 0$. To see this, first consider the case $e \neq e'$. Notice that the graph $G(\Ab_{e,e'})$ cannot contain paths passing through both $e$ and $e'$ unless $k \geq 2 d_{G_0}(e, e') + 2$. To see this, notice that no even length closed walk between $e$ and $e'$ can exist if the length of this walk is smaller than $2 d_{G_0}(e, e')$ plus the two edges $e$ and $e'$. This proves our claim in the case $e \neq e'$. In the special case $e = e'$, the length of the path trivially needs to be at least of length $2$ to pass through both $e$ and $e'$. 

We will now argue that for even $k \in \NN$ we have $T_1^k + T_2^k \leq 2(\|\Ab_0\|_2 + 2)^k$. In fact we will prove that $T_1^k \leq 0$ for all $k$ and $T_2^k \leq 2(\|\Ab_0\|_2 + 2)^k$ for all even $k$. To see that $T_1^k \leq 0$, note that $G(\Ab_0)$ contains less closed walks than $G(\Ab_0 + \Ab_e)$. 

Recall that for a symmetric matrix $\Ab_{d \times d}$ we denote its ordered eigenvalues with $\lambda_1(\Ab) \geq \lambda_2(\Ab) \geq \ldots \geq \lambda_d(\Ab)$. To this end we state a helpful result whose proof is deferred to the supplement.

\begin{lemma}\label{gen:lidskii} For two symmetric $m \times m$ matrices $\Ab$ and $\Bb$, and any constants $c_1 \geq c_2 \geq \ldots \geq c_m$, and a permutation $\sigma$ on $\{1,\ldots, m\}$ we have:
$$
\sum_{j = 1}^m c_{\sigma(j)} \lambda_j(\Ab + \Bb) \leq \sum_{j = 1}^m c_{\sigma(j)}\lambda_j(\Ab) + \sum_{j = 1}^m c_j \lambda_{j}(\Bb).
$$
\end{lemma}

Using Lemma \ref{gen:lidskii} for the matrices $\Ab = \Ab_e, \Bb = \Ab_0 + \Ab_{e'}$ with constants 
$$c_{\sigma(j)} = \sign(\lambda_j(\Ab_{e,e'}) - \lambda_j(\Ab_{e}))|\lambda_j(\Ab_{e,e'}) - \lambda_j(\Ab_{e})|^{k-1},$$
we obtain:
\begin{align*}
 \MoveEqLeft \sum_{j = 1}^d |\lambda_j(\Ab_{e,e'}) - \lambda_j(\Ab_{e})|^k  \leq \sum_{j = 1}^d c_j \lambda_j(\Ab_0  + \Ab_{e'}) \\
& \leq \Big[\sum_{j = 1}^d |c_j|^{\frac{k}{k-1}}\Big]^{\frac{k-1}{k}} \Big[ \sum_{j = 1}^d |\lambda_j(\Ab_0  + \Ab_{e'})|^{k}\Big]^{\frac{1}{k}},
\end{align*}
where the last inequality follows by H\"{o}lder's inequality. We conclude that:
\begin{equation}\label{eq:WH}
 \sum_{j = 1}^d |\lambda_j(\Ab_{e,e'}) - \lambda_j(\Ab_{e})|^k \leq \sum_{j = 1}^d|\lambda_j(\Ab_0  + \Ab_{e'})|^k.
\end{equation}

Next, observe that the negative adjacency matrix $-\Ab_{e}$ of the single edge graph $(\overline V, \{e\})$ has very simple eigenvalue structure: $1, -1$ and  $d-2$ zeros. Hence we conclude that for even $k$:
\begin{align}
\nonumber \MoveEqLeft T_2^k  = \Tr((\Ab_{e,e'})^k) - \Tr((\Ab_0 + \Ab_{e'})^k)  = \sum_{j = 1}^d |\lambda_j(\Ab_{e,e'})|^k  - \sum_{j = 1}^d |\lambda_j(\Ab_0 + \Ab_{e'})|^k \\
 & \leq |\lambda_{1}(\Ab_{e,e'})|^k - |\lambda_{1}(\Ab_{e,e'}) - 1|^k + |\lambda_{d}(\Ab_{e,e'})|^k - |\lambda_{d}(\Ab_{e,e'}) + 1|^k \nonumber \\
\nonumber & \leq |\lambda_{1}(\Ab_{e,e'})|^k + |\lambda_{d}(\Ab_{e,e'})|^k  \leq 2 \|\Ab_{e,e'}\|_2^k \leq  2(\|\Ab_0\|_2 + 2)^k. \nonumber
\end{align}

The last shows that indeed $T_1^k + T_2^k \leq 2(\|\Ab_0\|_2 + 2)^k$ as claimed. Putting everything together we obtain
\begin{align*}
\MoveEqLeft \sum_{k = 1}^{\infty} \frac{(- \theta)^k}{k} [T_1^k + T_2^k]
 \leq \sum_{2 | k, ~ k \geq2 d_{G_0}(e,e') + 2}^{\infty}\frac{\theta^k}{k}  [T_1^k + T_2^k] \\
 & \leq \sum_{2 | k, ~ k \geq 2 d_{G_0}(e,e') + 2}^{\infty} \frac{2((\|\Ab_0\|_2 + 2) \theta)^k}{k} \\
& \leq \frac{2((\|\Ab_0\|_2 + 2)  \theta)^{2 d_{G_0}(e,e') + 2}}{(2 d_{G_0}(e,e') + 2)(1 - (\theta(\|\Ab_0\|_2 + 2))^2)} \leq \frac{2((\|\Ab_0\|_2 + 2) \theta)^{2 d_{G_0}(e,e') + 2}}{d_{G_0}(e,e') + 1},
\end{align*}
where in the last inequality we used the fact that $\theta \leq \frac{1}{\sqrt{2}(\|\Ab_0\|_2 + 2)}$ which follows by the requirements on $\theta$. This completes the proof of (\ref{useful:bound}) where $R = \sqrt{2} (\|\Ab_0\|_2 + 2)$.

\vspace{.3cm}
\noindent {\large {\bf Step 4} \textit{(Rate Control).}}\\ 
The goal in this final step is to show that if (\ref{scaling:conditions}) holds, the minimax risk 
$$\liminf_{n \rightarrow \infty} \gamma(\cS_0(\theta, s), \cS_1(\theta, s)) =1.$$ 
The proof is technical, but the high-level idea is to clip the first $\log |\cC|$ degrees in (\ref{useful:bound}) and deal with two separate summations. It turns out that the scaling assumed on $\theta$ in (\ref{scaling:conditions}) is precisely enough to control both the summation of all degrees below $\log |\cC|$ and the summation of all degrees above $\log |\cC|$. Define the following quantities:
$$K_r := |\{(e,e') ~|~ e, e' \in \cC, d_{G_0} (e, e')= r\}|,$$
where $(e,e')$ are unordered edge pairs, and observe that $\sum_{r \geq 0} K_r = {|\cC| \choose 2} + |\cC|$ by definition. We will in fact, first show that if $\theta \leq \kappa \sqrt{\frac{\log |\cC|}{n}}$
for some small $\kappa$, and 
\begin{align}\label{geom:condition}
\sum_{r = 0}^{\lfloor\log |\cC|\rfloor} K_r = O(|\cC|^{2 - \gamma}),
\end{align}
for some $1/2 < \gamma \leq 1$, then $\liminf \gamma(\cS_0(\theta, s), \cS_1(\theta, s)) = 1$ provided that $|\cC| \rightarrow \infty$. We will then derive the Theorem as a corollary to this observation. 

By (\ref{useful:bound}) it suffices to control:
\begin{align}
\underbrace{\frac{2}{|\cC|^2}\sum_{(e,e'): d_{G_0}(e,e') \geq 1} \exp\bigg(n  \frac{\overline \theta^{2 d_{G_0}(e,e') + 2}}{d_{G_0}(e,e') + 1}\bigg)}_{I_1} + \underbrace{\frac{2K_0 - |\cC|}{|\cC|^2}\exp(n \overline \theta^2)}_{I_2},\label{useful:bound:prop}
\end{align}
where we will write $\overline \theta$ for $R \theta$ for brevity. 

First, observe that since $\theta < \frac{1-C^{-1}}{\sqrt{2}(D + 2)}$ then we have 
$$
\overline \theta < R \frac{(1-C^{-1}) \wedge e^{-1/2}}{\sqrt{2}(D + 2)} \le (1 - C^{-1})\wedge e^{-1/2} < e^{-1/2} < 1.
$$

We will show that when $\bar \theta$ is small, (\ref{useful:bound:prop}) is bounded by $1$ asymptotically, which in turn suffices to show that $\liminf \gamma(\cS_0(\theta, s), \cS_1(\theta, s))= 1$. Notice that $\overline \theta$ and $\theta$ are the same quantity up to the constant $R$ and hence $\theta < \kappa \sqrt{\frac{\log |\cC|}{n}}$ is equivalent to $\overline \theta < \bar \kappa\sqrt{\frac{\log |\cC|}{n}}$ for some sufficiently small $\bar \kappa$. We will require:
$$\bar \kappa < \sqrt{\gamma} \wedge \sqrt{2\gamma/c_0} \wedge (e c_0)^{-1/2}.$$
Observe that since $K_0 = O(|\cC|^{2-\gamma})$, $\bar \kappa^2 < \gamma$, and $\overline \theta < \bar \kappa \sqrt{\frac{\log |\cC|}{n}}$ we have:
$$
I_2 \leq \frac{(2K_0 - |\cC|)|\cC|^{\bar \kappa^2}}{|\cC|^2} \rightarrow 0.
$$
Next we tackle the term $I_1$ in (\ref{useful:bound:prop}). We will show that since $\overline \theta < 1$ by assumption, this term goes to $1$.
\begin{align*}
I_1 & = \frac{2}{|\cC|^2}\sum_{r \geq 1} K_r \exp(n \overline \theta^{2r + 2}/(r+1)) = \frac{2}{|\cC|^2}\sum_{r = 1}^{d-1} K_r |\cC|^{\bar \kappa^2 \overline \theta^{2r}/(r+1)}  + \frac{2K_{\infty}}{|\cC|^2}\\
& < \frac{2}{|\cC|^2}\sum_{r = 1}^{d-1} K_r |\cC|^{\overline \theta^{2r}/(r+1)}  + \frac{2K_{\infty}}{|\cC|^2},
\end{align*}
where the last inequality follows by the fact that $\bar \kappa^2 < \gamma <  1$. Splitting out the first $\lfloor \log |\cC| \rfloor$ terms out of this summation yields:
$$
I_1 < \underbrace{\frac{2}{|\cC|^2}\sum_{r = 1}^{\lfloor \log |\cC|\rfloor} K_r|\cC|^{\overline \theta^{2r}/(r + 1)}}_{I_{11}}  + \underbrace{\frac{2}{|\cC|^2}\sum_{r = \lfloor\log |\cC|\rfloor + 1}^{d-1} K_r|\cC|^{\overline \theta^{2r}/(r + 1)} + \frac{2K_{\infty}}{|\cC|^2}}_{I_{12}}.
$$
The first term is bounded by $
I_{11} \leq 2 \left(\sum_{r = 1}^{\lfloor \log |\cC|\rfloor } K_r\right) \frac{|\cC|^{\overline \theta^2/2}}{|\cC|^2} = o(1),$
where we used $\left(\sum_{r = 2}^{\lfloor \log |\cC| \rfloor} K_r\right) = O(|\cC|^{2 - \gamma})$ and the fact that $\overline \theta^2/2 < 1/2 < \gamma$. Next, we will argue that $|\cC|^{\overline \theta^{2r}/(r + 1)} \leq 1 + 3{\overline \theta^{2r}}(|\cC| - 1)/(r + 1).$
This follows by:
$$\exp(\log(|\cC|)\overline \theta^{2r}/(r + 1)) \leq 1 + 3\log(|\cC|)\overline \theta^{2r}/(r + 1) \leq 1 + 3(|\cC| - 1)\overline \theta^{2r}/(r + 1),$$ 
with the first inequality holding when $\log(|\cC|)\overline \theta^{2r}/(r + 1) < 1$, which is true since $\overline \theta < 1$, and $r \geq \lfloor\log|\cC|\rfloor + 1$. Hence we have:
\begin{align*}
I_{12} &\leq  \frac{2}{|\cC|^2}\sum_{r = \lfloor \log |\cC| \rfloor + 1}^{d-1} K_r (1 + 3(|\cC|-1){\overline \theta^{2r}}/(r + 1)) + \frac{2K_{\infty}}{|\cC|^2} \\
& \leq \left(1 - \frac{O(|\cC|^{2 - \gamma})}{|\cC|^2}\right) + \frac{6(|\cC| - 1)}{|\cC|^2} \sum_{r = \lfloor \log|\cC|\rfloor + 1}^{d - 1}\frac{K_r}{r} \overline \theta^{2r}\\
& \leq \left(1 - \frac{O(|\cC|^{2 - \gamma})}{|\cC|^2}\right) + \frac{6\overline \theta^{2 \lfloor \log |\cC|\rfloor + 2}}{1 - \overline \theta^2} |\cC|^{-1} \max_{ \lfloor \log|\cC|\rfloor + 1 \leq r} \frac{K_r}{r}.
\end{align*}

Paying closer attention to the second term we have:
\begin{align*}
\frac{6\overline \theta^{2 \lfloor \log |\cC|\rfloor + 2}}{1 - \overline \theta^2} |\cC|^{-1} \max_{ \lfloor \log|\cC|\rfloor + 1 \leq r} \frac{K_r}{r} & \leq \frac{6}{1 - e^{-1}} \overline \theta^{2 \lfloor \log |\cC| \rfloor+ 2} \frac{{|\cC| \choose 2} + |\cC|}{|\cC|}\\
& \leq \frac{6}{1 - e^{-1}} \overline \theta^{2 \lfloor \log |\cC| \rfloor+ 2} |\cC| = o(1),
\end{align*}
with the last equality holds since $
\overline \theta < \exp(-1/2),$
as we required. This combined with (\ref{useful:bound}) concludes the proof of 
$\liminf_n \gamma(\cS_0(\theta, s), \cS_1(\theta, s)) =  1,$
when $\theta \leq \kappa \sqrt{\frac{\log |\cC|}{n}}$.

Finally, notice that any subset of a divider $\cC$ is trivially a divider. Hence we can apply what we just showed to the set $N_{\log |\cC|} \subset \cC$ --- the maximal $\log |\cC|$-packing of $\cC$. Evaluating the constants $K_r$ on $N_{\log |\cC|}$ gives:
$$
K_0 = |N_{\log |\cC|}|, K_r = 0 \mbox{ for all } r \leq \lfloor \log |\cC|\rfloor,
$$
and since $\lfloor N_{\log |\cC|}\rfloor \leq \lfloor \log |\cC|\rfloor$ we conclude that $\sum_{r = 0}^{\lfloor N_{\log |\cC|}\rfloor} K_r = |N_{\log |\cC|}| = O( |N_{\log |\cC|}|^{2-\gamma})$ for any $0 < \gamma \leq 1$.

\appendix

\section*{Acknowledgements} The authors would like to thank the editor, associate editor and two referees for their suggestions, comments and remarks which led to significant improvements in the presentation of this manuscript.

\setlength{\bibsep}{1pt}
{\small
\bibliographystyle{ims}
\bibliography{sandwich}
}

\newpage

\begin{supplement}[id=suppA]
  \sname{supplement}
  \stitle{Supplementary Material to ``Combinatorial Inference for Graphical Models''}
  \slink[doi]{10.1214/00-AOSXXXXSUPP}
  \sdatatype{.pdf}
  The supplementary material is organized as follows:

\begin{itemize}

\item In Appendix \ref{simulation:sec}, we present extensive numerical studies, and a real dataset example.

\item In Appendix \ref{edge:removal:NAS:extension}, we show Proposition \ref{smart:bound} and the deletion-edge divider result on the lower bound parallel to Theorem~\ref{suff:cond:lower:bound}.

\item In Appendix \ref{detailed:upperbound:algos}, we give full details on the algorithms of Section \ref{struct:test:section} to match the lower bounds in the examples.

\item In Appendix \ref{extensions:GM:OTHERS}, we extend the algorithms in the previous section to  transelliptical graphical models.

\item In Appendix \ref{supp:bounded:edges}, we attach the proofs of the examples from Section \ref{general:lower:bound} of the main text, and we also formalize several algorithms matching the lower bounds.

\item In Appendix \ref{single:edge:app:proofs}, we provide the proof of Theorem \ref{suff:cond:lower:bound}.

\item In Appendix \ref{sec:alg-proof}, we prove the asymptotic type I and II errors for the algorithms in the paper.

\item In Appendix \ref{sec:multi-nas-proof}, we give the proofs of Theorems \ref{suff:cond:lower:bound:many} and \ref{comb:NAS}.

\item In Appendix \ref{appendix:clique:detection}, we prove the upper bound for clique detection.

\item In Appendix \ref{sec:aux}, we list some auxiliary results needed for the paper.

\item In Appendix \ref{sec:boot-proof}, we prove technical results related to the bootstrap procedure.
\end{itemize}

\section{Numerical Studies and Real Data Analysis}\label{simulation:sec}

In this section we present numerical analysis of Algorithm \ref{al:conn} for the hypothesis tests on the connectivity on a synthetic dataset. We also analyze the performance of Algorithm \ref{cycle:presence} for cycle presence testing. In addition, we implement the graph connectivity test to study the brain networks for  the ADHD-200 dataset.  

\subsection{Connectivity Testing}\label{conn:section:testing}

We present numerical simulations assessing the performance of Algorithm \ref{al:conn} for testing connectivity. Recall $\cG_0 = \{G \in \cG~|~ G \mbox{ disconnected} \}$ and $\cG_1 = \{G \in \cG~|~ G \mbox{ connected}\}$, and consider testing
$$\Hb_0: G(\bTheta^*) \in \cG_0 \mbox{ vs } \Hb_1: G(\bTheta^*) \in \cG_1.$$

In order to build the connected and disconnected graphs, we consider a chain graph of length $k$ with adjacency matrix as follows 
\begin{equation}\label{eq:chain}
 \Ab_{\text{Chain}(k)} = 
  \begin{pmatrix}
0 & 1 &&&0\\
1 & 0 & 1 \\
& 1 & \ddots & \ddots \\
& & \ddots & \ddots & 1\\
0& & & 1 & 0
\end{pmatrix} 
\in \RR^{k \times k}.
\end{equation}
We construct the connected graph with adjacency matrix $\Ab = \Ab_{\text{Chain}(d)}$ and the disconnected graphs with adjacency matrices 
\[
 \Ab_m = \diag(\Ab_{\text{Chain}(m)}, \Ab_{\text{Chain}(d-m)}), \text{ for } m \in [d-1].
\]
The corresponding graphs of $\Ab$ and $\Ab_1, \ldots, \Ab_{d-1}$ are depicted in Fig \ref{fig:conn-sim}.
\vspace{-1em}
 \begin{figure}[H]
\centering
 \begin{tikzpicture}[scale=.7]
\SetVertexNormal[Shape      = circle,
                  FillColor  = cyan!50,
                  MinSize = 11pt,
                  InnerSep=0pt,
                  LineWidth = .5pt]
   \SetVertexNoLabel
   \tikzset{LabelStyle/.style = {below, fill = white, text = black, fill opacity=0, text opacity = 1}}
   \tikzset{EdgeStyle/.style= {thin,
                                double          = red!50,
                                double distance = 1pt}}
    \begin{scope}\grPath[prefix=a,RA=2]{4}\end{scope}
                                                 \node[above left] at (a0.+60) {\tiny 1};
                                                 \node[above left] at (a1.+60) {\tiny 2};
                                                 \node[above left] at (a2.+60) {\tiny 3};
                                                 \node[above left] at (a3.+60) {\tiny 4};
         \Edge(a0)(a1)
                  \Edge(a1)(a2)
                  \Edge(a2)(a3)
     \begin{scope}[shift={(8,0)}]\grPath[prefix=b,RA=2]{3}\end{scope}
                                                      \node[above left] at (b0.+60) {\tiny 5};
                                                 \node[above left] at (b1.+60) {\tiny 6};
                                                 \node[above left] at (b2.+60) {\tiny 7};
     \Edge(b0)(a3)
     \Edge(b1)(b2)
          \Edge(b0)(b1)
        \begin{scope}[shift={(0,1)}]\grPath[prefix=c,RA=2]{4}\end{scope}
                                                         \node[above left] at (c0.+60) {\tiny 1};
                                                 \node[above left] at (c1.+60) {\tiny 2};
                                                 \node[above left] at (c2.+60) {\tiny 3};
                                                 \node[above left] at (c3.+60) {\tiny 4};
         \Edge(c0)(c1)
                  \Edge(c1)(c2)
     \begin{scope}[shift={(8,1)}]\grPath[prefix=d,RA=2]{3}\end{scope}
                                                           \node[above left] at (d0.+60) {\tiny 5};
                                                 \node[above left] at (d1.+60) {\tiny 6};
                                                 \node[above left] at (d2.+60) {\tiny 7};
     \Edge(d1)(d2)
          \Edge(d0)(d1)
\end{tikzpicture}
\vspace{-1em}
\caption{$G(\Ab_4)$ and $G(\Ab)$ for $d = 7$.}\label{fig:conn-sim}
\end{figure}
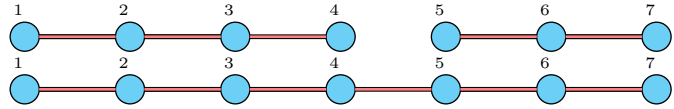
\vspace{-1em}
To show that  our test is uniformly valid over various disconnected graphs under the null hypothesis, we use different precision matrices for different repetitions. For each repetition, we randomly select $M \sim \text{Uniform}([d-1])$, choose the precision matrix $ \bTheta^*_M(\theta) = \Ib_d + \theta \Ab_{M}$ and generate i.i.d. samples $\bX_i \sim N({\bf 0},( \bTheta^*_M(\theta) )^{-1})$ for $i \in [n]$. Under the alternative hypothesis, we generate $n$ i.i.d. samples $\bX_1, \ldots, \bX_n$ from $N({\bf 0},(\bTheta^*(\theta))^{-1})$, where $\bTheta^*(\theta) = \Ib_d + \theta \Ab$.

The sample size is chosen from $n = 400$ and $600$ and the dimension varies in $d = 100, 125$ and $150$. The significance  level of the tests is set to $\alpha = 0.05$. Within each test we generate $3,000$ bootstrap samples to estimate the quantiles (\ref{quantile:def}). We estimate the precision matrix by the CLIME estimator  \citep{Cai2011Constrained} as follows
\begin{align} \label{generalCLIMEopt}
\hat{\bTheta}_{\lambda}=\argmin\|\bTheta\|_1,~~\textrm{s.t.}~~ \|\hat \bSigma\bTheta-\Ib_d\|_{\max} \leq \lambda.
\end{align} 
The tuning parameter is chosen by minimizing a $K$-fold cross validation risk
\[
 \text{CV}(\lambda) = \sum_{k=1}^K \|\hat \bSigma^{(k)} \hat \bTheta_{\lambda}^{(-k)} - \Ib_d \|_{\max},
\]
where $\hat \bSigma^{(k)}$ is the sample covariance matrix estimated on the $k$\textsuperscript{th} subset of data and $\hat\bTheta_{\lambda}^{(-k)}$ is the CLIME estimator without using the $k$\textsuperscript{th} subset of data. We use the first scenario and implement a 5-fold cross validation. It selects $\lambda = 1.5\sqrt{\log d/n}$ and we keep using this value of $\lambda$ throughout all remaining settings. We vary the signal strength $
\theta \in [0.25, 0.45]$ and estimate the risk of our test $\psi$ given by Algorithm \ref{al:conn}:
$$
\PP_{G(\bTheta^*) \in \cG_0 }( \psi = 1) + \PP_{G(\bTheta^*) \in \cG_1}( \psi = 0),
$$
by averaging the type I and type II errors under null and alternative settings through $200$ repetitions.
The results are visualized  in Fig \ref{simu:figure}, where the $Y$-axis represents the risk and the $X$-axis plots the signal strength $\theta$. As predicted by our results, with the increase in $\theta$ the risk curves become well controlled about the target value $0.05$. Moreover, there is little difference between the three risk curves within each plot. The curves for $d = 125, 150$ have a slightly worse performance than the curves for smaller dimension $d = 100$. This effect is expected, as the signal strength required to separate the null and alternative depends on $d$ only logarithmically. Furthermore, smaller signal strength is required to distinguish the null from the alternative when we increase $n$ from $400$ to $600$. Table \ref{size:values} reports the type I errors of $\psi$. The type I errors remain within the desired range of $0.05$. As we increase the signal strength $\theta$, the values of the errors increase from $0$ to the target level. This phenomenon is expected and is explained below Proposition \ref{multiple:edge:testing:validity}. 

\begin{figure}[t]
\begin{subfigure}{.5\textwidth}
  \centering
  \includegraphics[width=.8\linewidth]{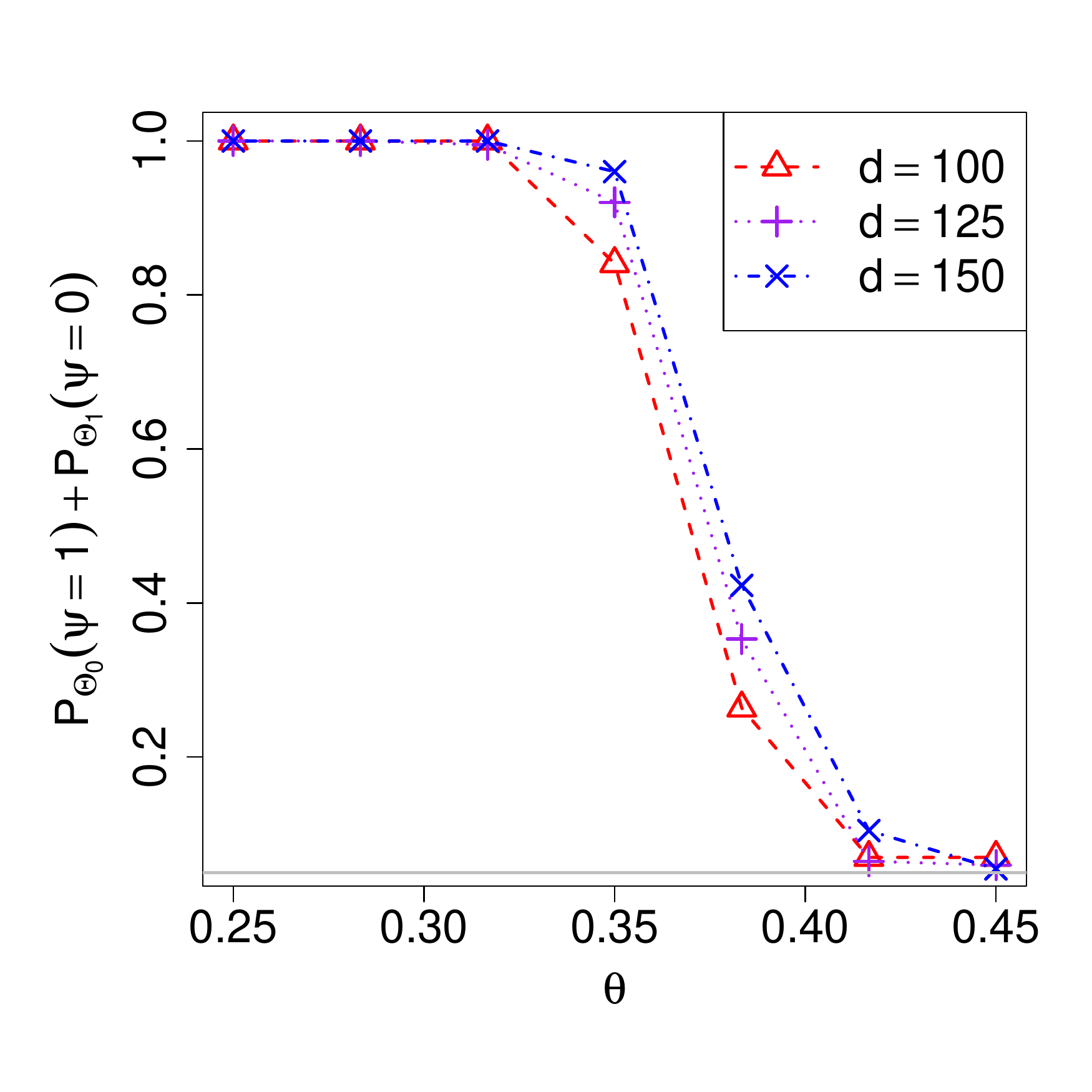}
  \vspace{-1em}
  \caption{$n = 400$}
  \label{fig:sfig1}
\end{subfigure}%
\begin{subfigure}{.5\textwidth}
  \centering
  \includegraphics[width=.8\linewidth]{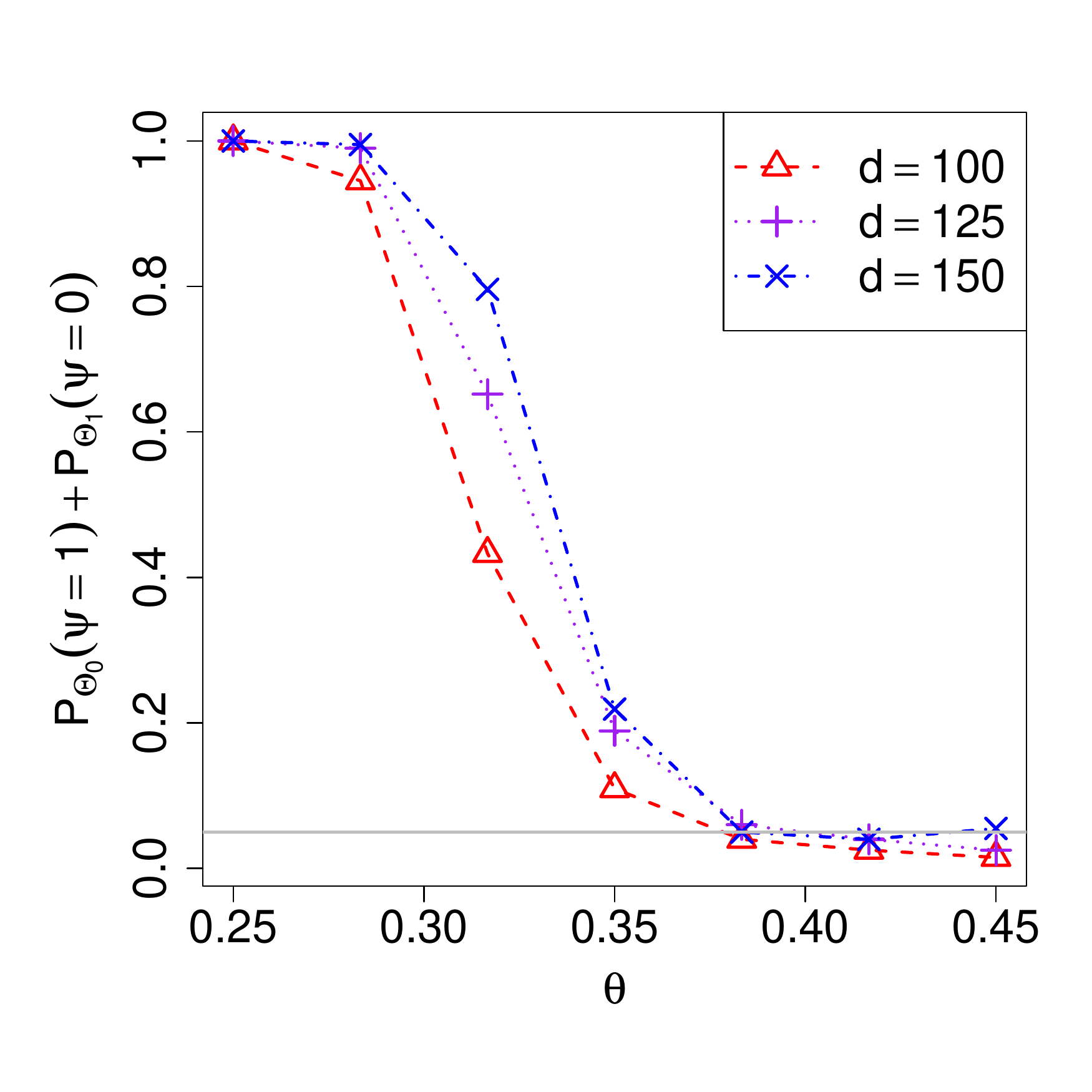}
  \vspace{-1em}
  \caption{$n = 600$}
  \label{fig:sfig3}
\end{subfigure}%
\vspace{-1em}
\caption{Type I and type II errors control for connectivity test with $\theta \in [0.25, 0.45]$. The grey horizontal line is the significance level $\alpha = 0.05$.} \label{simu:figure}\vspace{-18pt}
\end{figure}

\begin{table}[h] 
 \centering
\caption{Sizes for connectivity testing $\theta \in [0.25, 0.45]$}\label{size:values}
\begin{tabular}{ccrrrrrrr}
\hline
  \hline
  $n$ & \diagbox{$d$}{$\theta$} & 0.25 &  0.28 &  0.32 & 0.35 & 0.38 &  0.42 & 	0.45 \\ 
  \hline
  \multirow{3}{*}{$ 400$} & $100$ & 0.000 & 0.000 & 0.000 & 0.020 & 0.045 & 0.045 & 0.060 \\ 
  &$125$ & 0.000 & 0.000 & 0.000 & 0.010 & 0.045 & 0.050 & 0.060 \\ 
  &$150$ &  0.000 & 0.000 & 0.000 & 0.010 & 0.040 & 0.050 & 0.055 \\ 
  \hline
   \multirow{3}{*}{$600$} &100 & 0.000 & 0.005 & 0.020 & 0.045 & 0.030 & 0.025 & 0.030 \\ 
  &125 & 0.000 & 0.000 & 0.025 & 0.050 & 0.040 & 0.045 & 0.040 \\ 
  &150 & 0.000 & 0.000 & 0.015 & 0.055 & 0.040 & 0.045 & 0.040 \\ 
   \hline
   \hline
\end{tabular}
\end{table}

\subsection{Cycle Testing}

In this section we present numerical analysis of Algorithm \ref{cycle:presence} for cycle presence testing. Recall the sub-decomposition $\cG_0 = \{G \in \cG~|~ G \mbox{ has no cycles}\}$ and $\cG_1 = \{G \in \cG~|~ G \mbox{ contains a cycle}\}$. We aim to conduct the hypothesis
$$\Hb_0: G(\bTheta^*) \in \cG_0 \mbox{ vs } \Hb_1: G(\bTheta^*) \in \cG_1.$$

Similarly to the previous section, we also generate data under both the null and alternative hypotheses. We use the chain graph with adjacency matrix $\Ab_{\text{Chain}(d)}$ as a forest and construct the loopy graphs by adding edges to the chain graph. More specifically, we construct the graphs with adjacency matrices 
\[
   \Ab_m = \Ab_{\text{Chain}(d)} + \Eb_{m}, \text{ for } m = 3, 4, \ldots, 10,
\] 
where $\Eb_{m}$ is the adjacency matrix of the graph $(\overline V, \{(1,m)\})$.
The graphs we construct are illustrated in Fig \ref{fig:loop-sim}. Under the null hypothesis, we generate $n$ i.i.d. samples $\bX_1, \ldots, \bX_n \sim N({\bf 0}, (\bTheta^*(\theta))^{-1})$ where $\bTheta^*(\theta) = \Ib_d + \theta \Ab_{\text{Chain}(d)}$ and repeat the simulation for $N = 200$ times. Under the alternative, for each repetition, we randomly select $M \sim \text{Unif}(\{3, 4, \ldots, 10\})$, set the precision matrix as $\bTheta_M^*(\theta) = \Ib_d + \theta \Ab_{M}$ and generate i.i.d. samples $\bX_1, \ldots, \bX_n \sim N({\bf 0}, (\bTheta_M^*(\theta))^{-1})$. We also repeat this procedure for $N = 200$ times to calculate the type II error.

 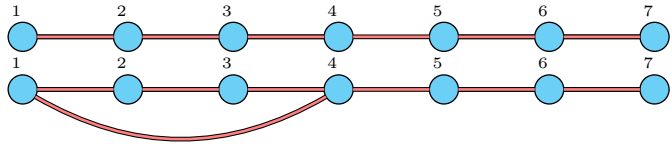
\begin{figure}[H]
\centering
 \begin{tikzpicture}[scale=.7]
\SetVertexNormal[Shape      = circle,
                  FillColor  = cyan!50,
                  MinSize = 11pt,
                  InnerSep=0pt,
                  LineWidth = .5pt]
   \SetVertexNoLabel
   \tikzset{LabelStyle/.style = {below, fill = white, text = black, fill opacity=0, text opacity = 1}}
   \tikzset{EdgeStyle/.style= {thin,
                                double          = red!50,
                                double distance = 1pt}}
    \begin{scope}\grPath[prefix=a,RA=2]{4}\end{scope}
                                                     \node[above left] at (a0.+60) {\tiny 1};
                                                 \node[above left] at (a1.+60) {\tiny 2};
                                                 \node[above left] at (a2.+60) {\tiny 3};
                                                 \node[above left] at (a3.+60) {\tiny 4};
         \Edge(a0)(a1)
                  \Edge(a1)(a2)
                  \Edge(a2)(a3)
     \begin{scope}[shift={(8,0)}]\grPath[prefix=b,RA=2]{3}\end{scope}
     \Edge(b1)(b2)
          \Edge(b0)(b1)
                    \Edge(a3)(b0)
                                                                          \node[above left] at (b0.+60) {\tiny 5};
                                                 \node[above left] at (b1.+60) {\tiny 6};
                                                 \node[above left] at (b2.+60) {\tiny 7};
								             \tikzset{EdgeStyle/.style= {thin,
                                double          = red!50,
                                double distance = 1pt}}
        \begin{scope}[shift={(0,1)}]\grPath[prefix=c,RA=2]{4}\end{scope}
         \Edge(c0)(c1)
                  \Edge(c1)(c2)
                  \Edge(c2)(c3)
                                                 \node[above left] at (c0.+60) {\tiny 1};
                                                 \node[above left] at (c1.+60) {\tiny 2};
                                                 \node[above left] at (c2.+60) {\tiny 3};
                                                 \node[above left] at (c3.+60) {\tiny 4};
     \begin{scope}[shift={(8,1)}]\grPath[prefix=d,RA=2]{3}\end{scope}
          \Edge(c3)(d0)
     \Edge(d1)(d2)
          \Edge(d0)(d1)
                                                                     \node[above left] at (d0.+60) {\tiny 5};
                                                 \node[above left] at (d1.+60) {\tiny 6};
                                                 \node[above left] at (d2.+60) {\tiny 7};
                                        								\tikzset{EdgeStyle/.append style = { thin, bend right}}
								          \Edge(a0)(a3)
\end{tikzpicture}
\vspace{-1em}
\caption{ $G(\Ab)$ and $G(\Ab_4)$ for $d = 7$.} \label{fig:loop-sim}
\vspace{-18pt}
\end{figure}

\begin{figure}[t]
\begin{subfigure}{.5\textwidth}
  \centering
  \includegraphics[width=.8\linewidth]{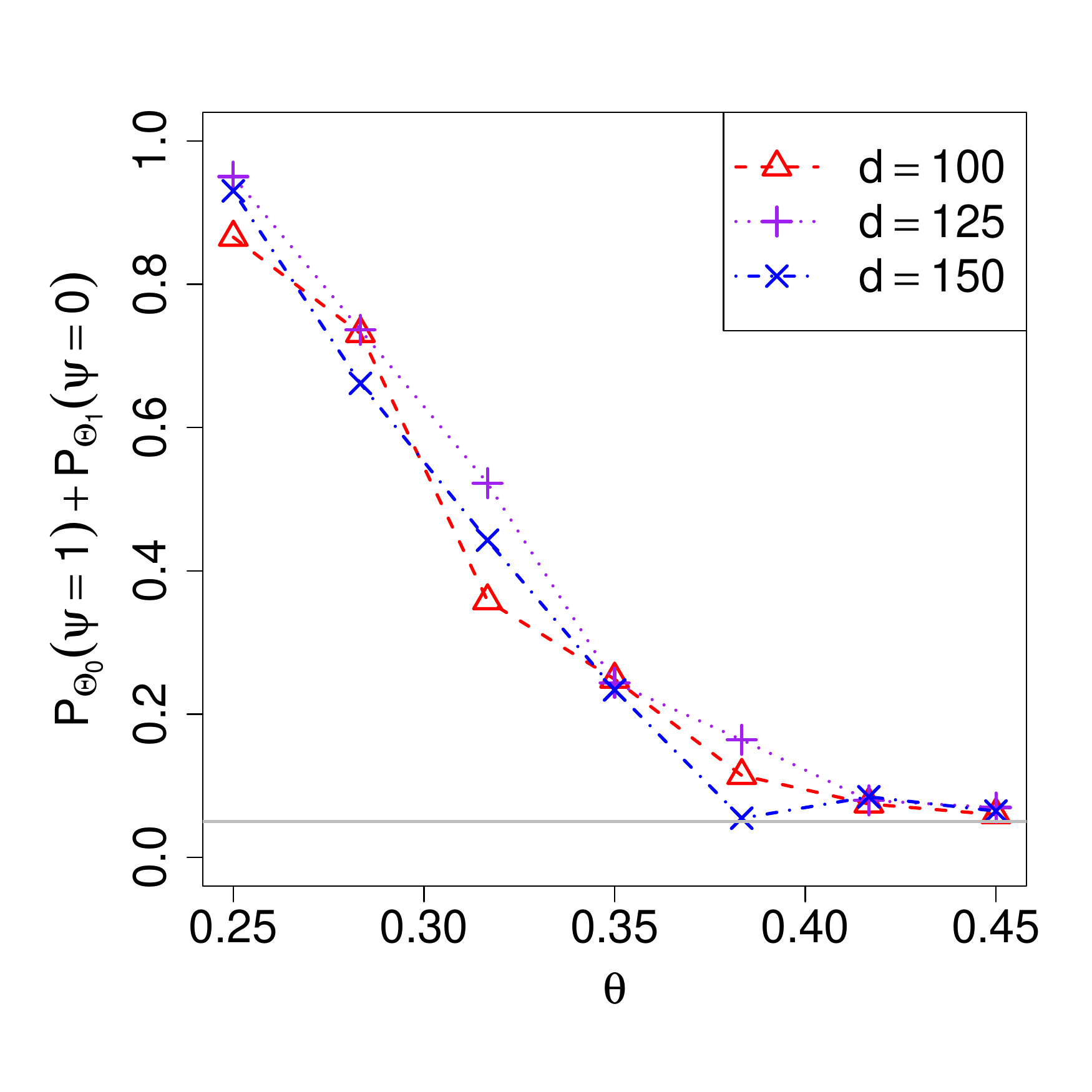}
  \vspace{-1em}
  \caption{$n = 400$}
  \label{fig:sfig1}
\end{subfigure}%
\begin{subfigure}{.5\textwidth}
  \centering
  \includegraphics[width=.8\linewidth]{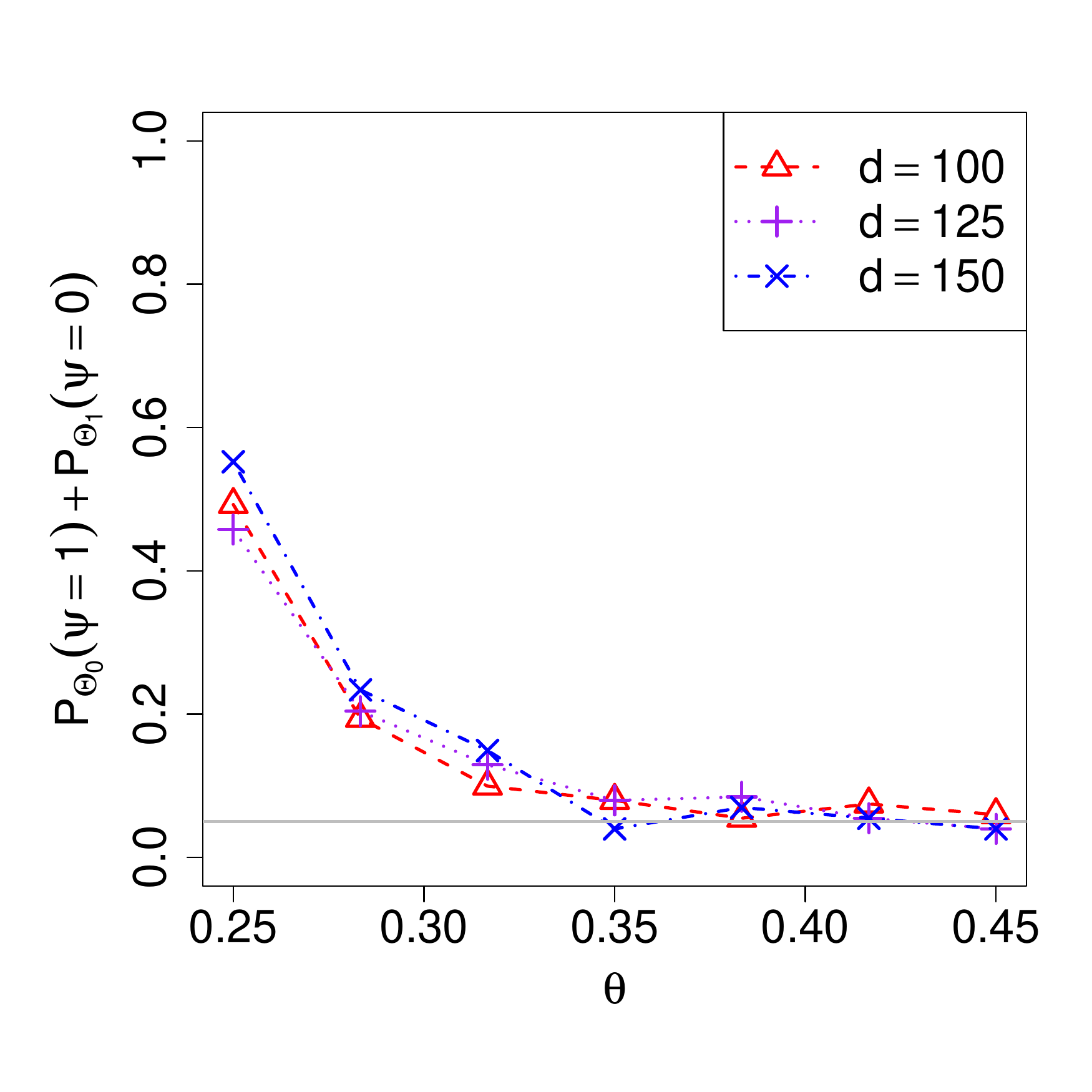}
  \vspace{-1em}
  \caption{$n = 600$}
  \label{fig:sfig3}
\end{subfigure}%
\vspace{-1em}
\caption{Type I + type II error control for cycles presence test with $\theta \in [0.25, 0.45]$.  The grey horizontal line is the significance level $\alpha = 0.05$.} \label{simu:figure:cycle}
\vspace{-18pt}
\end{figure}

As before the significance level is set to $\alpha = 0.05$, and the tuning parameter as $\lambda = 1.5\sqrt{\log d/n}$. The sample size varies in $n = \{400, 600\}$, the dimension $d = 100, 125$ and $150$ and the number of bootstrap samples is  $3,000$. 

In Fig \ref{simu:figure:cycle}, we plot the risk curves of the test $\psi$ based on Algorithm \ref{cycle:presence} for $\theta \in [0.25, 0.45]$. With the increase of signal strength or sample size, the risk diminishes to its target value which confirms the results of Section \ref{ham:cycle:det:sec}. 

\begin{table}[H] 
 \centering
\caption{Sizes for cycle presence testing $\theta \in [0.25, 0.45]$}\label{size:cycle:values}
\begin{tabular}{ccrrrrrrr}
  \hline
  \hline
  $n$ & \diagbox{$d$}{$\theta$} & 0.25 &  0.28 &  0.32 & 0.35 & 	0.38 &  0.42 & 	0.45 \\ 
  \hline
  \multirow{3}{*}{$ 400$} & $100$ &0.010 & 0.045 & 0.015 & 0.045 & 0.040 & 0.040 & 0.055 \\
  &$125$ &  0.025 & 0.005 & 0.045 & 0.030 & 0.070 & 0.070 & 0.060 \\ 
  &$150$ & 0.005 & 0.005 & 0.020 & 0.020 & 0.015 & 0.060 & 0.050 \\ 
  \hline
   \multirow{3}{*}{$600$} &100 & 0.045 & 0.040 & 0.060 & 0.065 & 0.045 & 0.060 & 0.040 \\ 
  &125 & 0.040 & 0.040 & 0.060 & 0.055 & 0.065 & 0.050 & 0.035 \\ 
  &150 &  0.030 & 0.030 & 0.050 & 0.015 & 0.060 & 0.050 & 0.040 \\ 
   \hline
   \hline
\end{tabular}
\end{table}
In Table \ref{size:cycle:values} we report the type I errors of the tests for each of $6$ values of the signal strength $\theta$. The type I error remains well controlled and is closer to its the target level of $0.05$ when the signal strength becomes larger as Corollary \ref{cycle:nosignal:str} predicts.

\subsection{ADHD-200 Data Analysis}\label{sec:ADHD}

In this section we analyze the difference in connectivity levels of the brain networks of subjects with and without attention deficit hyperactivity disorder (ADHD). We use the ADHD-200 brain-imaging dataset, which is made publicly available by the ADHD-200 consortium \citep{milham2012adhd}. The dataset  consists of $614$ subjects with $491$ controls and $195$ cases. 

For each subject a number between $76$ and $276$ fMRI scans are available. Each of the scans provides a brain imaging vector of $264$ voxels, which we use to obtain estimates of the brain network. To simplify the analysis we treat each of the repeated measurements as independent. In addition to the fMRI scans, limited clinical data such as age, gender, handedness etc is also available for each subject. The ages range from $7$ to $21$ years, with medians $11.75$ for the controls and $10.58$ for the cases. 

Using the ADHD-200 dataset, \cite{cai2015aberrant} explored the connectivity of the brain network, and concluded that the brain networks of ADHD cases are in general less connected compared to the brain networks of healthy subjects. Additionally, \cite{gelfand2003spatial, bartzokis2001age} studied how brain connectivity scales with age, and demonstrated that brain connections typically increase from junior to adult age. 

To explore potential changes in brain network connectivity in terms of age from junior to senior, we split the dataset into subjects whose age is less than $11$ and greater than $11$ years, where $11$ is chosen since it is close to the median ages of the cases and controls. Moreover, to further study the brain networks at different levels, we define the graph at level $\mu \ge 0$ for the precision matrix $\bTheta^*$ as 
\[
    G(\bTheta^*, \mu) = (\overline V, E(\bTheta^*, \mu)), \text{ where } E(\bTheta^*, \mu) = \big\{e \mid |\Theta^*_e| \ge \mu\big\}.
\]
When $\mu = 0$, $G(\bTheta^*, \mu)$ reduces to the typical graph  $G(\bTheta^*)$ defined in Section \ref{sec:notation}. We explore the connectivity of $G(\bTheta^*, \mu)$ for larger $\mu$'s, because the edges in the brain networks with stronger signal strengths are more informative to reflect the essential topological structure of brains. We therefore consider the hypothesis test 
\begin{align}\label{conn:comp:strength:test-1}
\Hb_0: G(\bTheta^*, \mu) \text{ disconnected vs }\Hb_1: G(\bTheta^*, \mu) \text{ connected}.
\end{align}
	We can easily modify Algorithm \ref{al:conn} to conduct the test above. In fact, we only need to change the output $\psi^B_{\alpha, \widehat T}(\cD_2)$ in Algorithm \ref{al:conn} to $\psi^B_{\alpha, \widehat T}(\cD_2, \mu)$  obtained from Algorithm \ref{step:down-mu}. 
\begin{algorithm}[t]
\normalsize
\caption{Multiple Edge Testing}\label{step:down-mu}
\begin{algorithmic}
\STATE Initialize $\widehat{E^{\rm n}} \leftarrow \hat T$;
\REPEAT
\STATE Reject $R \leftarrow \{e \in \widehat{E^{\rm n}}: \sqrt{n}|\tilde \Theta_{e}| \geq \sqrt{n} \mu + c_{1 - \alpha, \widehat{E^{\rm n}}}\}$
\STATE{Update $\widehat{E^{\rm n}} \leftarrow \widehat{E^{\rm n}} \setminus R$}
\UNTIL{$R = \varnothing$ or $\widehat{E^{\rm n}} = \varnothing$}
\RETURN{$\hat{E^{\rm nc}} \leftarrow \hat T \setminus \widehat{E^{\rm n}}$}
\end{algorithmic}
\end{algorithm}

Comparing Algorithm \ref{step:down-mu} with the original Algorithm \ref{step:down}, the only difference is that we change the first step in while loop from $R \leftarrow \{e \in \widehat{E^{\rm n}}~|~  \sqrt{n}|\tilde{\Theta}_{e} |\geq  c_{1 - \alpha, \widehat{E^{\rm n}}}\}$ to $R \leftarrow \{e \in \widehat{E^{\rm n}}~|~  \sqrt{n}|\tilde{\Theta}_{e} |\geq \sqrt{n} \mu + c_{1 - \alpha, \widehat{E^{\rm n}}}\}$. Following the same proof as Corollary \ref{conn:test:prop} , we can show that $\psi^B_{\alpha, \widehat T}(\cD_2, \mu) $ is a valid test for testing \eqref{conn:comp:strength:test-1}. A detailed proof can be found in Corollary \ref{conn:test:prop-mu} in the supplement. We apply the test $\psi^B_{\alpha, \widehat T}(\cD_2, \mu) $ to all $4$ groups (junior controls, senior controls, junior cases and senior cases) of subjects results 
for $\mu \in [0.16, 0.3]$. We denote $\hat T(\mu) = \hat{E^{\rm nc}}$ where $\hat{E^{\rm nc}}$ is defined in Algorithm \ref{step:down-mu}.

If $\hat T(\mu)$ is connected, since $\hat T(\mu)$ is a tree, we can see that $\psi^B_{\alpha, \widehat T}(\cD_2, \mu) = 1$ and $\Hb_0$ in \eqref{conn:comp:strength:test-1} is rejected. For robustness the data is shuffled $100$ times before splitting it into two halves and the number of connected components reported is averaged over these $100$ data shuffles. 
Fig \ref{adhd:conn:fig} demonstrates the number of connected components for $\hat T(\mu)$ when $\mu$ varies in the range $[0.16, 0.3]$. We can see that the number of connected components is always larger than 1 and thus we do not reject $\Hb_0$ in \eqref{conn:comp:strength:test-1} for all $\mu \in [0.16, 0.3]$.

\begin{figure}[t]
  \centering
  \includegraphics[width=.6\linewidth]{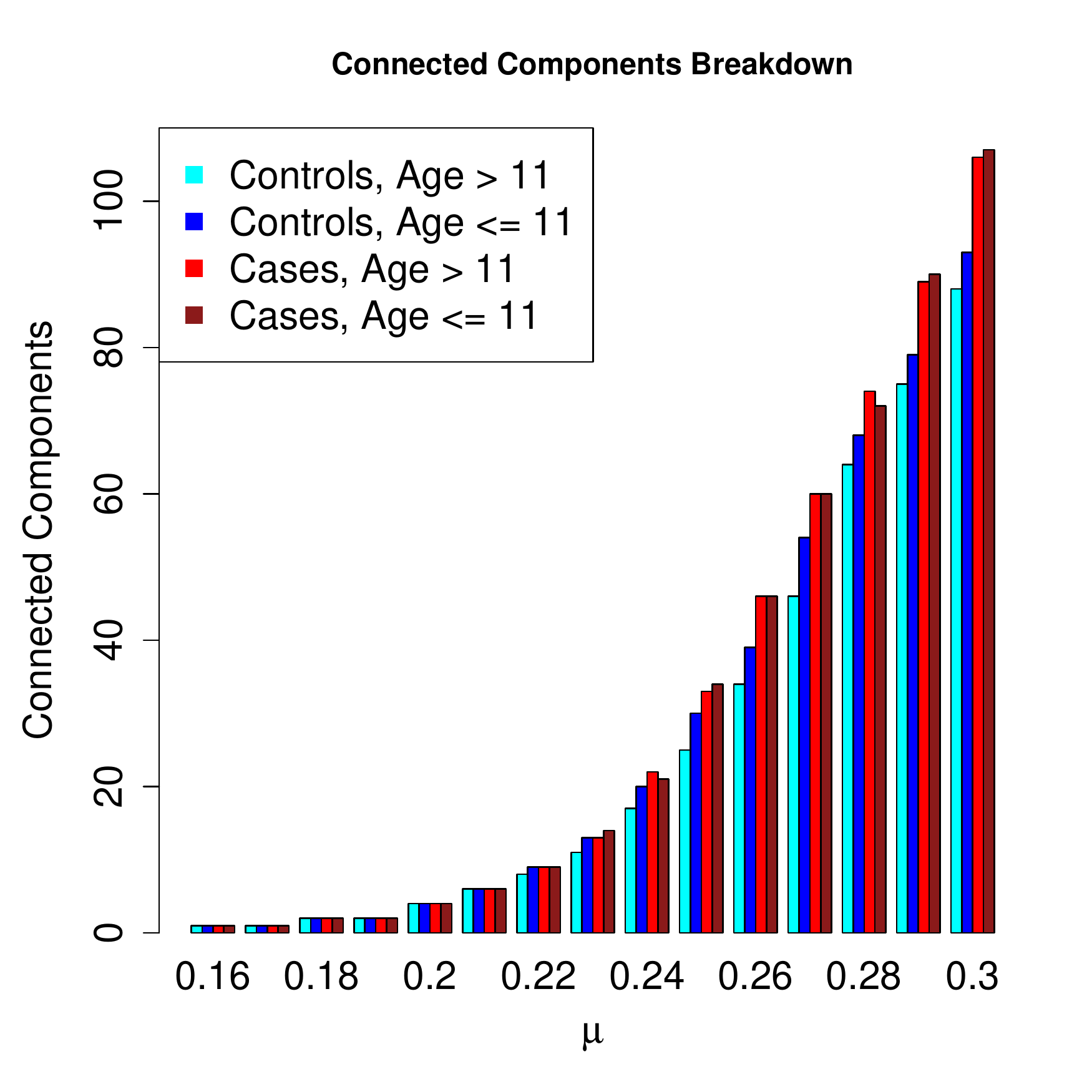}
  \vspace{-2em}
  \caption{Number of connected components of the graphs when ranging the parameter $\mu$ is in the range $[0.16, 0.3]$; We report averaged numbers over $100$ data split tests.} \label{adhd:conn:fig}\vspace{-18pt}
\end{figure}

Moreover, we observe in Fig \ref{adhd:conn:fig} that $\hat T(\mu)$ has more connected components for the controls than the cases. 
 These results confirm the findings of \cite{cai2015aberrant} that the brain network in ADHD patients are less connected compared to the brain network in control subjects at the same age. What is more, Fig \ref{adhd:conn:fig} also shows that the number of connected components tends to be less in controls whose age ranges within $11$ to $21$ years as compared to controls who are younger than $11$ years of age, which agrees with the results of \cite{gelfand2003spatial, bartzokis2001age}. Interestingly, there exists less of a difference in the number of components in the brain network of ADHD patients in the two age groups. This observation suggests that subjects with ADHD have a slower rate of forming strong brain connectivity over time. In addition to the the connected components analysis, in Fig \ref{adhd:brain:plots} we plot the largest connected subgraphs of $\hat T(\mu)$ for $\mu = 0.25$ for the controls of age at least $11$ years and the top two largest connected subgraphs for the cases. The value $\mu = 0.25$ is chosen as Fig \ref{adhd:conn:fig} reveals that at $\mu = 0.25$ there exists a clear separation in the number of connected components between cases and controls. 

\begin{figure}[t]
\begin{subfigure}{.45\textwidth}
  \centering
  \includegraphics[width=.6\linewidth]{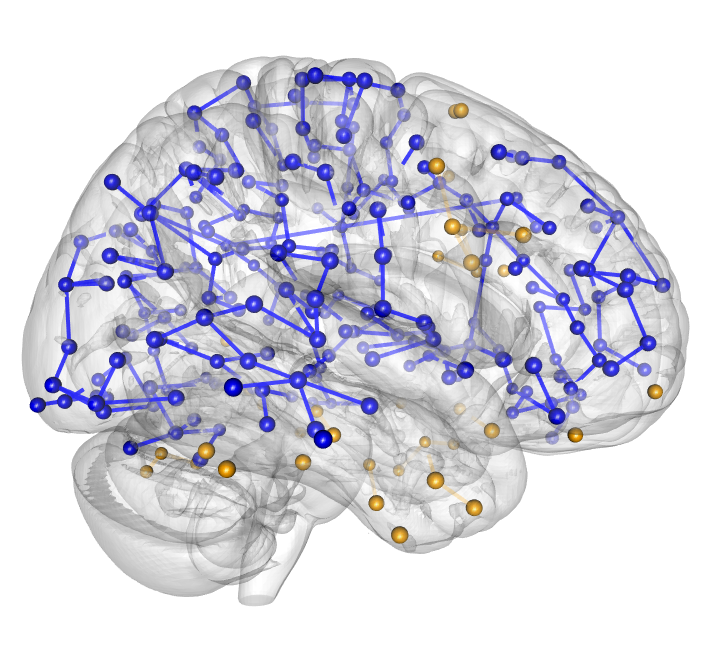}
  \vspace{-1em}
  \caption{Controls $\geq 11$ years}
  \label{fig:sfig2}
\end{subfigure}
\begin{subfigure}{.45\textwidth}
  \centering
  \includegraphics[width=.6\linewidth]{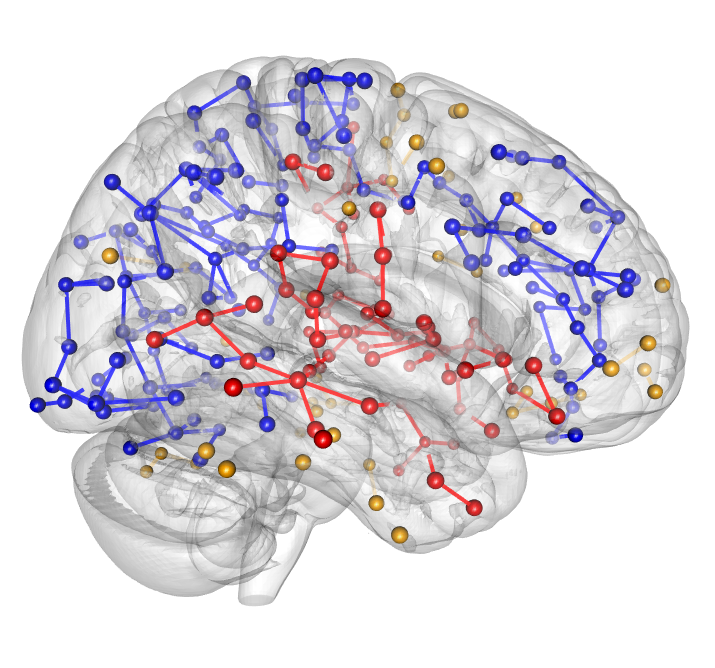}
  \vspace{-1em}
  \caption{Cases $\geq 11$ years}
  \label{fig:sfig2}
\end{subfigure}
\vspace{-1em}
\caption{Connected components for age groups $\geq$ 11 years. Components consisting of at least $6$ nodes are displayed in different colors (blue and red), while nodes falling in smaller components are grouped together and plotted in orange.} \label{adhd:brain:plots}\vspace{-18pt}
\end{figure}

Remarkably, there are two major connected components in the ADHD cases while there is only one major connected component in the controls. In addition, the second major component in the ADHD patients, which is automatically identified by our algorithm, bares close similarity to the salience network (SN) discussed in \cite{cai2015aberrant}. Indeed, \cite{cai2015aberrant} studied the connectivity of the SN to two other subgraphs in the brain networks and showed the connectivity between the SN and the other subgraphs is weaker in ADHD cases compared to the connectivity in the controls. This is also shown in Fig  \ref{adhd:brain:plots}.

\section{Edge Deletion Divider Extension, Proof of Proposition 2.1 and SAP example}\label{edge:removal:NAS:extension}

In this section, we prove Proposition \ref{smart:bound} and we provide the example of self-avoiding path. Also, we show the deletion-edge divider result on the lower bound parallel to Theorem~\ref{suff:cond:lower:bound} and give other examples. 

\begin{proof}[Proof of Proposition \ref{smart:bound}] To prove this result we construct the set $\cC$ explicitly in an iterative manner. For a vertex set $V \subset \overline V$ define
$$
\cN_l(V) := \{v \in \overline V ~|~ d_{G_0}(v, w) \leq l ~ \forall w \in V\}.
$$
Start with $\cC = \varnothing$. Take any vertex $v_1 \in \overline V$, find $w_1 \in W_v$, and set $\cC \leftarrow \cC \cup \{(v_1, w_1)\}$. We now bound the cardinality of $\cN_{l}(\{v_1, w_1\})$. Suppose that the maximum degree of $G_0$ is at most $D$. We have $|\cN_{l}(\{v_1, w_1\})| \leq 2 + 2D^{l}$. Update the graph $G_0$ by deleting all vertices and edges associated with the set $|\cN_{l}(\{v_1, w_1\})|$. The new graph has at least $d - (2 + 2D^{l})$ vertices. If $c d^\gamma - (2 + 2D^{l}) > 0$ then there exists a pair of vertices $v_2, w_2$ so that $(\overline V, E_0 \cup \{(v_2, w_2)\}) \in \cG_1$. Set $\cC \leftarrow \cC \cup \{(v_2, w_2)\}$. Iterating this procedure $k$ times while maintaining $c d^\gamma - k(2 + 2D^{l}) > 0$ gives a set $\cC$ of cardinality $k + 1$, each two edges of which are at least $l$ apart. Since $2 + 2D^{l}\leq 4D^{l}$ the latter inequality is implied if $\log (c d^\gamma) -\log 4 \geq \log k + l \log D$. Selecting $l = \lceil \log k \rceil + 1 \geq \lceil \log(k  + 1)\rceil$ yields that we can select $k$ as large as $\log k \leq \frac{\log (c d^\gamma) -\log 4 - 2\log D}{1 + \log D}$, which completes the proof.
\end{proof}

\begin{example}[Self-Avoiding Path (SAP) of Length $\fm$ vs $\fm+1$, $\fm < \sqrt{d}$] \label{self:avoid:path:example} 
In this example we test the existence of a SAP of length $\leq \fm$ vs SAP of length $\geq \fm+1$, i.e., $\cG_0 = \{G \in \cG~|~ \mbox{All SAPs have length} \leq \fm\}$  vs $\cG_1 = \{G \in \cG~|~ \exists \mbox{ SAP of length } \fm+1\}$. We assume that $\fm < \sqrt{d}$ in order to be able to construct sufficiently large divider set which yields tight bounds. The case $\fm \geq \sqrt{d}$ is treated in Example \ref{self:avoid:path:example:new}. Without loss of generality we further suppose that $\fm+2$ divides $d$ (at the sake of assigning some isolated vertices in the graph). Construct the following graph in $\cG_0$: $G_0 = (\overline V, E_0)$, where 
$$E_0 := \{(j, j + 1) ~|~ \mbox{unless }  (\fm+2)  \text{~divides~} j \mbox{ or } (\fm + 2)  \text{~divides~} (j + 1) \},$$
and hence $|E_0| = d\fm/(\fm + 2) $ (see Fig \ref{self:avoid:path:ex1}). Consider the class $\cC := \{(j (\fm + 2) - 1, j(\fm + 2))_{j \in [d/(\fm+2)]} \}.$ The set $\cC$ is a divider since adding any edge from $\cC$ to $G_0$ results in a graph with a longer self-avoiding path. Furthermore the maximum degree of $G_0$ is $2$. Notice that the set $\cC$ is its own a ($\log |\cC|$)-packing, since each two edges $e, e' \in \cC$ satisfy $d_{G_0}(e,e') = \infty$. Hence $M(|\cC|, d_{G_0}, \log|\cC|) = \log |\cC| \asymp \log d$. An example matching this restriction is provided in Section \ref{detailed:upperbound:algos}.
We conclude that if $\theta \leq \kappa \sqrt{\log d/n} \wedge \frac{1 - C^{-1}}{4\sqrt{2}}$ we cannot differentiate the null from the alternative hypothesis.

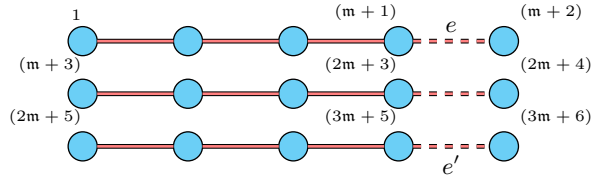
\begin{figure}[t] 
\centering
\begin{tikzpicture}[scale=.7]
 \SetVertexNormal[Shape      = circle,
                  FillColor  = cyan!50,
                  MinSize = 11pt,
                  InnerSep=0pt,
                  LineWidth  = .5pt]
  \SetVertexNoLabel
   \tikzset{EdgeStyle/.style= {thin,
                                double          = red!50,
                                double distance = 1pt}}
    \begin{scope}\grPath[prefix=a,RA=2]{4}\end{scope}
                                                               \node[above left] at (a0.+60) {\tiny $(2\fm+5)$};
  			                    \node[above left] at (a3.+60) {\tiny $(3\fm+5)$};
     \begin{scope}[shift={(8,0)}]\grEmptyPath[prefix=b,RA=2]{1}\end{scope}
                 			                    \node[above right] at (b0.+60) {\tiny $(3\fm+6)$};
         \begin{scope}[shift={(0,1)}]\grPath[prefix=d,RA=2]{4}\end{scope}
                                                      \node[above left] at (d0.+60) {\tiny $(\fm+3)$};
  			                    \node[above left] at (d3.+60) {\tiny $(2\fm+3)$};
     \begin{scope}[shift={(8,1)}]\grEmptyPath[prefix=e,RA=2]{1}\end{scope}
            			                    \node[above right] at (e0.+60) {\tiny $(2\fm+4)$};
              \begin{scope}[shift={(0,2)}]\grPath[prefix=f,RA=2]{4}\end{scope}
                                             \node[above left] at (f0.+60) {\tiny 1};
  			                    \node[above left] at (f3.+60) {\tiny $(\fm+1)$};
     \begin{scope}[shift={(8,2)}]\grEmptyPath[prefix=g,RA=2]{1}\end{scope}
       			                    \node[above right] at (g0.+60) {\tiny $(\fm+2)$};
\tikzset{EdgeStyle/.style= {dashed, thin,
                                double          = red!50,
                                double distance = 1pt}}
                                   \tikzset{LabelStyle/.style = {below, fill = white, text = black, fill opacity=0, text opacity = 1}}
    \Edge[label=$e'$](a3)(b0)
        \Edge(d3)(e0)
           \tikzset{LabelStyle/.style = {above, fill = white, text = black, fill opacity=0, text opacity = 1}}
    \Edge[label=$e\protect\vphantom{'}$](f3)(g0)
\end{tikzpicture}
\vspace{-1em}
\caption{The null base $G_0$ and the divider $\cC$ (dashed) with $d_{G_0}(e, e') = \infty$, $d = 15$, $\fm = 3$.}\label{self:avoid:path:ex1}
\end{figure}
\end{example}

Though Theorem \ref{suff:cond:lower:bound} delivers valid lower bounds, the obtained bounds may not always be tight. For example, recall the construction of Example \ref{self:avoid:path:example}. When $\fm = d - 2$ the divider contains only a single edge, and therefore (\ref{scaling:conditions}) is not sharp. Sometimes tighter bounds can be obtained via constructions which delete edges from the alternative graph instead of adding edges to the graph under null.  In this subsection we extend the divider definition to allow for deleting edges from the alternative hypothesis, which is sometimes more convenient and gives more informative lower bounds as we demonstrate below.

\begin{definition}[Single-Edge Deletion Null-Alternative Divider]\label{single:edge:del:null:alt:sep} Let $G_1 = (\overline V, E_1)$ be an alternative graph. An edge set $\cC$ is called a (single edge) deletion divider with alternative base $G_1$, if for all $e \in \cC$ the graphs $G_{\setminus e} = (\overline V, E_1 \setminus \{e\})  \in \cG_0$.
\end{definition}

In parallel to Theorem \ref{suff:cond:lower:bound} we have the following result using the deletion divider.

\begin{theorem} \label{suff:cond:lower:bound:deletion:NAS}  
Let $G_1 \in \cG_1$ be a graph whose maximum degree bounded by a fixed integer $D$, and let $\cC$ be a single-edge deletion divider with alternative base $G_1$. Suppose that $ n \geq (\log |\cC|)/c_0$ for some $c_0 > 0$ and that:
\begin{align}\label{scaling:conditions:edge:del}
\theta \leq \kappa \sqrt{\frac{M(\cC, d_{G_1}, \log |\cC|)}{n}} \wedge \frac{1 - C^{-1}}{\sqrt{2}D}.
\end{align}
Then, if $M(\cC, d_{G_1}, \log |\cC|) \rightarrow \infty$ as $n \rightarrow \infty$, for a sufficiently small $\kappa$  (depending on $D, C, L,c_0$) we have $\liminf_{n \rightarrow \infty} \gamma(\cS_0(\theta, s), \cS_1(\theta, s)) = 1.$
\end{theorem}

To illustrate the usefulness of Theorem \ref{suff:cond:lower:bound:deletion:NAS}  for  edge deletion divider, we consider the examples below.

\begin{example}[$\fm+1$ vs $\fm$ Connected Components, $\fm < \sqrt{d}$] \label{conn:comp:new} In this example we consider a complementary setting to Example \ref{conn:comp}, and test whether the graph contains $\fm + 1$ connected components vs $\fm$ connected components for $\fm < \sqrt{d}$. Recall the decomposition $\cG_0 = \{G \in \cG~|~ G \mbox{ has } \geq \fm + 1 \mbox{ conn. components} \}$ vs $\cG_1 = \{G \in \cG~|~ G \mbox{ has} \leq \fm \mbox{ conn. components}\}$. Construct an alternative base graph $G_1 = (\overline V, E_1)$: $ E_1:= \{(j,j+1)_{j = 1}^{d - \fm}\},$ and set $\cC := E_1 = \{(j,j+1)_{j = 1}^{d - \fm}\} \mbox{ (see Fig \ref{path:graph:conn:comp:ex1} for a visualization)}.$
Since removing any edge from $\cC$ results in increasing the number of connected components by $1$, the set $\cC$ is a single edge deletion divider with an alternative base graph $G_1$. Moreover, notice that the maximum degree of $G_1$ is $2$. 

The entropy $M(\cC, d_{G_1}, \log |\cC|)$ satisfies $M(\cC, d_{G_1}, \log |\cC|) \asymp  \log \frac{|\cC|}{\log |\cC|} \asymp \log |\cC| \asymp \log d$. Hence by Theorem \ref{suff:cond:lower:bound:deletion:NAS} we conclude that testing connectivity is impossible when $\theta < \kappa \sqrt{\log d/n}\wedge \frac{1 - C^{-1}}{2\sqrt{2}}$. 

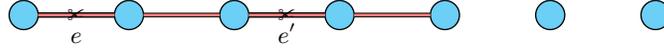
\begin{figure}[H] 
\centering
\begin{tikzpicture}[scale=.7]
\SetVertexNormal[Shape      = circle,
                  FillColor  = cyan!50,
                  MinSize = 11pt,
                  InnerSep=0pt,
                  LineWidth = .5pt]
   \SetVertexNoLabel
   \tikzset{LabelStyle/.style = {below, fill = white, text = black, fill opacity=0, text opacity = 1}}

   \tikzset{EdgeStyle/.style= {thin,
                                double          = red!50,
                                double distance = 1pt}}
    \begin{scope}\grPath[,RA=2]{5}\end{scope}
     \begin{scope}[shift={(10,0)}]\grEmptyPath[prefix=b,RA=2]{2}\end{scope}
    \Edge[label=$e\protect\vphantom{'}$](a0)(a1)
     \Edge[label=$e'$](a2)(a3)
          \tikzset{LabelStyle/.style = {fill = white, text = black, fill opacity=0, text opacity = 1}}
          \Edge[label=\Cutright](a0)(a1)
     \Edge[label=\Cutright](a2)(a3)
\end{tikzpicture}
\vspace{-1em}
\caption{The graph $G_1$ with $d - \fm$ edges and $\fm-1$ isolated nodes, d = 7.}\label{path:graph:conn:comp:ex1}
\end{figure}
\vspace{-18pt}
\end{example} 

\begin{example}[Self-Avoiding Path (SAP) of Length $\fm$ vs $\fm+1$, $\fm \geq \sqrt{d}$] \label{self:avoid:path:example:new} 
In parallel to Example \ref{self:avoid:path:example} we now consider testing that SAP $\leq \fm$ vs SAP of length $\geq \fm+1$ when $\fm \geq \sqrt{d}$. Define $\cG_0 = \{G \in \cG~|~ \forall \mbox{ SAP have length} \leq \fm\}$  vs $\cG_1 = \{G \in \cG~|~ \exists \mbox{ SAP of length } \fm+1\}$. Construct the alternative graph $G_1 = (\overline V, E_1)$, where $E_1 := \{(j, j + 1)_{j =1 }^{\fm + 1}\}).$ Define the set  $\cC := E_1 = \{(j, j + 1)_{j = 1}^{\fm+1}\},$
and note that removing any edge from $\cC$ results in a graph from the null. Hence $\cC$ is a deletion divider with base $G_1$. Moreover, the maximum degree of $G_1$ is $2$.  
Similarly to Example \ref{conn:comp:new} it follows that $M(\cC, d_{G_1}, \log |\cC|) \asymp \log |\cC| \asymp \log d$. An application of Theorem \ref{suff:cond:lower:bound:deletion:NAS} implies that for values of $\theta < \kappa\sqrt{\log d/n} \wedge \frac{1 - C^{-1} }{2 \sqrt{2}}$ testing the length of SAP is impossible.

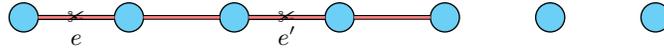
\begin{figure}[H] 
\centering
\begin{tikzpicture}[scale=.7]
\SetVertexNormal[Shape      = circle,
                  FillColor  = cyan!50,
                  MinSize = 11pt,
                  InnerSep=0pt,
                  LineWidth = .5pt]
   \SetVertexNoLabel
   \tikzset{LabelStyle/.style = {below, fill = white, text = black, fill opacity=0, text opacity = 1}}
   \tikzset{EdgeStyle/.style= {thin,
                                double          = red!50,
                                double distance = 1pt}}
    \begin{scope}\grPath[,RA=2]{5}\end{scope}
     \begin{scope}[shift={(10,0)}]\grEmptyPath[prefix=b,RA=2]{2}\end{scope}
    \Edge[label=$e\protect\vphantom{'}$](a0)(a1)
     \Edge[label=$e'$](a2)(a3)
          \tikzset{LabelStyle/.style = {fill = white, text = black, fill opacity=0, text opacity = 1}}
          \Edge[label=\Cutright](a0)(a1)
     \Edge[label=\Cutright](a2)(a3)
\end{tikzpicture}
\vspace{-1em}
\caption{The graph $G_1$ with $\fm+1$ edges and two edges $e,e' \in \cC$ with $d_{G_1}(e,e') = 1$.}\label{self:avoid:path:ex2}
\end{figure}
\vspace{-18pt}
\end{example}

\section{Details for Combinatorial Testing Algorithms} \label{detailed:upperbound:algos}

In this section we give full details on the algorithms of Section \ref{struct:test:section}.

\subsection{Connected Components Testing}

Connected component test is a more general test compared to the connectivity test. For  $\fm \in [d - 1]$ let $\Hb_0: \mbox{\# connected components} \geq \fm + 1$ vs $\Hb_1: \mbox{\# connected components} \leq \fm$. Testing connectivity is a special case when $\fm = 1$. 

Recall that a sub-graph $F$ of $G$, i.e., $E(F) \subset E(G)$ and $V(F) = V(G) = \overline V$, is called a \textit{spanning forest} of $G$, if $F$ contains no cycles, adding any edge $e \in E(G) \setminus E(F)$ to $E(F)$ creates a cycle, and $|E(F)|$ is maximal. This definition extends naturally to graphs with positive weights on their edges. For a graph $G$ let $\operatorname{SF}(G)$ denote the class of spanning forests of $G$. Define the set of precision matrices $\cS\cF(\theta) := \{\bTheta~|~ \exists~ F \in \operatorname{SF}(G(\bTheta)) \mbox{ such~that } \min_{e \in F} |\Theta_e| \geq \theta\}$. For a fixed $\fm \in [d-1]$ definitions (\ref{generic:S0:def:new}) and (\ref{generic:S1:def:new}) reduce to:
\begin{align*}
\cS_0(s) & := \{\bTheta \in \cM(s) ~|~ G(\bTheta) \in  \cG_{\fm + 1}\},\\
\cS_1(\theta,s) & := \{\bTheta \in \cS\cF(\theta) \cap \cM(s) ~|~G(\bTheta) \in \cup_{j \leq \fm}\cG_{j}\},
\end{align*}
where $\cG_{j} := \{G \in \cG~|~ G \mbox{ has exactly $j$ connected components}\}$ for all $j \in \NN$.
\begin{algorithm}[H]
\normalsize
\caption{Connected Components Test}\label{al:conn:comp}
\begin{algorithmic}
\STATE \textbf{Input:} $\cD=\{\bX_{i}\}_{i = 1}^n$, level $0 < \alpha < 1$.
\STATE Split the data into  $\cD_1 = \{\bX_i\}_{i = 1}^{\lfloor n/2 \rfloor}, \cD_2 = \{\bX_i\}_{i = \lfloor n/2 \rfloor + 1}^{n}$
\STATE Using $\cD_1$ obtain an estimate $\hat \bTheta^{(1)}$ satisfying (\ref{maxnorm:prec:ass})
\STATE Find an MSF $\hat F$ with $\fm$ connected components with weights $|\hat \Theta_e^{(1)}|$, let $\widehat E = E(\hat F)$.
\STATE \textbf{Output:} $\psi^B_{\alpha, \widehat E}(\cD_2)$.
\end{algorithmic}
\end{algorithm}
We have the following result regarding the asymptotic performance of Algorithm \ref{al:conn:comp}
\begin{corollary}\label{conn:comp:test:prop} Let $\theta = (2K + C' C^4)\sqrt{\log d/\lfloor n/2\rfloor}$ for an absolute constant $C' > 0$ and  assume that (\ref{bootstrap:rates}) holds. Then for any fixed level $\alpha$, the test $\psi^B_{\alpha, \widehat E}(\cD_2)$ from the output of Algorithm \ref{al:conn:comp} satisfies:
\begin{align*}
\limsup_{n \rightarrow \infty} \sup_{\bTheta^* \in \cS_0(s)} \PP(\mbox{reject } \Hb_0) \leq \alpha, ~~~~ \liminf_{n \rightarrow \infty} \inf_{\bTheta^* \in \cS_{1}(\theta, s)} \PP(\mbox{reject } \Hb_0) = 1.
\end{align*}
\end{corollary}

\subsection{Cycle Testing}\label{ham:cycle:det:sec}

In this section we outline an algorithm testing whether the graph is a forest. Recall the sub-decomposition $\cG_0 = \{G \in \cG: G \mbox{ contains no cycles}\}$ and $\cG_1 = \{G \in \cG: G \mbox{ contains a cycle}\}$. The parameter spaces (\ref{generic:S0:def:new}) and (\ref{generic:S1:def:new}) reduce to
\begin{align*}
\cS_0(s) & := \{\bTheta \in \cM(s): G(\bTheta) \in \cG_0  \},\\
\cS_1(\theta,s) & := \{\bTheta \in \cM(s) : G(\bTheta) \in \cG_1, \exists \underset{{\tiny\mbox{cycle}}}{E} \!\!\subset \!\!E(G(\bTheta)), \min_{e \in E} |\Theta_e| \geq \theta\}.
\end{align*} 

The cycle presence test is compactly summarized in Algorithm \ref{cycle:presence}.
\begin{algorithm}[t]
\normalsize
\caption{Cycle Presence Tests}\label{cycle:presence}
\begin{algorithmic}
\STATE \textbf{Input:} $\cD=\{\bX_{i}\}_{i = 1}^n$, level $0 < \alpha < 1$.
\STATE Split the data into  $\cD_1 = \{\bX_i\}_{i = 1}^{\lfloor n/2\rfloor }, \cD_2 = \{\bX_i\}_{i = \lfloor n/2\rfloor + 1}^{n}$
\STATE Using $\cD_1$ obtain estimate $\hat \bTheta^{(1)}$  satisfying (\ref{maxnorm:prec:ass}); Sort edges according to weights $|\hat \Theta_e^{(1)}|$
\STATE Add edges $e$ from high to low $|\hat \Theta_e^{(1)}|$ values until a cycle is present.\footnotemark
\STATE Let $\widehat E$ be set of edges forming the cycle.
\STATE \textbf{Output:} $\psi^B_{\alpha, \widehat E}(\cD_2)$.
\end{algorithmic}
\end{algorithm}
\footnotetext{This action requires no more than $d$ iterations.}
The following result justifies the validity of  Algorithm \ref{cycle:presence}.

\begin{corollary} \label{cycle:nosignal:str} Let $\theta = (2K + C' C^4)\sqrt{\log d/\lfloor n/2\rfloor}$ for an absolute constant $C' > 0$ and  assume that (\ref{bootstrap:rates}) holds. Then for any fixed level $\alpha$, the test $\psi^B_{\alpha, \widehat E}(\cD_2)$ from the output of Algorithm \ref{cycle:presence} satisfies:
\begin{align*}
\limsup_{n \rightarrow \infty} \sup_{\bTheta^* \in \cS_0(s)} \PP(\mbox{reject } \Hb_0) \leq \alpha, ~~~~ \liminf_{n \rightarrow \infty} \inf_{\bTheta^* \in \cS_{1}(\theta, s)} \PP(\mbox{reject } \Hb_0) = 1.
\end{align*}
\end{corollary}

\subsection{Triangle-Free Graph Testing}\label{triangle:free:testing:sec}

In this section we discuss an algorithm for testing whether the graph is triangle-free. The corresponding sub-decomposition is $\cG_0 = \{G \in \cG~|~ G \mbox{ is triangle-free}\}$ and $\cG_1 = \{G \in \cG~|~ \exists \mbox{ $3$-clique subgraph of } G\}$. The parameter sets of precision matrices specified by (\ref{generic:S0:def:new}) and (\ref{generic:S1:def:new}) reduce to
\begin{align*}
\cS_0(s) & := \{\bTheta \in \cM(s)~|~ G(\bTheta) \in \cG_0\}, \\
\cS_1(\theta,s) & := \{\bTheta \in \cM(s) ~|~ G(\bTheta) \in \cG_1, \exists \underset{\tiny\mbox{triangle}}{E} \subset E(G(\bTheta)) \min_{e \in E} |\Theta_e| \geq \theta\}.
\end{align*}
We summarize the cycle test below

\begin{algorithm}[H]
\normalsize
\caption{Triangle-Free Test}\label{triangle:test}
\begin{algorithmic}
\STATE \textbf{Input:} $\cD=\{\bX_{i}\}_{i = 1}^n$, level $0 < \alpha < 1$.
\STATE Split the data into  $\cD_1 = \{\bX_i\}_{i = 1}^{\lfloor n/2\rfloor }, \cD_2 = \{\bX_i\}_{i = \lfloor n/2\rfloor + 1}^{n}$
\STATE Using $\cD_1$ obtain estimate $\hat \bTheta^{(1)}$  satisfying (\ref{maxnorm:prec:ass}); Sort edges according to weights $|\hat \Theta_e^{(1)}|$
\STATE Add edges $e$ from high to low $|\hat \Theta_e^{(1)}|$ values until a triangle emerges.
\STATE Let $\widehat E$ be the set of edges forming the triangle.
\STATE \textbf{Output:} $\psi^B_{\alpha, \widehat E}(\cD_2)$
\end{algorithmic}
\end{algorithm}

The following proposition characterizes the performance of $\psi^B_{\alpha, \widehat E}(\cD_2)$.

\begin{corollary}[Honest Triangle-Free Test] \label{triangle:test:prop} Let $\theta = (2K + C' C^4)\sqrt{\log d/\lfloor n/2\rfloor}$ for an absolute constant $C' > 0$ and  assume that (\ref{bootstrap:rates}) holds. Then for any fixed level $\alpha$, the test $\psi^B_{\alpha, \widehat E}(\cD_2)$ from the output of Algorithm \ref{sap:length:test} satisfies:
\begin{align*}
\limsup_{n \rightarrow \infty} \sup_{\bTheta^* \in \cS_0(s)} \PP(\mbox{reject } \Hb_0) \leq \alpha, ~~~~ \liminf_{n \rightarrow \infty} \inf_{\bTheta^* \in \cS_{1}(\theta, s)} \PP(\mbox{reject } \Hb_0) = 1.
\end{align*}
\end{corollary}

\subsection{Self-Avoiding Path Length Testing}\label{sap:testing:sec}

In this section we summarize an algorithm for testing whether the longest SAP in the graph is of length $\leq \fm$ vs the longest SAP in the graph has length $\geq \fm + 1$. The sub-decomposition is $\cG_0 = \{G \in \cG ~|~ \text{All SAPs in } G \mbox{ have length } \leq \fm\}$ and $\cG_1 = \{G \in \cG ~|~ \exists \mbox{ a SAP of length } \fm + 1\}$. The parameter sets of precision matrices specified by (\ref{generic:S0:def:new}) and (\ref{generic:S1:def:new}) reduce to:
\begin{align*}
\cS_0(s) & := \{\bTheta \in \cM(s) ~|~ G(\bTheta) \in \cG_0\}, \\
\cS_1(\theta,s) & := \{\bTheta \in \cM(s) ~|~ G(\bTheta) \in \cG_1, \exists \underset{{\tiny \mbox{SAP of length }\scriptsize \fm+1}}{E} \!\!\subset \!\!E(G(\bTheta)), \min_{e \in E} |\Theta_e| \geq \theta\}.
\end{align*}
The SAP length test is summarized below:

\begin{algorithm}[H]
\caption{SAP Length Test}\label{sap:length:test}
\begin{algorithmic}
\normalsize
\STATE \textbf{Input:} $\cD=\{\bX_{i}\}_{i = 1}^n$, level $0 < \alpha < 1$.
\STATE Split the data into  $\cD_1 = \{\bX_i\}_{i = 1}^{\lfloor n/2\rfloor }, \cD_2 = \{\bX_i\}_{i = \lfloor n/2\rfloor + 1}^{n}$
\STATE Using $\cD_1$ to obtain estimate $\hat \bTheta^{(1)}$ satisfying (\ref{maxnorm:prec:ass}); 
\STATE Add edges $e$ from high to low $|\hat \Theta_e^{(1)}|$ values until a SAP of length $\fm + 1$ emerges.
\STATE Let $\widehat E$ be the set of edges forming the SAP of length $\fm + 1$.
\STATE \textbf{Output:} $\psi^B_{\alpha, \widehat E}(\cD_2)$
\end{algorithmic}
\end{algorithm}

Below we characterize the performance of the output of Algorithm \ref{sap:length:test} --- $\psi^B_{\alpha, \widehat E}(\cD_2)$.

\begin{corollary}[Honest SAP Test] \label{sap:test:prop} Let $\theta = (2K + C' C^4)\sqrt{\log d/\lfloor n/2\rfloor}$ for an absolute constant $C' > 0$ and  assume that (\ref{bootstrap:rates}) holds. Then for any fixed level $\alpha$, the test $\psi^B_{\alpha, \widehat E}(\cD_2)$ from the output of Algorithm \ref{sap:length:test} satisfies:
\begin{align*}
\limsup_{n \rightarrow \infty} \sup_{\bTheta^* \in \cS_0(s)} \PP(\mbox{reject } \Hb_0) \leq \alpha, ~~~~ \liminf_{n \rightarrow \infty} \inf_{\bTheta^* \in \cS_{1}(\theta, s)} \PP(\mbox{reject } \Hb_0) = 1.
\end{align*}
\end{corollary}

\subsection{Maximum Degree Test}\label{supp:max:deg:honest:test}

We propose a test matching lower bound in Example \ref{max:deg:1:edge}. Recall that the sub-decomposition is $\cG_0 = \{G~|~ d_{\max}(G) \leq s_0\}$ and $\cG_1 = \{G~|~ d_{\max}(G) \geq s_1\}$. The parameter sets (\ref{generic:S0:def:new}) and (\ref{generic:S1:def:new}) reduce to $\cS_0(s) = \{\bTheta \in \cM(s) ~|~ G(\bTheta) \in \cG_0\}$ and $\cS_1(\theta, s) = \{\bTheta \in \cM(s) ~|~ G(\bTheta) \in \cG_1, \exists \mbox{ vertex } k \in \overline V, \mbox{ and set } S \subset [d] \text{ such that } k \not\in S, |S| = s,\min_{j \in S} |\Theta_{jk}| \geq \theta\}$. Recall the definitions of $\cD_1, \cD_2, \hat \bTheta^{(1)}$ and $\hat \bTheta^{(2)}$ from Section \ref{struct:test:section}.

\begin{algorithm}[H]
\normalsize
\caption{Maximum Degree Test}\label{star:test}
\begin{algorithmic}
\STATE \textbf{Input:} $\cD=\{\bX_{i}\}_{i = 1}^n$, level $0 < \alpha < 1$.
\STATE Split the data into  $\cD_1 = \{\bX_i\}_{i = 1}^{\lfloor n/2\rfloor }, \cD_2 = \{\bX_i\}_{i = \lfloor n/2\rfloor + 1}^{n}$
\STATE Add edges $e$ from high to low $|\hat \Theta_e^{(1)}|$ values until a vertex with $(s_0 + 1)$ neighbors occurs.\STATE Let $\widehat E$ be the set of $s$ edges connecting the vertex to its $(s_0 + 1)$ neighbours.
\vspace{-1em}
\STATE \textbf{Output:} $\psi^B_{\alpha, \widehat E}(\cD_2)$.
\end{algorithmic}
\end{algorithm}

Following the formulation of Algorithm \ref{star:test} we characterize the performance of $\psi^B_{\alpha, \widehat E}(\cD_2)$.

\begin{corollary}[Honest Max Degree Test] \label{star:test:prop} Let $\theta = (2K + C' C^4)\sqrt{\log d/\lfloor n/2\rfloor}$ for an absolute constant $C' > 0$ and  assume that (\ref{bootstrap:rates}) holds. Then for any fixed level $\alpha$, the test $\psi^B_{\alpha, \widehat E}(\cD_2)$ from the output of Algorithm \ref{star:test} satisfies:
\begin{align*}
\limsup_{n \rightarrow \infty} \sup_{\bTheta^* \in \cS_0(s_0)} \PP(\mbox{reject } \Hb_0) \leq \alpha, ~~~~ \liminf_{n \rightarrow \infty}\inf_{\bTheta^* \in \cS_{1}(\theta, s_1)} \PP(\mbox{reject } \Hb_0) = 1.
\end{align*}
\end{corollary}

\section{Extensions to More General Graphical Models}\label{extensions:GM:OTHERS}

The alternative witness test can be applied far beyond the class of Gaussian graphical models. Our first extension is the transelliptical graphical models proposed by \cite{Liu2012Transelliptical}. We start with some definitions.

\begin{definition}[Elliptical distribution \cite{fang1990symmetric}] Let $\bmu^* \in \RR^d$ and $\bSigma^* \in \RR^{d\times d}$. A $d$-dimensional random vector $\bX$ has an elliptical distribution, denoted with $\bX \sim EC_d(\bmu^*,\bSigma^*, \xi)$, if $\bX \stackrel{D}{=}\bmu^* + \xi \Ab \bU$, where $\bU \in \RR^q$  is a random vector uniformly distributed on the unit sphere, $\xi \geq 0$ is a scalar random variable independent of $\bU$, and $\Ab \in \RR^{d \times q}$ is a deterministic matrix such that $\Ab \Ab^T = \bSigma^*$.
\end{definition}

\begin{definition}[Transelliptical Distribution \cite{Liu2012Transelliptical}] 
A continuous random vector $\bX= (X_1,\ldots, X_d)^T$ has a transelliptical distribution, if there exist monotone univariate functions $f_1,\ldots, f_d$ and a non-negative random variable $\xi$, with $\PP(\xi = 0) = 0$, such that:
	$$
	(f_1(X_1),\ldots, f_d(X_d))^T \sim EC_d(0,\bSigma^*, \xi),
	$$
	where $\bSigma^*$ is symmetric with $\diag(\bSigma^*) = 1$ and $\bSigma^* > 0$. 
\end{definition}

A graphical model structure over the family of transelliptical distributions 
is defined through the notion of the ``latent generalized concentration matrix'' --- $\bTheta^* =  (\bSigma^*)^{-1}$. An edge is present between the two variables $X_j,X_k$ if and only if $\Theta^*_{jk} \neq 0$. To construct an estimate of $\bTheta^*$, one first estimates $\bSigma^*$ by a ranked based estimator $\hat \bSigma$. Next one obtains a consistent estimator of $\hat \bTheta$ satisfying (\ref{maxnorm:prec:ass}) and (\ref{other:norms:prec:ass}) by applying CLIME on $\hat \bSigma$ (see \cite{Liu2012Transelliptical} for details). A decorrelation method can then be applied to substitute the estimates $\hat \Theta_{jk}$ with $\tilde \Theta_{jk}$, where $\tilde \Theta_{jk}$ are guaranteed to have asymptotic normal distributions (see \cite{neykov2015aunified} e.g). The alternative witness test immediately applies to this setup. 

Our second extension is the semiparametric exponential family graphical model proposed by \cite{yang2014semiparametric}. Below by indexing a vector $\vb = (v_1, \ldots, v_d)^T$ with $\vb_{\setminus j}$ we mean omitting the $j$\textsuperscript{th} element of that vector.

\begin{definition}[Semiparametric Exponential Family Distribution]A $d$-dimensional random vector $\bX = (X_1, \ldots , X_d) \in \RR^d$ follows a semiparametric exponential family distribution if for any $j \in [d]$, the conditional density of $X_j | \bX_{\setminus j}$ satisfies
\begin{align}
p(x_j |\bx_{\setminus j}) = \exp\Big[x_j (\bbeta_j^{*T} \bx_{\setminus j}) + f_j (x_j) - b_j (\bbeta_j, f_j)\Big],\label{codnt:densities}
\end{align}
where $f_j (\cdot)$ is a base measure and $b_j (\cdot, \cdot)$ is the log-partition function.\end{definition}
The above model is \textit{semiparametric} since the base measure functions $f_j$ are unknown, and it forms a rich family of models including Gaussian, Poisson and Ising models. It is easy to see that $\beta_{jk} = 0$  if and only if $X_j$ and $X_k$ are conditionally independent. Thus the semiparametric exponential family naturally defines a graphical model.   A paiwrise rank based procedure has been proposed by \cite{yang2014semiparametric}  to obtain an estimator $\hat \bbeta_j$ of $\bbeta_j$ for all $j \in [d]$. \cite{yang2014semiparametric} also provide finite sample bounds on $\|\hat \bbeta_j - \bbeta_j\|_2$, $\|\hat \bbeta_j - \bbeta_j\|_1$ with the optimal rates of convergence. They further propose a decorrelated estimator $\tilde \beta_{jk}$ satisfying $\sqrt{n}(\tilde \beta_{jk} - \beta_{jk}) \rightsquigarrow N(0,\sigma^2_{jk})$ for some limiting  variance $\sigma^2_{jk}$.  Therefore, we can apply the alternative witness test to obtain valid combinatorial structure tests.

\section{Multi-Edge Divider Examples}\label{supp:bounded:edges}

\subsection{Bounded Edge Sets}\label{sub:supp:bounded:edges:examples}

\begin{proof}[Proof of Example \ref{max:deg:1:edge}]
In order to show the above result, we start by building a graph $G_0$ by constructing $\lfloor\frac{d}{s_1 + 1}\rfloor$ non-intersecting $s_0$-star graphs first (see Fig \ref{sparse:star:prior:one:edge}). Define the star graph centers $C_j = (s_1 + 1)j + 1$ for $j = 0,\ldots, \lfloor\frac{d}{s_1 + 1}\rfloor - 1$. Next, define $G_0$ by
$$
G_0 := \Bigr(\overline V, \bigcup_{j = 0}^{\lfloor\frac{d}{s_1 + 1}\rfloor - 1} \bigcup_{k = 1}^{s_0}\{(C_j, C_j + k)\}\Bigr).
$$
We define the devider $\bC$ as follows: $$
\bC := \Big\{\bigcup_{k = s_0 + 1}^{s_1}\big\{ (C_j, C_j + k)\big\} ~\Big | ~ j = 0, \ldots,  \lfloor\frac{d}{s_1 + 1} \rfloor- 1\Big\},
$$
where we simply connect each of the vertices $C_j$ to the remaining $s_1 - s_0$ vertices in the block.

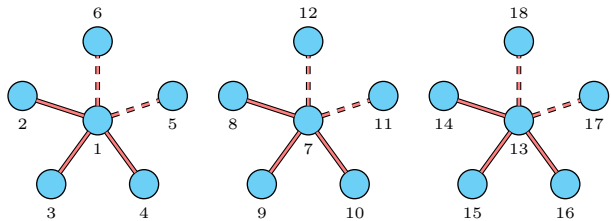
\begin{figure}[t]
\centering
\begin{tikzpicture}[scale=.7]
\tikzset{LabelStyle/.style = {above, fill = white, text = black, fill opacity=0, text opacity = 1}}
\SetVertexNormal[Shape      = circle,
                  FillColor  = cyan!50,
                  MinSize = 11pt,
                  InnerSep=0pt,
                  LineWidth = .5pt]
   \SetVertexNoLabel
      \tikzset{EdgeStyle/.style= {thin,
                                double          = red!50,
                                double distance = 1pt}}
                                \begin{scope}[rotate=90]
 \grEmptyStar[prefix=a,RA=1.5]{6}%
 \end{scope}
 \Edge(a5)(a1)
  \Edge(a5)(a2)
    \Edge(a5)(a3)
            \node[below] at (a5.-90) {\tiny $1$};
        \node[below] at (a1.-90) {\tiny $2$};
    \node[below] at (a2.-90) {\tiny $3$};
        \node[below] at (a3.-90) {\tiny $4$};
       \node[below] at (a4.-90) {\tiny $5$};
        \node[above] at (a0.+90) {\tiny $6$};
   \begin{scope}[shift={(4cm, 0cm)},rotate=90]
    \grEmptyStar[prefix=b,RA=1.5]{6}%
 \Edge(b5)(b1)
  \Edge(b5)(b2)
    \Edge(b5)(b3)
   \end{scope}
               \node[below] at (b5.-90) {\tiny $7$};
        \node[below] at (b1.-90) {\tiny $8$};
    \node[below] at (b2.-90) {\tiny $9$};
        \node[below] at (b3.-90) {\tiny $10$};
                \node[below] at (b4.-90) {\tiny $11$};
                        \node[above] at (b0.+90) {\tiny $12$};
      \begin{scope}[shift={(8cm, 0cm)},rotate=90]
    \grEmptyStar[prefix=c,RA=1.5]{6}%
 \Edge(c5)(c1)
  \Edge(c5)(c2)
    \Edge(c5)(c3)
   \end{scope}
               \node[below] at (c5.-90) {\tiny $13$};
        \node[below] at (c1.-90) {\tiny $14$};
    \node[below] at (c2.-90) {\tiny $15$};
        \node[below] at (c3.-90) {\tiny $16$};
                \node[below] at (c4.-90) {\tiny $17$};
                        \node[above] at (c0.+90) {\tiny $18$};
        \tikzset{EdgeStyle/.style= {thin,dashed,
                                double          = red!50,
                                double distance = 1pt}}
   \Edge(a5)(a0)
   \Edge(b5)(b0)
      \Edge(c5)(c0)
          \Edge(a5)(a4)
              \Edge(b5)(b4)
              \Edge(c5)(c4)
\end{tikzpicture}
\vspace{-1em}
 \caption{Test for the maximum degree $\cG_0 := \{G~|~ d_{\max}(G) \leq s_0\}$ vs $\cG_1 = \{G~|~ d_{\max}(G) \geq s_1\}$ with $s_0 = 3$, $s_1 = 5$ and $d = 18$.  The solid edges represent $G_0 \in \cG_0$ with maximum degree $s_0 = 3$. We construct the devider $\bC = \{\{(1,5), (1,6)\}, \{(7,11), (7, 12)\}, \{(13,17), (13, 18)\}\}$.}\label{sparse:star:prior:one:edge}
 \vspace{-12pt}
\end{figure}
Since the predistance between any two different edge sets $S, S'\in \bC,$ is $d_{G_0}(S,S') = \infty$, the set $\bC$ itself is a $(\log |\bC|)$-packing set. The latter implies that $M(\bC, d_{G_0}, \log |\bC|) = \log |\bC| \asymp \log (d/(s_1 + 1)) \asymp \log d$. In addition, one can easily check that $\|\Ab_0\|_2 = \sqrt{s_0}$ and $\|\Ab_0\|_1 = s_0$. By Theorem \ref{suff:cond:lower:bound:many} we have that under the required scaling $\liminf_{n \rightarrow \infty} \gamma(\cS_0(\theta, s), \cS_1(\theta, s)) = 1$, when $\theta < \kappa \sqrt{\log d/n}$ for a sufficiently small $\kappa$. 
\end{proof}
\subsection{Unbounded Edge Sets}

\begin{proof}[Proof of Example \ref{unfixed:sparse:star}]
Before we show how this example follows from Theorem \ref{comb:NAS} we would like to highlight the difference between Examples \ref{max:deg:1:edge} and \ref{unfixed:sparse:star}. First, note that Example \ref{unfixed:sparse:star} is more flexible compared to Example \ref{max:deg:1:edge} since the number of edges is allowed to scale with $n$. However  Example \ref{unfixed:sparse:star} is also more restrictive in that we require $s = O(d^{\gamma})$ for some $\gamma < 1/2$, which is not required by Example \ref{max:deg:1:edge}. 

We now first explain the construction of $(G_0, \bC)$ and then we formalize it. We start by splitting the vertices into two parts $\{1, \ldots,  \lfloor \sqrt{d} \rfloor\}$ and $\{ \lfloor \sqrt{d} \rfloor+1, \ldots, d\}$. The graph $G_0$ is constructed based only on the first part of vertices, and consists of non-intersecting $s_0$-star graphs. To build the divider $\bC$, we select any $s_1-s_0$ vertices from the second set vertices $\{ \lfloor \sqrt{d} \rfloor+1, \ldots, d\}$ and connect them to any center of the $s_0$-star graphs in $G_0$ (e.g., vertex 1). The set $\bC$ contains all such edge sets (see Fig \ref{sparse:star:prior:multi:edge} in the main text). Formally, to construct a graph $G_0$ we split the vertex set $\{1,\ldots, \lfloor \sqrt{d} \rfloor\}$ into $\lfloor\frac{\sqrt{d}}{s_0 + 1}\rfloor$ non-intersecting $s_0$-star graphs. We take centers $C_j = (s_0 + 1)j + 1$ for $j = 0, \ldots, \lfloor \frac{\sqrt{d}}{s_0 + 1}\rfloor - 1$ and connect them to the remaining vertices as follows:
$$
G_0 := \Bigr(\overline V,\bigcup_{j = 0}^{\lfloor\frac{d}{s_0 + 1}\rfloor - 1} \bigcup_{k = 1}^{s_0}\{(C_j, C_j + k)\}\Bigr).
$$

Since $s_0 < s_1 \leq s = O(d^{\gamma})$, this creates at least $\lfloor \sqrt{d} /d^{\gamma}\rfloor \asymp d^{1/2 - \gamma}$ graphs. Denote the set of centers of the $s_0$-star graphs with 
$$\cJ := \bigcup_{j = 0}^{\lfloor \frac{\sqrt{d}}{s_0 + 1}\rfloor - 1} \{C_j\}.$$
To construct the divider $\bC$ from the remaining  $d - \lfloor \sqrt{d} \rfloor$ isolated vertices we choose a set $I$ of $s_1 - s_0$ vertices and connect them with any of the vertex $C$ from the center set $\cJ$. Formally we let $\bC := \{S_{C,I}\}_{C \in \cJ, I \in \cI},$ where 
$$S_{C,I} = \{(C,i)\}_{i \in I}\mbox{ and } \cI = \Big\{I \in \{M \in 2^{[d]} ~|~ |M| =  s_1 - s_0\} ~\big|~  \min_{i \in I} i >  \lfloor \sqrt{d} \rfloor \Big\}.$$ 

Consider the collection of vertex buffers $ \cV_{S,S'} =  V(S) \cap V(S')$. A visualization of the key quantities $G_0, \bC, \cV_{S,S'}$ for two specific $S,S' \in \bC$ is provided on Fig \ref{sparse:star:prior:multi:edge} in the main text. Now we will argue that the set $\{\mathds{1}(v \in \cV_{S,S'})\}_{v \in V(S)}$ is negatively associated, as required by Assumption \ref{asmp:na}. Note that uniformly selecting a set $S' \in \bC$ is equivalent to uniformly selecting a center $C \in \cJ$ and a vertex set $I \in \cI$ and construct $S'$ by connecting $C$ to every vertex in $I$. By construction the selection of a center $C$ is independent of the selection of the vertex set $I$. By a result in Section 3.1 (c) of \cite{joag1983negative} the random variables $\{\mathds{1}(v \in \cV_{S,S'})\}_{v \in V(S)}$ are negatively associated. Next we calculate the remaining quantities from Theorem \ref{comb:NAS}.

By the triangle inequality $\|\Ab_{S,S'}\|_2 \leq \sqrt{s_1} + \sqrt{s_1 - s_0}\leq 2 \sqrt{s_1}$ and $\|\Ab_{S,S'}\|_1 \leq 2 s_1 $ and therefore  $\Lambda \leq 2 \sqrt{s_1}$, $\Gamma \leq 2 s_1$, and $\cB \leq \Lambda^4 \leq 16 s_1^2 \leq 16 s^2$. To calculate upper and lower bounds of $\cR$ we only need to consider $|S \cap S'| \neq 0$ and in this case $S$ and $S'$ must share the same center. In this situation by definition we have $|S \cap S'| = |\cV_{S,S'}| -1$ and therefore $1/2 \leq \cR \leq 1$. 
 
Finally we calculate $\max_{S \in \bC} \EE_{S'} |\cV_{S,S'}|$. Recall that $\cV_{S,S'} = V(S) \cap V(S')$ and that a set $S'$ can be constructed by first selecting a center $C$ uniformly from the set $\cJ$ and then selecting a set $I$ uniformly from $\cI$. Hence for all $S \in \bC$
$$
\EE_{S'} |\cV_{S,S'}| = \EE [\mathds{1}(C \in V(S)) | S] +  \EE [|I \cap V(S)| ~|~ S]  = \frac{1}{|\cJ|} + \frac{(s_1 - s_0)^2}{d - (s_0 + 1)|\cJ|}.
$$
The last equality implies that $M_{\text{B}}(\bC, G_0) \asymp \log d$. An application of Theorem \ref{comb:NAS}, and taking into account the required scaling completes the result.
\end{proof}

\begin{proof}[Proof of Example \ref{clique:detection:example}]
Consider $G_0 = (\overline V, \varnothing)$, and take the divider: 
$$\bC = \big\{\mbox{all } s \mbox{ cliques with vertices in }\overline V\big\}.$$ 
Take two sets $S,S' \in \bC$. Define the vertex buffer as $\cV_{S,S'} = V(S) \cap  V(S')$. Similarly to Example \ref{unfixed:sparse:star}, conditioning on $S$ the random variables $\{\mathds{1}(v \in \cV_{S,S'})\}_{ v \in V(S)}$ are negatively associated (see Section 3.1 (c) of \cite{joag1983negative}). By the definition of vertex buffer sets we have that $|\cV_{S,S'}| \leq s$, $|S \cap S'| = {\cV_{S,S'} \choose 2} < |\cV_{S,S'}|^2 \leq s |\cV_{S,S'}|$ and hence ${s-1}/{2} \leq \cR \leq s$. Furthermore $\Gamma \leq 2s$, and therefore $\Lambda \leq 2s$. Therefore $\cB \leq 4s^3$. For all $v \in V(S)$ we have $\EE_{S'}[ \mathds{1}(v \in \cV_{S,S'})] = {s}/{d}$ which implies $\EE_{S'}[|\cV_{S,S'}|] = {s^2}/{d}$ for all $S \in \bC$, and therefore $M_{\text{B}}(\bC, G_0) = \log(d/s^2)$. An application of Theorem \ref{comb:NAS} shows the statement.
\end{proof}

\begin{proof}[Proof of Example \ref{cycle:detection:example}]
Consider $G_0 = (\overline V, \varnothing)$, and take the divider: 
$$\bC = \big\{\mbox{all } \mbox{cycles joining $s$-vertices from $\overline V$ in an anti-clockwise increasing order}\big\}.$$ 
Two cycles from the set $\bC$ are visualized on Fig \ref{cycle:detection:subfig}. Take two sets $S,S' \in \bC$. Define the vertex buffer as $\cV_{S,S'} = V(S) \cap  V(S')$. Conditioning on $S$ the random variables $\{\mathds{1}(v \in \cV_{S,S'})\}_{ v \in V(S)}$ are negatively associated (see Section 3.1 (c) of \cite{joag1983negative}). By the definition of vertex buffer sets we have that $|\cV_{S,S'}| \leq s$, $|\cV_{S,S'}|  - 1 \leq |S \cap S'| \leq |\cV_{S,S'}| \leq s$ and hence $\cR = 1$. Furthermore $\Gamma \leq 2$, and therefore $\Lambda \leq 2$. Therefore $\cB \leq 16$. For all $v \in V(S)$ we have $\EE_{S'}[ \mathds{1}(v \in \cV_{S,S'})] = {s}/{d}$ which implies $\EE_{S'}[|\cV_{S,S'}|] = {s^2}/{d}$ for all $S \in \bC$, and therefore $M_{\text{B}}(\bC, G_0) = \log(d/s^2)$. An application of Theorem \ref{comb:NAS} shows the statement.
\end{proof}

\subsection{Upper Bounds}\label{upper:bounds:multi:NAS}

In this section we develop algorithms matching the bounds in Sections \ref{bounded:edges} and \ref{scaling:edges}.

\subsubsection{Maximum Degree Test}\label{max:deg:honest:test}

We propose a test matching lower bound in both Examples \ref{max:deg:1:edge} and \ref{unfixed:sparse:star}. Recall that the sub-decomposition is $\cG_0 = \{G~|~ d_{\max}(G) \leq s_0\}$ and $\cG_1 = \{G~|~ d_{\max}(G) \geq s_1\}$. The maximum degree test is presented in Algorithm \ref{star:test:main}. 

\begin{algorithm}[t]
\normalsize
\caption{Maximum Degree Test}\label{star:test:main}
\begin{algorithmic}
\STATE \textbf{Input:} $\cD=\{\bX_{i}\}_{i = 1}^n$, level $0 < \alpha < 1$.
\STATE Split the data into  $\cD_1 = \{\bX_i\}_{i = 1}^{\lfloor n/2\rfloor }, \cD_2 = \{\bX_i\}_{i = \lfloor n/2\rfloor + 1}^{n}$
\STATE Add edges $e$ from high to low $|\hat \Theta_e^{(1)}|$ values until a vertex of degree $s_0 + 1$ occurs.\STATE Let $\widehat E$ be the set of $s$ edges connecting the vertex to its $s_0 + 1$ neighbours.\vspace{-8pt}
\STATE \textbf{Output:} $\psi^B_{\alpha, \widehat E}(\cD_2)$.
\end{algorithmic}
\end{algorithm}

We provide a result on the performance of the test $\psi^B_{\alpha, \widehat E}(\cD_2)$ (see Algorithm \ref{star:test:main}), in Section \ref{supp:max:deg:honest:test} of the supplement, which matches the signal strength limitation of both Examples \ref{max:deg:1:edge} and \ref{unfixed:sparse:star}.

\subsubsection{Clique Detection Test}

Here we devise a test to match the lower bound of Example \ref{clique:detection:example}. Unlike previous tests, the clique detection test is not computationally feasible. Recall that $\hat \bSigma = n^{-1} \sum_{i = 1}^n \bX_i^{\otimes 2}$. Define $\hat \lambda_{\min} = \min_{|C| = s} \lambda_d(\hat \bSigma_{CC}),$
where $\hat \bSigma_{CC}$ is a sub-matrix of $\hat \bSigma$ with both column and row indices in $C$, and $\lambda_d(\Ab)$ is the smallest eigenvalue of the matrix $\Ab$. Consider the test $\psi = \mathds{1}(\hat \lambda_{\min} < \nu)$, where $\nu$ is
$$
\nu := \Bigr(1 - (\sqrt{2} + 1) \sqrt{({s \log (ed/s) + \log (2\alpha^{-1})})/{n}}\Bigr)^2,
$$ for some small constant $\alpha \geq 0$. We have

\begin{proposition}\label{min:eigenval:test:prop} Suppose that ${({s \log (ed/s) + \log (2\alpha^{-1})})/{n}} = o(1)$ for a constant $\alpha \in (0,1)$. Then for values of $\theta \in (0,1)$ satisfying $\theta > \kappa\sqrt{\frac{\log (e d/s)}{sn}},$
for an absolute constant $\kappa$, we have
$$
\limsup_{n \rightarrow \infty}\sup_{\bTheta \in \cS_0} \PP_{\bTheta}(\mbox{reject } \Hb_0) \leq \alpha ~~~~ \liminf_{n \rightarrow \infty} \inf_{\bTheta \in \cS_1(\theta, s)} \PP_{\bTheta}(\mbox{reject } \Hb_0) = 1.
$$
\end{proposition}

The proof of Proposition \ref{min:eigenval:test:prop} is deferred to Appendix \ref{appendix:clique:detection} in the supplement. We  remark that the bound of Proposition \ref{min:eigenval:test:prop} matches the lower bound of Example \ref{clique:detection:example} up to a scalar in the regime $s = O(d^{\gamma})$ for $\gamma < 1/2$, and ${s \log d}/{n} = o(1)$.

\subsubsection{Cycle Detection Test}

To obtain a cycle detection test one can simply apply Algorithm \ref{cycle:presence}. Since the signal strength lower bound established by Example \ref{cycle:detection:example} is of the magnitude $\sqrt{\log d/n}$ when $s = O(d^{\gamma})$ for $\gamma < 1/2$, Algorithm \ref{cycle:presence} matches the bound up to constants.

\section{Single-Edge Divider Lower Bound Proofs} \label{single:edge:app:proofs}

\begin{proof}[Proof of Theorem \ref{suff:cond:lower:bound:deletion:NAS}] The proof of this result follows the same line of argument as the proof of Theorem \ref{suff:cond:lower:bound}. Here we just sketch the differences. The first important observation is that the risk $\gamma(\cS_0, \cS_1)$ is symmetric in the sense that $\gamma(\cS_0, \cS_1) = \gamma(\cS_1, \cS_0)$. This implies that if controlling the eigenvalues of precision matrices made in the proof of Theorem \ref{suff:cond:lower:bound} hold for edge deletion divider, the proof will continue to hold. If we denote by $\Ab_e$ the adjacency matrix of $(\overline V, \{e\})$ for any $e \in \cC$, then the adjacency matrices of the graphs under the null $G_{\setminus e}$ are given by $\Ab_1 - \Ab_e$. It is now a simple exercise to check that:
\begin{align*}
& \Tr(\Ab_1^k) - \Tr((\Ab_1 - \Ab_{e})^k)  \leq 2 \|\Ab_1\|_2^k, \\
& \Tr((\Ab_1 - \Ab_{e'})^k) - \Tr((\Ab_1 - \Ab_{e} - \Ab_{e'})^k)  \leq 0,
\end{align*}
and the proof is completed as in Theorem \ref{suff:cond:lower:bound}.
\end{proof}

\section{Structure Testing Algorithms Proofs}\label{sec:alg-proof}

\subsection{Connectivity Testing Proofs}\label{sec:proof-conn}

Define formally definitions (\ref{generic:S0:def:new}) and (\ref{generic:S1:def:new}) in the case of connectivity testing:
\begin{align}
\cS_0(s) & = \{\bTheta \in \cM(s) ~|~ G(\bTheta) \in \cG_{0}\}, \label{conn:test:param0:span:forst}\\
\cS_1(\theta, s) & = \{\cM(s) ~|~ G(\bTheta) \in \cG_{1}, \exists ~ \underset{\tiny \mbox{tree}} {E}\subset E(G(\bTheta)), \min_{e \in E} |\Theta_e| \geq \theta \} \label{conn:test:paramA:span:forst},
\end{align}

\begin{lemma}\label{ehat:consistency} Assume the estimate $\hat \bTheta^{(1)}$ satisfies (\ref{maxnorm:prec:ass}) on the parameter space $\cM(s)$ and let $\theta = r K \sqrt{\log d/\lfloor n/2\rfloor}$, where $r > 2$. Then for all $e \in E(\hat T)$, we have $|\Theta^*_{e}| \geq \frac{r - 2}{r}\theta$ when $\bTheta^* \in \cS_1(\theta,s)$ (defined in (\ref{conn:test:paramA:span:forst})) with probability at least $1 - 1/d$. Moreover, when $\bTheta^* \in \cS_0(s)$ (see definition in (\ref{conn:test:param0:span:forst})) for $\widehat E = \argmin_{e \in E(\hat T)} |\Theta^{(1)}_{e}|$ we have $\Theta^*_{\widehat E} = 0$ with probability at least $1 - 1/d$.
\end{lemma}

\begin{proof}[Proof of Lemma \ref{ehat:consistency}] By (\ref{maxnorm:prec:ass}) we obtain the following:
\begin{align}
|\hat \Theta_e^{(1)}| & \geq \frac{r-1}{r}\theta, \mbox{ when } |\Theta^*_e| \geq \theta, \label{true:signals} \\
|\hat \Theta_e^{(1)}| & \leq \frac{\theta}{r}, \mbox{ when } |\Theta^*_e| = 0, \label{null:signals}
\end{align}
with probability at least $1 - 1/d$ uniformly for all $e$. First consider the case when $\bTheta^* \in \cS_0(s)$. Since the graph is disconnected, we have that there exists at least one edge in $ e \in \hat T$, such that $\Theta^*_e = 0$ and hence by (\ref{null:signals}) we have $|\hat \Theta^{(1)}_e|  \leq \theta/r$. This implies that:  $|\hat \Theta^{(1)}_{\widehat E}| \leq \theta/r$. Now suppose that $\Theta^*_{\widehat E} \neq 0$. Denote the connected component of $G(\bTheta^*)$ containing $\widehat E$ with $\cC^*$. Since $\bTheta^* \in \cS_{0}(s)$ there exists a spanning tree $T_{\cC^*}$ of $\cC^*$ such that for all $e \in E(T_{\cC^*})$ we have $|\Theta^*_e| \geq \theta$. By (\ref{true:signals}) the latter implies that $|\hat \Theta_e^{(1)}| > |\hat \Theta^{(1)}_{\widehat E}|$ for all $e \in E(T_{\cC^*})$. On the other hand, note that the value $|\hat \Theta^{(1)}_{\widehat E}|$ is the largest among the values of all edges connecting the two connected components of $G(\hat \bTheta^{(1)})$ which $\widehat E$ connects. However, as we argued  there exist an edge $\tilde e \in T_{\cC^*}$ connecting these two components satisfying $|\hat \Theta^{(1)}_{\tilde e}| > |\hat \Theta^{(1)}_{\widehat E}|$. This is a contradiction with the fact that $\hat T$ is a MST.  Hence $\Theta^*_{\widehat E} = 0$.

Next let $\bTheta^* \in S_1(\theta,s)$. In this case by (\ref{true:signals}) the tree will cover only edges with $|\hat \Theta_e| \geq (r-1)\theta/r$ and hence by assumption (\ref{maxnorm:prec:ass}) and the fact that $\bTheta^* \in S_1(\theta,s)$ we have $|\Theta^*_{\widehat E}| \geq (r-2)\theta/r$. This completes the proof.
\end{proof}

\begin{proof}[Proof of Corollary \ref{conn:test:prop}] We first verify the size of the test. By Lemma \ref{ehat:consistency} we conclude that with asymptotic probability $1$ we have $\min_{e \in E(\hat T)}|\Theta^*_{e}| = 0$ when $\bTheta^* \in \cS_0(s)$. Conditioning on this event by Proposition \ref{multiple:edge:testing:validity} we conclude that:
$$\limsup_{n \rightarrow \infty} \sup_{\bTheta^* \in \cS_0(\theta)}\PP(\psi = 1) \leq \alpha.$$
For the second part, with asymptotic probability $1$ we have that $|\Theta^*_{\widehat E}| \geq \theta \frac{\epsilon}{2 + \epsilon} \geq \frac{\sqrt{2}\epsilon}{2 + \epsilon} K \sqrt{\log d/n}$. Using the second part of Proposition \ref{multiple:edge:testing:validity} we can guarantee:
$$
\liminf_{n \rightarrow \infty} \sup_{\bTheta^* \in \cS_1(\theta,s)}\PP(\psi = 0) = 0.
$$
This completes the proof.
\end{proof}

\begin{lemma}\label{ehat:consistency:conn:comp}  Assume the estimate $\hat \bTheta^{(1)}$ satisfies assumption (\ref{maxnorm:prec:ass}) on the parameter space $\cM(s)$ and $\theta \geq rK \sqrt{\log d/\lfloor n/2\rfloor}$ for some $r > 2$. If $\bTheta^* \in \cS_0(s)$ then $|\Theta^*_{\widehat E}| = 0$ where $\widehat E = \argmin_{e \in E(\hat F)} |\hat \Theta^{(1)}_{ e}|$ with probability at least $1 - d^{-1}$. Furthermore, when  $\bTheta^* \in \cS_{1}(\theta,s)$, for all $e \in  E(\hat F)$ we have $|\Theta^*_e| \geq \frac{r-2}{r}\theta$ with probability at least $1 - d^{-1}$. 
\end{lemma}

\begin{proof}[Proof of Lemma \ref{ehat:consistency:conn:comp}] First observe that we have that $\hat F$ with $I$ connected components has an edge set $E(\hat F) = \{\widehat E_i\}_{i \in [d - I]}$, where $\{\widehat E_{i}\}_{i \in [d-1]}$ is an ordering of $E(\hat T)$ such that $|\hat \Theta^{(1)}_{\widehat E_i}| \geq |\hat \Theta^{(1)}_{\widehat E_{i + 1}}|$ for all $i \in [d-2]$. The remainder of the proof is similar to that of Lemma \ref{ehat:consistency} so we omit the details. 
\end{proof}

\begin{proof}[Proof of Corollary \ref{conn:comp:test:prop}] Setting $r = 2 + \frac{C' C^4}{K}$, we get that 
$$\theta = \max(r C'C^4/(r-2), r K)\sqrt{\log d/\lfloor n/2\rfloor}.$$ 
Hence $\theta \geq r K \sqrt{\log d/\lfloor n/2\rfloor}$, and a direct application of Lemma \ref{ehat:consistency:conn:comp} ensures that with high probability the edge set $E(\hat F)$ satisfies $\min_{e \in E(\hat F)}|\Theta^*_e| \geq (r - 2)\theta/r$ when $\bTheta^* \in \cS_{1}(\theta,s)$ and $\min_{e \in E(\hat F)}|\Theta^*_e|  = 0$ when $\bTheta^* \in \cS_{0}(s)$. The remaining part of the proposition follows directly by Proposition \ref{multiple:edge:testing:validity}.
\end{proof}

\subsection{Cycle Testing Proofs}

\begin{lemma}\label{ehat:consistency:cycle} Assume the estimate $\hat \bTheta^{(1)}$ satisfies (\ref{maxnorm:prec:ass}) on the parameter space $\cM(s)$ and $\theta = r K \sqrt{\log d/\lfloor n/2\rfloor}$ for some $r \geq 2$. Let $\widehat E$ be the cycle obtained by Algorithm \ref{cycle:presence}. We have $\min_{e \in \widehat E}|\Theta^*_{e}| \geq (r - 2)\theta/r$ when $\bTheta^* \in \cS_1(\theta,s)$ and $\min_{e \in \widehat E} |\Theta^*_e| = 0$ when $\bTheta^* \in \c1_0(s)$ with probability at least $1 - 1/d$.
\end{lemma}

\begin{proof}[Proof of Lemma \ref{ehat:consistency:cycle}] When $\bTheta^* \in \cS^1_1(\theta,s)$, there exists a cycle $E \subset E(G(\bTheta^*))$. Hence  by (\ref{true:signals}) when the algorithm terminates we will have covered only edges such that $|\hat \Theta_e| > (r-1) \theta/r$. The latter, combined with (\ref{true:signals})  implies that until a cycle is reached, Algorithm \ref{cycle:presence} will have included only edges of true magnitude $|\Theta^*_e| \geq (r-2)\theta/r$, which completes the proof of the first claim. 

On the other hand, when $\bTheta^* \in \cS_0(s)$, we have that there exists an edge $e \in \widehat E$ such that $\Theta^*_e = 0$, which completes the proof. \end{proof}

\begin{proof}[Proof of Corollary \ref{cycle:nosignal:str}] To handle the null hypothesis observe that by Lemma \ref{ehat:consistency:cycle} we have $\min_{e \in \widehat E} |\Theta^*_e| = 0$. Invoking Proposition \ref{multiple:edge:testing:validity} gives the desired result. To show that the test is asymptotically powerful we use Proposition \ref{multiple:edge:testing:validity} in conjunction with the other result of Lemma \ref{ehat:consistency:cycle}.
\end{proof}

\subsection{SAP Length Test Proofs}

\begin{lemma}\label{sap:test:lemma} Assume the estimate $\hat \bTheta^{(1)}$ satisfies (\ref{maxnorm:prec:ass}) on the parameter space $\cM(s)$ and $\theta = r K \sqrt{\log d/\lfloor n/2\rfloor}$ for some $r \geq 2$. If $\bTheta^* \in \cS_0(s)$, we have that $\min_{e \in \widehat E} |\Theta^*_e| = 0$. Conversely when $\bTheta^* \in \cS_1(\theta, s)$ we have $\min_{e \in \widehat E} |\Theta^*_e| \geq (r-2)\theta/r$ with probability at least $1 - 1/d$.
\end{lemma}

\begin{proof}[Proof of Lemma \ref{sap:test:lemma}] First, we consider the case when $\bTheta^* \in \cS_0(s)$. Since $\widehat E$ is a SAP with $I+1$ edges and $G(\bTheta^*)$ is free of such SAPs when $\bTheta^* \in \cS_0(s)$ we must have $\min_{e \in \widehat E} |\Theta^*_e| = 0$. Next, when $\bTheta^* \in \cS_1(\theta, s)$, by (\ref{true:signals})  $\min_{e \in \widehat E} |\hat \Theta_e| > (r-1) \theta/r$. Thus by \ref{maxnorm:prec:ass} $|\Theta^*_e| \geq (r-2)\theta/r$ with probability at least $1 - 1/d$.
\end{proof}

\begin{proof}[Proof of Corollary \ref{sap:test:prop}] To handle the null hypothesis observe that by Lemma \ref{sap:test:lemma} we are guaranteed to have $\min_{e \in \widehat E} \Theta^*_e = 0$. Invoking Proposition \ref{multiple:edge:testing:validity} gives the desired result. To show that the test is asymptotically powerful we use Proposition \ref{multiple:edge:testing:validity} in conjunction with Lemma \ref{sap:test:lemma}.
\end{proof}

\subsection{Triangle-Free Tests Proofs}

\begin{lemma}\label{triangle:test:lemma} Assume the estimate $\hat \bTheta^{(1)}$ satisfies (\ref{maxnorm:prec:ass}) on the parameter space $\cM(s)$ and $\theta = r K \sqrt{\log d/\lfloor n/2\rfloor}$ for some $r \geq 2$. If $\bTheta^* \in \cS_0(s)$, we have that $\min_{e \in \widehat E} |\Theta^*_e| = 0$. Conversely when $\bTheta^* \in \cS_1(\theta, s)$ we have $\min_{e \in \widehat E} |\Theta^*_e| \geq (r-2)\theta/r$ with probability at least $1 - 1/d$.
\end{lemma}

\begin{proof}[Proof of Lemma \ref{triangle:test:lemma}]  Proof is the same as Lemma \ref{sap:test:lemma}, we omit the details.
\end{proof}

\begin{proof}[Proof of Corollary \ref{triangle:test:prop}] Proof is the same as Corollary \ref{sap:test:prop}, we omit the details.
\end{proof}

\subsection{Results for Thresholded Graphs}

In this section, we discuss the hypothesis in \eqref{conn:comp:strength:test-1}
\[
 \Hb_0: G(\bTheta^*, \mu) \text{ disconnected vs }\Hb_1: G(\bTheta^*, \mu) \text{ connected}
\]
and prove the result on the test $\psi^B_{\alpha, \widehat T}(\cD_2)$ proposed in Section \ref{sec:ADHD}. Denote
\begin{align*}
\cS_0(s) &:= \Big \{\bTheta\in  \cM(s) ~|~ G(\bTheta^*, \mu) \text{ disconnected} \Big \} \mbox{ and, } \\
\cS_1(\theta, s) &:= \Big \{\bTheta\in  \cM(s) ~|~ G(\bTheta^*, \mu+ \theta) \text{ connected} \Big \}.
\end{align*}
 We have the following result parallel to Corollary \ref{conn:test:prop}.

\begin{corollary}\label{conn:test:prop-mu} Let $\theta = (2K + C' C^4){\log d/\lfloor n/2\rfloor}$ for an absolute constant $C' > 0$.  Assume that $\mu \ge \theta$ and (\ref{bootstrap:rates}) holds.Then for any fixed level $\alpha$, the test $\psi^B_{\alpha, \widehat E}(\cD_2)$ from the output of Algorithm \ref{al:conn:comp} satisfies:
\begin{align*}
\limsup_{n \rightarrow \infty} \sup_{\bTheta^* \in \cS_0(s)} \PP(\mbox{reject } \Hb_0) \leq \alpha, ~~~~ \liminf_{n \rightarrow \infty} \inf_{\bTheta^* \in \cS_{1}(\theta, s)} \PP(\mbox{reject } \Hb_0) = 1.
\end{align*}
\end{corollary}

\begin{proof}
 The proof is same as the one of Corollary \ref{conn:test:prop}. We only need to change the proof of Proposition \ref{multiple:edge:testing:validity} in Section \ref{sec:supp-1} to the thresholded version. In fact, we only need to change \eqref{eq:tjk}. When $|\Theta_{jk}^*| < \mu$, we have
 \[
  \sqrt{n} |\tilde \Theta_{jk}| \geq \sqrt{n}|\kappa \sqrt{\log d/n}| - \sqrt{n} \max_{(j,k) \in E} |\tilde \Theta_{jk} - \Theta^*_{jk}| + \sqrt{n} \mu \geq \sqrt{n} \mu +  c_{1 - \alpha, E}.
 \]
 All the remaining part of the argument is the same.
\end{proof}

\section{Multi-Edge Divider Proofs}\label{sec:multi-nas-proof}

Before proceeding with the proofs we remind the reader definition (\ref{matrix:Ass}) of $\Ab_{S,S'} = \Ab_0 + \Ab_S + \Ab_{S'}$. In what follows we will use the shorthand notation:
$$
\cV_{S,S' | S} = \cV_{S', S} \cap V(S).
$$

\begin{proof}[Proof of Theorem \ref{suff:cond:lower:bound:many}] 
We denote by $\cL := \exp(M(\bC, d_{G_0}, \log|\bC|))$ the cardinality of the $(\log|\bC|)$-packing. Using Proposition \ref{refined:const:thm} with the constants from Setting 1, and the fact that we have a $(\log|\bC|)$-packing it suffices to control:
\begin{align*}
\MoveEqLeft \frac{2}{\cL^2}\sum_{d_{G_0}(S,S') \geq \log |\bC|} \exp\bigg(\frac{n |\cV_{S,S' | S}|\|\Ab_S\|_2\|\Ab_{S'}\|_2 (2\|\Ab_{S,S'}\|_2\theta)^{2 d_{G_0}(S,S')}\theta^2}{2d_{G_0}(S,S') + 2}\bigg) \\
\MoveEqLeft + \frac{2}{\cL^2}\sum_{d_{G_0}(S,S') = 0}\exp\Big(n |S \cap S'|\theta^2 + \frac{n |\cV_{S,S' | S}| \|\Ab_S\|_2\|\Ab_{S'}\|_2 (2\|\Ab_{S,S'}\|_2\theta)^{2}\theta^2}{4}\Big),
\end{align*}
When the cardinality of each $|S| \leq U$ the above expression can further be controlled by:
\begin{align*}
\MoveEqLeft\underbrace{\frac{2}{\cL^2}\sum_{(S,S'): d_{G_0}(S,S') \geq \log |\bC|} \exp\bigg(\frac{2 n U^3 (2\Lambda_0 \theta)^{2 d_{G_0}(S,S')}\theta^2}{2d_{G_0}(S,S') + 2}\bigg)}_{I_1} \\
& + \underbrace{\frac{2}{\cL}\exp\Big(n U\theta^2 + 2 n U^3 (\Lambda_0\theta)^{2}\theta^2\Big)}_{I_2},
\end{align*}
where $\Lambda_0 := \|\Ab_0\|_2 + 2U$, and we used the facts that $|S \cap S'|, \|\Ab_S\|_2, \|\Ab_{S'}\|_2 \leq U$, $ |\cV_{S', S|S}| \leq |V(S)| \leq 2 U$ and that $d_{G_0}(S,S') = 0$ is only possible when $S \equiv S'$. 
We first deal with the term $I_2$. For values of $\theta < \frac{1}{2 U \Lambda_0},$
we have $I_2 \leq \frac{2}{\cL}\exp(2 nU \theta^2)$. Hence when $\theta \leq \kappa \sqrt{\frac{\log \cL}{nU}}$ for some sufficiently small $0 < \kappa$ we have $I_2 = o(1)$. Next, we proceed similarly to the proof of Step 4 in Theorem \ref{suff:cond:lower:bound}. Observe the identity
$$
I_1 \leq \frac{2}{\cL^2}\sum_{(S,S'): d_{G_0}(S,S') \geq \log |\bC|} \exp\bigg(\log \cL \frac{ \kappa^2 U (2\Lambda_0 \theta)^{2 d_{G_0}(S,S')}}{d_{G_0}(S,S') + 1}\bigg),
$$
which follows since $\theta \leq \kappa \sqrt{\frac{\log \cL}{nU}}$. Using $\theta < \frac{1}{2 U \Lambda_0}$, $\cL < |\bC|$ for $r \geq \log |\bC|$ we have
$$
\log \cL \frac{ \kappa^2 U (2\Lambda_0 \theta)^{2 r}}{r + 1} < \log \cL \frac{ \kappa^2 U^{-2 r + 1}}{r + 1} < 1,
$$
provided that $\kappa$ is small enough (e.g. $\kappa < 1$). By the inequities $e^x \leq 1 + 3 x$ for $x \leq 1$ and $\log x \leq x - 1$ we have
$$
I_1 \leq \frac{2}{\cL^2}\sum_{(S,S'): d_{G_0}(S,S') \geq \log |\bC|} 1 + \frac{3(\cL - 1) (2 \Lambda_0 \theta)^{2 d_{G_0}(S,S') - 1}}{d_{G_0}(S,S') + 1}
$$
Let $K_r := |\{(S,S') ~|~ S, S' \in \bC,  d_{G_0} (S, S')= r\}|$. We have
\begin{align*}
I_{1} < \left(1 - \frac{1}{\cL}\right) + \frac{12 (2 \Lambda_0 \theta)^{2 \lfloor \log \cL\rfloor - 1}}{1 -  (2\Lambda_0\theta)^2} \cL^{-1} \max_{ \lfloor \log\cL\rfloor + 1 \leq r} \frac{K_r}{r}.
\end{align*}
Paying closer attention to the second term we have:
\begin{align*}
\frac{12 (2 \Lambda_0 \theta)^{2 \lfloor \log \cL\rfloor - 1}}{1 -  (2\Lambda_0\theta)^2} \cL^{-1} \max_{ \lfloor \log\cL\rfloor + 1 \leq r} \frac{K_r}{r} & \leq 16 (2 \Lambda_0\theta)^{2 \lfloor \log \cL \rfloor - 1} \frac{{\cL \choose 2} + \cL}{\cL}\\
& \leq 16 (2 \Lambda_0\theta)^{2 \lfloor \log \cL \rfloor - 1} \cL = o(1),
\end{align*}
with the last equalities hold for $2\Lambda_0 \theta \leq \frac{1}{2} < \exp(-1/2)$. This completes the proof.
\end{proof}

\begin{proof}[Proof of Theorem \ref{comb:NAS}] Note that by the definition of $\cB$, the fact that $ |\cV_{S, S' | S}| \vee |\cV_{S, S' | S'}| \leq |\cV_{S, S'}|$, we have:
$$
\max_{S,S' \in \bC} \Big((\|\Ab_S\|_2\|\Ab_{S'}\|_2 \|\Ab_{S,S'}\|_2^2) \wedge (|\cV_{S, S' | S'}|\|\Ab_{S,S'}\|_1^2)\Big) \leq \cB.
$$
Therefore using Proposition \ref{refined:const:thm} it suffices to control:
$$
\overline D_{\chi^2} := \frac{1}{|\bC|^2}\sum_{S,S' \in \bC}\exp\bigg[|\cV_{S, S' | S}| n\theta^{2}\bigg(\cR +  \cB \theta^{2}\bigg)\bigg].
$$
First note that by $|\cV_{S, S' | S}| = \sum_{v \in V(S)} \mathds{1}(v \in \cV_{S,S'}),$ we have:
$$
\overline D_{\chi^2} \leq \frac{1}{|\bC|^2}\sum_{S,S' \in \bC}\exp\bigg[n\theta^{2}\bigg(\cR +  \cB \theta^{2}\bigg) \sum_{v \in V(S)} \mathds{1}(v \in \cV_{S,S'})\bigg].
$$

Denote by $\PP_{S'}$ the measure induced by drawing $S'$ uniformly from $\bC$. Under the assumption: $\theta < \sqrt{{\cR}/{\cB}}$, and using the fact that the random variables $\{\mathds{1}(v \in \cV_{S,S'}) ~|~ v \in V(S)\}$ are negatively associated for every fixed $S \in \bC$, we obtain:
\begin{align*}
\log \overline D_{\chi^2} & \leq \max_{S \in \bC}\Bigr[ \sum_{v \in V(S)} \log \Big[\exp(2 \cR n\theta^2) \PP_{S'}(v \in \cV_{S,S'}) + (1 - \PP_{S'}(v \in \cV_{S,S'}))\Big] \Bigr]\\
& \leq  \max_{S \in \bC} \Bigr[ (\exp(2 \cR n\theta^2) - 1)  \sum_{v \in V(S)} \PP_{S'}(v \in \cV_{S,S'})\Bigr]\\
&  \leq \exp(2 \cR n\theta^2)  \max_{S \in \bC} \EE_{S'} |\cV_{S,S'}|,
\end{align*}
where the expectation $\EE_{S'}$ is taken with respect to a uniform draw of $S' \in \bC$.  The first inequality above is due to negative association, the second inequality is due to $\log (1+x) \le x$. 
Hence for values of $\theta$
$$
\theta \leq \sqrt{ \frac{\log\Big([\max_{S \in \bC} \EE_{S'} |\cV_{S,S'}|]^{-1}\Big)}{4 n \cR}},
$$
we have $\overline D_{\chi^2} \leq \exp([\max_{S \in \bC} \EE_{S'} |\cV_{S,S'}|]^{1/2})$, and therefore:
$$
\liminf_{n} \gamma(\cS_0(\theta, s), \cS_1(\theta, s)) \geq 1 - \frac{1}{2}\sqrt{ \exp([\max_{S \in \bC} \EE_{S'} |\cV_{S,S'}|]^{1/2}) - 1} = 1,
$$
where the last equality holds since $M_{\text{B}}(\bC, G_0) \rightarrow \infty$ implies 
$$\max_{S \in \bC} \EE_{S'} |\cV_{S,S'}| = o(1).$$
\end{proof}

\begin{proposition}\label{refined:const:thm} Let $G_0 \in \cG_0$ and $\bC$ be a divider with null base $G_0$. Then for any collection of vertex buffers $\bV=\{\cV_{S,S'}\}_{S,S' \in \bC}$ and any of the following two settings:
\begin{align*} 
\text{\bf S1:} & ~\cH_{S, S'} := \frac{(|\cV_{S,S' | S}|\wedge |\cV_{S,S' | S'}|) \|\Ab_S\|_2\|\Ab_{S'}\|_2}{\|\Ab_{S,S'}\|^2_2} \text{ and  } \cK_{S, S'} := 
2\|\Ab_{S,S'}\|_2;\\
\text{\bf S2:} & ~ \cH_{S, S'} := \frac{|\cV_{S,S' | S}| |\cV_{S,S' | S'}|}{\|\Ab_{S,S'}\|_1^2}  ~~~~~~~~~~\text{ and }~~~~~~~~~~~~~~  \cK_{S, S'} := 2\|\Ab_{S,S'}\|_1.
\end{align*}
when the signal strength satisfies 
\begin{align}\label{multi:edge:NAS:theta}
\theta < \min_{S,S' \in \bC}\frac{1 - C^{-1}}{2 \sqrt{2} \|\Ab_{S,S'}\|_1},
\end{align} 
we have:
\begin{align*}
\MoveEqLeft \gamma(\cS_0(\theta, s), \cS_1(\theta, s)) \geq \\
\MoveEqLeft \geq 1 - \frac{1}{2}\sqrt{\frac{1}{|\bC|^2}\sum_{S,S' \in \bC}\exp\bigg[n\bigg(|S \cap S'|  \theta^2 + \frac{\cH_{S, S'}(\cK_{S, S'} \theta)^{2(d_{G_0}(S,S')\vee 1 + 1)}}{2(d_{G_0}(S,S')\vee 1 + 1)} \bigg)\bigg] - 1}, 
\end{align*}
\end{proposition}

\begin{remark}\label{thm:extension:rem} Proposition \ref{refined:const:thm} continues to hold for parameter classes $\cM(s)$ not imposing bounds on the $\ell_1$ norm of the precision matrix $\bTheta$, if we substitute $\|\Ab_{S,S'}\|_1$ with $\|\Ab_{S,S'}\|_2$ in (\ref{multi:edge:NAS:theta}). In fact tracking the proof, it is easy to see that the theorem also remains valid for parameter spaces such that $\Ib + \theta \Ab_0 \in \cS_0(\theta, s)$ and $\Ib + \theta (\Ab_0 + \Ab_S) \in \cS_1(\theta, s)$ for all $S \in \bC$, with (\ref{multi:edge:NAS:theta}) replaced by $\theta < \min_{S,S' \in \bC}\frac{1 - C^{-1}}{\|\Ab_{S, S'}\|_2 \vee \sqrt{2} \cK_{S, S'}}$. 
\end{remark}

\begin{definition}\label{ass:tr:abstr:ineq} Let $\bC$ be a divider with null base $G_0 \in \cG_0$ whose adjacency matrix is $\Ab_0$. We call a set of constants $\{\cH_{S,S'}, \cK_{S,S'} ~|~ S, S' \in \bC\}$ \textit{admissible} with respect to the pair $(G_0, \bC)$ if for all even integers $k \geq 4$ the following holds:
\begin{align}\label{tr:abstr:ineq}
\Tr(\Ab_{S,S'}^k + \Ab_0^k - (\Ab_0 + \Ab_S)^k  - (\Ab_0 + \Ab_{S'})^k) \leq \cH_{S,S'} \cK_{S,S'}^k,
\end{align}
for all $S,S' \in \bC$.
\end{definition}
We will see that any admissible set of constants $\{\cH_{S,S'}, \cK_{S,S'} ~|~ S, S' \in \bC\}$ yields the lower bound on the  minimax risk claimed by Proposition \ref{refined:const:thm}.

\begin{proof}[Proof of Proposition \ref{refined:const:thm}]
We will in fact show a slightly stronger result than presented, involving the following additional combinatorial quantity:
$$\Delta_{S, S'} := |\{\mbox{triangles in $G(\Ab_{S,S'})$ with $\geq 1$ edges in $S$ and $\geq 1$ edges in $S'$}\}|.$$

\noindent {\large{\bf Step 1}\textit{(Matrix Construction).}}\\
In this step we construct a set of precision matrices and argue that they fall into the parameter set $\cM(s)$. Take $\bTheta_0 = \Ib + \theta \Ab_0$, $\bTheta_S = \Ib + \theta (\Ab_0 + \Ab_S)$, $\bTheta_{S,S'} = \Ib + \theta \Ab_{S,S'}$, for $S,S' \in \bC$ and some $\theta > 0$. For any $S, S' \in \bC$ we have:
\begin{align*}
\max(\|\Ab_0\|_2, \|\Ab_0 + \Ab_S\|_2, \|\Ab_{S,S'}\|_2) &\leq \|\Ab_{S,S'}\|_2 \leq \|\Ab_{S,S'}\|_1\\\max(\|\Ab_0\|_1, \|\Ab_0 + \Ab_S\|_1, \|\Ab_{S,S'}\|_1) & \leq \|\Ab_{S,S'}\|_1,
\end{align*}
where these inequalities hold since all matrices $\Ab_0$, $\Ab_0 + \Ab_S$ and $\Ab_{S,S'}$ consist only of non-negative entries.

Similarly to the proof of Theorem \ref{suff:cond:lower:bound} we can make sure that the matrices $\bTheta_0$ and $\bTheta_S$ fall into the set $\cM(s)$ and in addition the matrix $\bTheta_{S,S'}$ is strictly positive definite if $\theta < \frac{1 - C^{-1}}{\Gamma}$. Thus by assumption the graphs $G(\bTheta_0) \in \cG_0$ and $G(\bTheta_S) \in \cG_1$ for all $S \in \bC$, and hence $\bTheta_0 \in \cS_0(\theta,s)$ and $\bTheta_S \in \cS_1(\theta,s)$ for all $S \in \bC$. In addition we also have that matrices $\bTheta_0$, $\bTheta_S$, $\bTheta_{S,S'}$ are strictly positive definite for any $S, S' \in \bC$.

\vspace{.3cm}
\noindent {\large{\bf Step 2}\textit{(Risk and Trace Bounds).}}\\
In this step we will lower bound the risk, and will further derive some combinatorial bounds on the traces of powers of adjacency matrices. These bounds are more detail tracking compared to bounds discussed in Theorem \ref{suff:cond:lower:bound}. Similarly to Step 2 of the proof of Theorem \ref{suff:cond:lower:bound} it suffices to bound:
\begin{align*}
& \bigg( \frac{\det(\Ib + \theta(\Ab_0 + \Ab_S))}{\det(\Ib + \theta \Ab_0)}\bigg)^{n/2}\bigg(\frac{\det(\Ib + \theta(\Ab_0 + \Ab_{S'}))}{\det(\Ib + \theta (\Ab_{S,S'}))}\bigg)^{n/2} \\
& = \exp\bigg(\frac{n}{2}\sum_{k = 1}^{\infty} \frac{(-\theta)^{k}}{k} \Tr\Big(\Ab_{S,S'}^k + \Ab_0^k- (\Ab_0 + \Ab_{S})^k - (\Ab_0 + \Ab_{S'})^k  \Big) \bigg),
\end{align*}
Similarly to Step 3 of Theorem \ref{suff:cond:lower:bound} it is easy to argue that for any $k \in \NN$ we have:
$$\Tr(  \Ab_{S,S'}^k + \Ab_0^k - (\Ab_0 + \Ab_{S})^k -  (\Ab_0 + \Ab_{S'})^k ) \ge 0.$$
We will consider three cases: (1) $k < 2(d_{G_0}(S,S') + 1)$, (2) $k < 4$ and $k \ge 2(d_{G_0}(S,S') + 1)$ and (3) $k \ge  4$ and $k \ge 2(d_{G_0}(S,S') + 1)$.
For  $k < 2(d_{G_0}(S,S') + 1)$, similarly to the argument in the Step 3 of the proof of  Theorem \ref{suff:cond:lower:bound}, the above is in fact an equality. 
However, for the case $k < 4$ and $k \ge 2(d_{G_0}(S,S') + 1)$, we will instead show the following two more precise bounds:
\begin{align}
\Tr(\Ab_0^2+ \Ab_{S,S'}^2 - (\Ab_0 + \Ab_{S})^2 - (\Ab_0 + \Ab_{S'})^2) & \leq 4|S \cap S'|, \label{sq:deg:edges}\\
\Tr(\Ab_0^3 + \Ab_{S,S'}^3 - (\Ab_0 + \Ab_{S})^3 - (\Ab_0 + \Ab_{S'})^3) & \geq 6 \Delta_{S, S'}.\label{triangle:deg:edges}
\end{align}
The left hand side of (\ref{sq:deg:edges}) contains edges lying only in the intersection $S \cap S'$ since all closed walks containing at least one edge in $G_0$ cancel out. By definition the number of such edges is $|S \cap S'|$. In addition, each closed walk of length $2$ in (\ref{sq:deg:edges}), has precisely one edge in $S$ and one edge in $S'$. Fixing the first edge to be in $S$ and the second edge to be in $S'$ we notice that each edge appears twice in the total count --- once for each of its two vertices. Further multiplying by $2$ to adjust for the ordering of the edges we obtain $4 |S \cap S'|$. Next, notice that expression (\ref{triangle:deg:edges}) contains only closed walks which are triangles. In addition, similarly to the logic above, only walks containing one edge from $S$ and $S'$ survive in (\ref{triangle:deg:edges}). Further, each triangle is contained at least $6$ times --- once per each of its vertices and once per its $2$ orientations. This completes the proof of (\ref{triangle:deg:edges}). 

We will check the third case in the next step by checking the admissibility in Definition~\ref{ass:tr:abstr:ineq}. 

\vspace{.3cm}
\noindent {\large {\bf Step 3} \textit{(Verifying Admissibility).}}\\
In this step we show that both settings below are admissible for any pair $(G_0, \bC)$:
\begin{align*} 
\text{\bf S1:} & ~\cH_{S, S'} := \frac{(|\cV_{S,S' | S}|\wedge |\cV_{S,S' | S'}|) \|\Ab_S\|_2\|\Ab_{S'}\|_2}{\|\Ab_{S,S'}\|^2_2} \text{ and  } \cK_{S, S'} := 
2\|\Ab_{S,S'}\|_2;\\
\text{\bf S2:} & ~ \cH_{S, S'} := \frac{|\cV_{S,S' | S}| |\cV_{S,S' | S'}|}{\|\Ab_{S,S'}\|_1^2}  ~~~~~~~~~~\text{ and }~~~~~~~~~~~~~~  \cK_{S, S'} := 2\|\Ab_{S,S'}\|_1.
\end{align*}
Before we prove that the constants in Settings 1 and 2 are admissible, we will show a simple and general bound on closed walks over a sequence of graphs. Let $E_1, \ldots, E_j \subset \overline E$ be fixed edge sets and $G_1 = (\overline V, E_1), G_2 = (\overline V, E_2), \ldots G_j = (\overline V, E_j)$ be graphs with vertex set $\overline V$, adjacency matrices $\Ab_1, \ldots \Ab_j$. Denote $w_{ii}$ as the number of closed walks of length $j$ starting and ending at vertex $i$ such that its $\ell$\textsuperscript{th} edge locates on $G_\ell$ for all $\ell \in [j]$. Note that $w_{ii}$, precisely equals to the $(i,i)$-{th} entry of the matrix $\Ab$, where $\Ab = \prod_{\ell \in [j]} \Ab_\ell$, i.e.
\begin{equation}\label{eq:wii}
 w_{ii} = A_{ii} = \Big[\prod_{\ell \in [j]} \Ab_k\Big]_{ii} \le \|\Ab\|_2 \leq \prod_{\ell \in [j]} \|\Ab_\ell\|_2. 
\end{equation}
 We conclude that for any fixed vertex $i \in \overline V$, we have that the number of closed walks starting and ending at vertex $i$ walking on the edges of $G_\ell$ for $\ell \in [j]$ is at most: $\prod_{\ell \in [j]} \|\Ab_\ell\|_2$. 

Following the above argument, we will prove below the admissibility of the constants in Setting 1.  More precisely we will prove the following
\begin{align} \label{tr:ineq:comb:red}
\MoveEqLeft \Tr(\Ab_0^k + \Ab_{S,S'}^k- (\Ab_0 + \Ab_{S})^k - (\Ab_0 + \Ab_{S'})^k) \nonumber \\
& \leq \Big[\frac{2{k \choose 2} (|\cV_{S, S'|S}|\wedge |\cV_{S, S'|S'}|) \|\Ab_S\|_2 \|\Ab_{S'}\|_2}{\|\Ab_{S, S'}\|_2^{2}}\Big] \|\Ab_{S, S'}\|_2^{k}, \footnotemark
\end{align}
which implies the admissibility of Setting 1, by the trivial bound $2{k \choose 2} \leq 2^{k}$. 

To prove \eqref{tr:ineq:comb:red}, we remind the reader that the trace of an adjacency matrix, counts the number of closed walks of length $k$ in the graph. A walk will only be counted in the LHS of (\ref{tr:ineq:comb:red}), if it contains an edge from the set $S$ and another edge from the set $S'$.
In the remainder of this step of the proof, we will bound the number of closed walks containing edges from both $S$ and $S'$.
                                 
Denote $\cC^{(k)}_{S,S'} = \{\text{closed walks $\cC$ of length $k$ on }$  $G(\Ab_{S,S'}), \text{with edges } e, e' \in \cC, e \neq e', e \in S, e' \in S'\}.$
We will denote closed walks of length $k$ by $\cC = v^{\cC}_0 \rightarrow v^{\cC}_1 \rightarrow \ldots \rightarrow v^{\cC}_{k-1} \rightarrow v^{\cC}_0$, where $v^{\cC}_j$ is the $j$\textsuperscript{th} vertex of $\cC$ and $(v^{\cC}_j, v^{\cC}_{j+1})$ is its $j$\textsuperscript{th} edge, and indexation is taken modulo $k$. All indexation below will also be taken modulo $k$ where applicable.
 
 \begin{figure}[H] 
\centering
\begin{tikzpicture}[scale=.7]
\SetVertexNormal[Shape      = circle,
                  FillColor  = cyan!50,
                  MinSize = 11pt,
                  InnerSep=0pt,
                  LineWidth = .5pt]
   \SetVertexNoLabel
   \tikzset{EdgeStyle/.style= {thin, above, font=\tiny, sloped,anchor=south, fill opacity=0, text opacity = 1,
                                double          = red!50,
                                double distance = 1pt}}
                                                                \begin{scope}[rotate=90]\grEmptyCycle[prefix=a,RA=2]{8}{1}\end{scope}
                                                                    \Edge[label = $\in S$](a1)(a2)
                                                                    \Edge[label = $\in S'$](a6)(a7)
                                                                    \tikzset{LabelStyle/.style = {right, fill = white, text = black, fill opacity=0, text opacity = 1}}
								\tikzset{EdgeStyle/.append style = {dashed, thin, bend right}}
                                                 \node[above left] at (a1.+90) {\tiny $v = v^{\cC}_{i + 1}$};
                                                 \node[below left] at (a2.-90) {\tiny $v^{\cC}_{i}$};
                                                 \node[above right] at (a7.+5) {\tiny $v^{\cC}_{j}$};
                                                                                                  \node[below right] at (a6.-90) {\tiny $v^{\cC}_{j + 1}$};
                                                                                                     \tikzset{EdgeStyle/.style= {dotted, thin, above, font=\tiny, sloped,anchor=south, fill opacity=0, text opacity = 1,                                double          = red!50,
                                double distance = 1pt}}
                                                                    \Edge(a0)(a1)
                                                                    \Edge(a2)(a3)
                                                                    \Edge(a3)(a4)
                                                                    \Edge(a4)(a5)                                                                                                                                              								\Edge(a5)(a6)                                                                                                                                                                                                                                                                 								\Edge(a7)(a0)
\end{tikzpicture}
\caption{A closed walk $\cC$ from the set $\cC \in \cC_k(v,o,i,j)$.}\label{graph:from:the:compl:set}\vspace{-8pt}
\end{figure}
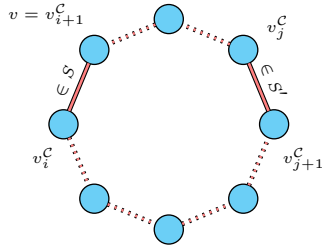

For any $0 \le i,j < k$, $i \neq j$, we first count the number of closed walks in the set 
$\cC_k(v, i,j) = \big\{\cC \in \cC^{(k)}_{S,S'}  \big| v^{\cC}_{i + 1} = v, (v^{\cC}_{i}, v^{\cC}_{i + 1}) \in S, (v^{\cC}_{j}, v^{\cC}_{j+1}) \in S', |i + 1 - j| \wedge (k -  (|i + 1 - j| \mod k)) \mbox{ is min from any edge in } S \mbox{ to any edge in } S' \mbox{ on } \cC\big\}$. 
Following \eqref{eq:wii}, we have
\begin{equation}
 |\cC_k(v, i,j)| \le  \|\Ab_S\|_2 \|\Ab_{S'}\|_2  \|\Ab_{S, S'}\|_2^{k-2}.
\end{equation}
Note that by the definition of $\cV_{S,S'|S}$ we have $v_{i + 1}^\cC \in \cV_{S,S'|S}$. To see this first consider the path segment subset of $\cC$, connecting $v_{i + 1}^\cC$ with $v_{j}^{\cC}$ (see Fig \ref{graph:from:the:compl:set} for a vizualization). All edges on the path between $v_{i + 1}^\cC$ and $v_{j}^{\cC}$ belong to the set $E_0$, or else it will not hold that it is the shortest possible path from an edge in $S$ to and edge in $S'$ on $\cC$. Hence vertex $v_{i + 1}^\cC \in V(E_0 \cup S')$ (the $S'$ comes since possibly $i + 1 = j$). On the other hand since $(v_i^{\cC}, v_{i + 1}^{\cC}) \in S$ we also have $v_{i + 1}^\cC \in V(E_0 \cup S') \cap V(S)$ and thus $v_{i + 1}^\cC \in \cV_{S,S'|S}$. 
\begin{align*}
\cC_k(i,j) := \bigcup_{v \in \overline V}\cC_k(v,i,j) \subseteq \bigcup_{v \in \cV_{S,S'|S}}\cC_k(v,i,j).
\end{align*}
Hence:
$$
|\cC_k(i,j)| \leq |\cV_{S,S'|S}|\|\Ab_S\|_2 \|\Ab_{S'}\|_2  \|\Ab_{S, S'}\|_2^{k-2},
$$
and therefore by symmetry:
\begin{align}
|\cC_k(i,j)| \leq (|\cV_{S,S'|S}| \wedge|\cV_{S,S'|S'}|) \|\Ab_S\|_2 \|\Ab_{S'}\|_2  \|\Ab_{S, S'}\|_2^{k-2},\label{useful:ineq:thm42}
\end{align}
We now observe that:
$$
\cC^{(k)}_{S,S'} \subseteq \bigcup_{0 \le i, j < k}  \cC_k(i,j)
$$
Hence applying (\ref{useful:ineq:thm42}) we obtain:
\begin{align*}
 |\cC^{(k)}_{S,S'}|  & \leq 2 {k \choose 2} (|\cV_{S,S'|S}| \wedge|\cV_{S,S'|S'}|) \|\Ab_S\|_2 \|\Ab_{S'}\|_2  \|\Ab_{S, S'}\|_2^{k-2} 
\end{align*}
which completes the proof of (\ref{tr:ineq:comb:red}).

To check the admissibility for Setting 2, just as before we must have two edges in $S$ and $S'$ respectively within a closed walk of length $k$ which is counted in the LHS of (\ref{tr:abstr:ineq}). Then, by the definition of vertex buffer set $\cV_{S,S'}$, we certainly have two vertices in the set $\cV_{S,S'}$ (one in $\cV_{S,S'|S}$ and one in  $\cV_{S,S'|S'}$). Notice that each vertex on the path is of degree at most $\|\Ab_{S,S'}\|_1$, and hence can give rise to at most $\|\Ab_{S,S'}\|_1$ continuations of the path. Therefore, having fixed two vertices from the sets $\cV_{S,S'|S}$ and $\cV_{S,S'|S'}$, we are left with at most $\|\Ab_{S,S'}\|_1^{k-2}$ paths. Taking into account that we can position the two vertices on at most $k(k-1) \leq 2^k$ spots completes the proof.

\vspace{.3cm}
\noindent {\large {\bf Step 4}\textit{ (Proof Completion).}}\\
In this step we complete the proof by arguing that if we are given any set of admissible constants $\{\cH_{S,S'}, \cK_{S,S'}\}_{S,S' \in \bC}$ the bound on the minimax risk from the Theorem statement follows.
\begin{align*}
\MoveEqLeft \sum_{k = 1}^{\infty}  \theta^k \Tr\Big(\Ab_{S,S'}^k + \Ab_0^k - (\Ab_0 + \Ab_{S})^k - (\Ab_0 + \Ab_{S'})^k \Big)/k \\
& \leq 2|S \cap S'| \theta^2 - 2 \Delta_{S,S'} \theta^3 + \cH_{S, S'} \sum_{2|k, ~ k\geq 2(d_{G_0}(S,S')\vee 1 + 1)} \frac{(\cK_{S, S'} \theta)^k}{k} \\
& \leq  2|S \cap S'| \theta^2 - 2 \Delta_{S,S'} \theta^3 + \frac{2\cH_{S, S'}(\cK_{S, S'} \theta)^{2(d_{G_0}(S,S')\vee 1 + 1)}}{2(d_{G_0}(S,S')\vee 1 + 1)},
\end{align*}
where in the last inequality we used that $(\cK_{S, S'} \theta)^2 \leq \frac{1}{2}$ by (\ref{multi:edge:NAS:theta}). In step 3 we verified that the constants $\{\cH_{S,S'}, \cK_{S,S'}\}$ given in the Theorem statement are admissible for $(G_0, \bC)$ in the sense of Definition \ref{ass:tr:abstr:ineq}. This completes the proof. 
\end{proof}
\subsection{Star Graph/Maximum Degree Proofs}\label{star:graph:max:deg:proofs}

\begin{lemma}\label{star:test:lemma} Assume the estimate $\hat \bTheta^{(1)}$ satisfies (\ref{maxnorm:prec:ass}) on the parameter space $\cM(s)$ and $\theta = r K \sqrt{\log d/\lfloor n/2\rfloor}$ for some $r \geq 2$. If $\bTheta^* \in \cS_0(s)$, we have that $\min_{e \in \widehat E} |\Theta^*_e| = 0$. Conversely when $\bTheta^* \in \cS_1(\theta, s)$ we have $\min_{e \in \widehat E} |\Theta^*_e| \geq (r-2)\theta/r$ with probability at least $1 - 1/d$.
\end{lemma}

\begin{proof}[Proof of Lemma \ref{star:test:lemma}] First, we consider the case when $\bTheta^* \in \cS_0(s)$. Since $\widehat E$ is a set of $s$ edges we must have $\min_{e \in \widehat E} |\Theta^*_e| = 0$. Next, when $\bTheta^* \in \cS_1(\theta, s)$, by (\ref{true:signals})  $\min_{e \in \widehat E} |\hat \Theta_e| > (r-1) \theta/r$. Thus by assumption (\ref{maxnorm:prec:ass}) $|\Theta^*_e| \geq (r-2)\theta/r$ with probability at least $1 - 1/d$.
\end{proof}

\begin{proof}[Proof of Corollary \ref{star:test:prop}] To handle the null hypothesis observe that by Lemma \ref{star:test:lemma} we are guaranteed to have $\min_{e \in \widehat E} |\Theta^*_e| = 0$. Invoking Proposition \ref{multiple:edge:testing:validity} gives the desired result. To show that the test is asymptotically powerful we use Proposition \ref{multiple:edge:testing:validity} in conjunction with Lemma \ref{star:test:lemma}.
\end{proof}

\section{Clique Detection Upper Bound Proof}\label{appendix:clique:detection}

\begin{lemma}\label{null:lemma:min:eigenval} Under $\Hb_0$ we have:
\begin{align}\label{min:eigen:ineq:null}
\sqrt{\hat \lambda_{\min}} & \geq 1 - \sqrt{\frac{s}{n}} - \sqrt{2\frac{s \log (ed/s) + \log \alpha^{-1}}{n}} \nonumber \\
&  \geq 1 - (\sqrt{2} + 1) \sqrt{\frac{s \log (ed/s) + \log \alpha^{-1}}{n}},
\end{align}
with probability no less than $1 - 2\alpha$.
\end{lemma}

\begin{proof}[Proof of Lemma \ref{null:lemma:min:eigenval}] The proof relies on an implication of Gordon's comparison theorem for Gaussian processes, see \cite[Corollary 5.35 e.g.]{Vershynin2012Introduction}. By this result we have that:
$$
\sqrt{\lambda_{\min}(\hat \bSigma_{CC})} \geq 1 - \sqrt{\frac{s}{n}} - t,
$$ 
with probability at least $1 - 2\exp(-nt^2/2)$ for any $t \geq 0$. Using the union bound in conjunction to the standard inequality ${d \choose s} \leq \left(\frac{ed}{s}\right)^s$, by setting $t = \sqrt{2\frac{s \log (ed/s) + \log \alpha^{-1}}{n}}$, we can ensure that (\ref{min:eigen:ineq:null}) holds.
\end{proof}

\begin{lemma}\label{alt:lemma:min:eigenval} Under $\Hb_1$ we have:
\begin{align}\label{min:eigen:ineq:alt}
\sqrt{\hat \lambda_{\min}} \leq \frac{1 + \sqrt{\frac{s}{n}} +\sqrt{\frac{2\log \alpha^{-1}}{n}} }{\sqrt{1 + (s-1)\theta}} \leq \frac{1 + (\sqrt{2} + 1)\sqrt{\frac{s\log (e d/s) + \log \alpha^{-1}}{n}}}{\sqrt{1 + (s-1)\theta}},
\end{align}
with probability at least $1 - 2\alpha$.
\end{lemma}

\begin{proof}[Proof of Lemma \ref{alt:lemma:min:eigenval}] Taking in mind that $\bSigma = \bTheta^{-1}$ by a simple calculation one can verify that for the set $C^* = \supp(\vb)$, $\lambda_{\min}(\bSigma_{C^*C^*}) = [\lambda_{\max}(\bTheta_{C^* C^*})]^{-1} = \frac{1}{1 + (s-1)\theta}$ with a corresponding eigenvector $\frac{\vb_{C^*}}{\sqrt{s}}$.  Again using Corollary 5.35 of \cite{Vershynin2012Introduction}, and the fact that for two symmetric psd matrices $\Ab, \Bb$ we have $\lambda_{\min} (\Ab \Bb \Ab) \leq \lambda_{\min}(\Ab)^2 \lambda_{\max}(\Bb)$, we have:
$$
\sqrt{\lambda_{d}(\hat \bSigma_{C^*C^*})} \leq \frac{1 + \sqrt{\frac{s}{n}} +\sqrt{\frac{2\log \alpha^{-1}}{n}} }{\sqrt{1 + (s-1)\theta}},
$$
with probability at least $1 - 2 \alpha$.
\end{proof}

\begin{proof}[Proof of Proposition \ref{min:eigenval:test:prop}] 
 Combining the results of Lemma \ref{null:lemma:min:eigenval} and setting $\alpha = d^{-1}$ in Lemma \ref{alt:lemma:min:eigenval} it suffices to show there will be a gap between the bounds in (\ref{min:eigen:ineq:null}) and (\ref{min:eigen:ineq:alt}). By simple algebra when
 $$\theta > \kappa\sqrt{ \frac{\log (e d/s)}{sn}},$$
for a sufficiently large $\kappa$
the gap between (\ref{min:eigen:ineq:null}) and (\ref{min:eigen:ineq:alt}) is implied, which completes the proof.
\end{proof}

\section{Auxiliary Results}\label{sec:aux}

Recall that for a symmetric matrix $\Ab \in \RR^{d \times d}$ we denote its eigenvalues in decreasing with $\lambda_1(\Ab) \geq \lambda_2(\Ab) \geq \ldots \geq \lambda_d(\Ab)$.

\begin{theorem}[Lidskii's Inequality \citep{helmke1995eigenvalue}]\label{lidskiis:ineq} For two symmetric $m \times m$ matrices $\Ab$ and $\Bb$ we have:
$$
\sum_{j = 1}^k \lambda_{i_j}(\Ab + \Bb) \leq \sum_{j = 1}^k \lambda_{i_j}(\Ab) + \sum_{j = 1}^k \lambda_{j}(\Bb),
$$
for any $1 \leq i_1 < \ldots < i_k \leq m.$
\end{theorem}

We now derive Lemma \ref{gen:lidskii} as a Corollary to Theorem \ref{lidskiis:ineq}. Lemma \ref{gen:lidskii} can also be viewed as a more general formulation of the latter result.

\begin{proof}[Proof of Lemma \ref{gen:lidskii}] First notice that since 
$$
\sum_{j = 1}^m \lambda_j(\Ab + \Bb) = \sum_{j = 1}^m \lambda_j(\Ab) + \sum_{j = 1}^m \lambda_{j}(\Bb),
$$
it suffices to show the bound when for all $j$ we have $c_j \geq 0$. Let $C = \{j ~|~ c_j \neq 0\}$. We will induct on the cardinality $|C|$, which we denote with $k$. When $k = 1$, the inequality immediately follows by Theorem \ref{lidskiis:ineq}. Suppose the inequality hods for $k = l$. Notice that the inequality can equivalently be expressed as
\begin{align}\label{almost:there:gen:lidskii}
\sum_{j = 1}^m c_{j} \lambda_{\sigma^{-1}(j)}(\Ab + \Bb) \leq \sum_{j = 1}^m c_{j}\lambda_{\sigma^{-1}(j)}(\Ab) + \sum_{j = 1}^m c_j \lambda_{j}(\Bb),
\end{align}
To see why it holds for $k = l+1$ let $c^* = \min_{j \in C} c_j$, and notice that by Theorem \ref{lidskiis:ineq} we have:
$$
\sum_{j \in C} c^* \lambda_{\sigma^{-1}(j)}(\Ab + \Bb) \leq  \sum_{j \in C} c^* \lambda_{\sigma^{-1}(j)}(\Ab) + \sum_{j \in C}c^* \lambda_{j}(\Bb).
$$
This implies that inequality (\ref{almost:there:gen:lidskii}) follows by an inequality in which we subtract $c^*$ from all $c_j, j \in C$. This completes the proof by the induction hypothesis. 
\end{proof}

\section{Bootstrap Validity}\label{sec:boot-proof}

For a real valued random variable $X$ and a random $\ell \geq 1$, we define the Orlicz-norm $\psi_{\ell}$ as:
\begin{align} \label{psi1norm}
	\|X\|_{\psi_\ell} = \sup_{p \geq 1} p^{-1/\ell} (\EE|X|^p)^{1/p}.
\end{align} 
We mainly use the $\psi_1$ and $\psi_2$ norms. Recall that random variables with bounded $\psi_1$ and $\psi_2$ norms are called \textit{sub-exponential} and \textit{sub-Gaussian} \cite[e.g.]{Vershynin2012Introduction}.

\subsection{Minimal Structure Certification}\label{sec:supp-1}

\begin{lemma}\label{sub:gauss:tails:lemma} The random variable $\bX \sim N(0,(\bTheta^*)^{-1})$ satisfies:
\begin{align}
\sup_{\|\alpha\|_2 = 1}\|\balpha^T \bX\|_{\psi_2} \leq \|Z\|_{\psi_2}C < C,
\end{align}
where $Z\sim N(0,1)$.
\end{lemma}

\begin{proof}[Proof of Lemma \ref{sub:gauss:tails:lemma}]
The first part follows since we assume $\|\bTheta^{*-1}\|_2 \leq C$, and the second inequality follows via a direct calculation.
\end{proof}

\begin{lemma} \label{lower:bound:var:lemma} We have
$$
\inf_{j,k \in [d]} \Delta_{jk} \geq C^{-2}. 
$$
\end{lemma}
\begin{proof}[Proof of Lemma \ref{lower:bound:var:lemma}]
By Isserlis' theorem we have $
\Delta_{jk} = \Theta^*_{jj}\Theta^*_{kk} + \Theta^*_{jk}\Theta^*_{kj} \geq \Theta^*_{jj}\Theta^*_{kk} \geq C^{-2} > 0$. 
\end{proof}

\begin{lemma}[Covariance Concentration] \label{cov:conc} There exists a universal constant $R > 0$, such that:
$$
\|\hat \bSigma - \bSigma^*\|_{\max} \leq R C^{2} \sqrt{\log d/n},
$$
with probability at least $1 - 1/d$ uniformly over the parameter space $\cM$.
\end{lemma}

\begin{proof}[Proof of Lemma \ref{cov:conc}] The proof follows by applying the inequality 
$$\|X_i X_j\|_{\psi_1} \leq 2\|X_i\|_{\psi_2} \|X_j\|_{\psi_2} \leq 2 c^{-2}$$ 
(see Lemma \ref{sub:gauss:tails:lemma}) in combination with Proposition 5.16 \citep{Vershynin2012Introduction} and the union bound.
\end{proof}

\begin{lemma}\label{unif:conv} For large enough $n$ and $d$ we have: 
$$
\sup_{\bTheta^* \in \cM(s)}  \PP(\max_{j,k \in [d]} \sqrt{n} |(\tilde \Theta_{jk} - \Theta^*_{jk}) + \bTheta^{*T}_{*j}(\hat\bSigma - \bSigma^*) \bTheta^*_{*k} | > \xi_1) < \xi_2,
$$
where $\xi_1 = \Xi_1 s \log d/\sqrt{n}$, $\xi_2 = 2/d$, where $\Xi_1$ is an absolute constant depending solely on $K, L$ and $C$.
\end{lemma}

\begin{proof}[Proof of Lemma \ref{unif:conv}] By elementary algebra we obtain the following representation:
\begin{align}\label{easy:exp}
\MoveEqLeft (\tilde \Theta_{jk} - \Theta^*_{jk}) + \bTheta^{*T}_{*j}(\hat\bSigma - \bSigma^*) \bTheta^*_{*k} \nonumber \\
& = -\underbrace{\frac{\hat \bTheta_{*j}^T \hat \bSigma_{*\setminus j} (\hat \bTheta_{\setminus j,k} - \bTheta^*_{\setminus j,k}) + (\hat \bTheta_{*j} - \bTheta_{*j}^*)^T(\hat \bSigma - \bSigma^*)\bTheta^*_{*k}}{\hat \bTheta^T_{*j} \hat \bSigma_{*j}}}_{I^{jk}_1} \nonumber\\
& - \underbrace{\bTheta^{*T}_{*j}(\hat \bSigma - \bSigma^*) \bTheta^*_{*k}\bigg[\frac{1}{\hat \bTheta^T_{*j} \hat \bSigma_{*j}} - 1\bigg]}_{I_2^{jk}},
\end{align}
where indexing with $\setminus j$ means dropping the corresponding column or element from the matrix or vector respectively.  We deal with the terms $I^{jk}_1$ first. By the triangle inequality followed by H\"{o}lder's inequality we have:
\begin{align*}
\max_{j,k \in [d]} |I^{jk}_1|&  \leq  \max_{j,k \in [d]} \{|\hat \bTheta^T_{*j} \hat \bSigma_{*j}|^{-1} (\|\hat \bTheta^T_{*j} \hat \bSigma_{*\setminus j}\|_{\infty} \|\hat \bTheta_{\setminus j,k} - \bTheta^*_{\setminus j,k}\|_1 \\
& + \|\hat \bTheta^{}_{*j} - \bTheta^{*}_{*j}\|_1\|\hat \bSigma - \bSigma^*\|_{\max} \|\bTheta^*_{*k}\|_1)\}
\end{align*}
Since it is assumed that $\log d/n = o(1)$, take $d$ and $n$ large enough so that $\log d/n \leq 1/(4K^2)$. Let $E$ be the event where (\ref{maxnorm:prec:ass}) and (\ref{other:norms:prec:ass}) hold. Then on $E$ we have:
$$
\max_{j \in [d]} |\hat \bTheta^T_{*j} \hat \bSigma_{*j} - 1| \leq K \sqrt{\log d/n}.
$$
This implies that on the same event $\sup_{j \in [d]}|\hat \bTheta^T_{*j} \hat \bSigma_{*j}|^{-1} \leq 2$. Continuing our bounds on the event $E$ we have:
$$
\|\hat \bTheta^T_{*j} \hat \bSigma_{*\setminus j}\|_{\infty} \|\hat \bTheta_{\setminus j,k} - \bTheta^*_{\setminus j,k}\|_1 \leq K^2 s \log d/n = o(n^{-1/2}).
$$
Furthermore on the intersection of $E$ and the event from Lemma \ref{cov:conc}, which holds with probability no less than $1 - 2/d$, we have:
$$
\|\hat \bTheta^{}_{*j} - \bTheta^{*}_{*j}\|_1\|\hat \bSigma - \bSigma^*\|_{\max} \|\bTheta^*_{*k}\|_1 \leq R C^{2} K L s \log d/n = o(n^{-1/2}).
$$
Combining the last four bounds we conclude that $\max_{j,k \in [d]} |I^{jk}_1| = o_p(n^{-1/2})$. Next we handle the terms $|I_2^{jk}|$. We have:
$$
\max_{j,k \in [d]} |I^{jk}_2| \leq \max_{j, k \in [d]} \left\{|\bTheta^{*T}_{*j}(\hat \bSigma - \bSigma^*) \bTheta^*_{*k}|\bigg|\frac{1}{\hat \bTheta^T_{*j} \hat \bSigma_{*j}} - 1\bigg|\right\}.
$$
Clearly on the event $E$ under the assumption $\log d/n \leq 1/(4K^2)$, we have that:
\begin{align}\label{var:bound}
\max_{j \in [d]}\bigg|\frac{1}{\hat \bTheta^T_{*j} \hat \bSigma_{*j}} - 1\bigg| \leq 2 K \sqrt{\log d/n}. 
\end{align}
Next we consider the random variables $\bTheta^{*T}_{*j} \bX^{\otimes 2}\bTheta^{*}_{*k}$ for all $j,k \in [d]$. By Lemma \ref{sub:gauss:tails:lemma} we have:
$$
\|\bTheta^{*T}_{*j} \bX^{\otimes 2}\bTheta^{*}_{*k}\|_{\psi_1} \leq 2 \|\bTheta^{*T}_{*j} \bX\|_{\psi_2} \|\bX^T\bTheta^{*}_{*k}\|_{\psi_2} \leq 2\sup_{j \in [d]}\|\bTheta^{*}_{*j}\|^2_2 C^{2} \leq 2 C^4,
$$
with the last inequality following by the fact that $\sup_{j \in [d]}\|\bTheta^{*}_{*j}\|^2_2 \leq \|\bTheta^*\|^2_2 \leq C^2$ since $\bTheta^* \in \cM(s)$. Clearly then $\|\bTheta^{*T}_{*j} \bX^{\otimes 2}\bTheta^{*}_{*k} - \Theta^*_{jk}\|_{\psi_1} \leq  4 C^4$. Finally by the union bound and Proposition 5.16 \citep{Vershynin2012Introduction}, one concludes that there exists an absolute constant $\tilde C$, such that when $\sqrt{\log d/n} \leq 4 C^4$, we have:
\begin{align}\label{bound:theta:sigma}
\max_{j, k \in [d]} |\bTheta^{*T}_{*j}(\hat \bSigma - \bSigma^*) \bTheta^*_{*k}| \leq 4 \tilde C  C^4 \sqrt{\log d/n},
\end{align}
with probability at least $1 - 2/d$. Combining this bound with (\ref{var:bound}) and our conditions we conclude that $\max_{j,k\in[d]} |I^{jk}_2| = o_p(n^{-1/2})$. This completes the proof.
\end{proof}

\begin{lemma} \label{main:Mbound} Let $\{\bX_i\}_{i = 1}^n$ be identical (not necessarily independent), $d$-dimensional sub-Gaussian vectors with $\max_{i \in [n], j \in[d]} \|X_{ij}\|_{\psi_2} = C$. Then there exists an absolute constant $U > 0$ depending solely on $C$ such that :
\begin{align*}
\max_{i \in [n]} \|\bX_i^{\otimes 2}\|_{\max} < U\log(nd),
\end{align*}
with probability at least $1 - 1/d$.
\end{lemma}

\begin{proof}[Proof of Lemma \ref{main:Mbound}] The proof is trivial upon noting that,

$\max_{i \in [n]} \|\bX_i^{\otimes 2}\|_{\max} = \max_{i \in [n]} \|\bX_i\|^2_{\infty}$ and combining (5.10) in \cite{Vershynin2012Introduction} with the union bound. We omit the details.
\end{proof}

\begin{lemma}\label{delta:consist} Assuming that $s\sqrt{\log d/n} \sqrt{\log(nd)} = o(1)$, we have that for any fixed $j,k \in [d]$:
$$
\lim_{n} \inf_{\bTheta^* \in \cM(s)} \PP(|\hat \Delta_{jk} - \Delta_{jk}| \leq \tau(n)) = 1,
$$
where $\tau(n) = o(1)$.
\end{lemma}
\begin{proof}[Proof of Lemma \ref{delta:consist}] 
We start by showing that:
$$
I := \max_{j,k\in[d]} n^{-1}\sum_{i = 1}^n (\hat \bTheta_{*j}^T \bX_i^{\otimes 2}\hat \bTheta_{*k} - \hat \Theta_{jk} -  \bTheta^{*T}_{*j} \bX_i^{\otimes 2} \bTheta^*_{*k} +  \Theta^*_{jk})^2,
$$
is asymptotically negligible. Since $(a-b)^2 \leq 2(a^2 + b^2)$, we have that:
$$
|I| \leq 2 \underbrace{\max_{j,k\in[d]} n^{-1} \bigg[\sum_{i = 1}^n (\hat \bTheta_{*j}^T \bX_i^{\otimes 2}\hat \bTheta_{*k}  -  \bTheta^{*T}_{*j} \bX_i^{\otimes 2} \bTheta^*_{*k})^2\bigg]}_{I_{1}}  + 2 \underbrace{\max_{j,k\in[d]} (\hat \Theta_{jk} -  \Theta^*_{jk})^2}_{I_2}.
$$
Clearly by assumption (\ref{maxnorm:prec:ass})
$$
|I_2| \leq K^2 \log d/n,
$$
with probability at least $1 - 1/d$. Next we handle $I_1$. By the triangle inequality the following holds:
\begin{align*}
I_1^{1/2} & \leq \underbrace{\max_{j,k\in[d]}  \bigg[n^{-1}\sum_{i = 1}^n (\hat \bTheta_{*j}^T \bX_i^{\otimes 2}(\hat \bTheta_{*k}  -  \bTheta^*_{*k}))^2\bigg]^{1/2}}_{I_{11}} \\
& + \underbrace{\max_{j,k\in[d]}  \bigg[n^{-1}\sum_{i = 1}^n ((\hat \bTheta_{*j} - \bTheta^*_{*j})^T \bX_i^{\otimes 2} \bTheta^*_{*k})^2\bigg]^{1/2}}_{I_{12}}.
\end{align*}
We first handle $I^2_{11}$:
$$
I^2_{11} = \max_{j,k\in[d]}  (\hat \bTheta_{*k} - \bTheta^*_{*k})^T \underbrace{\frac{1}{n}\sum_{i = 1}^n \bX_i^{\otimes 2} \hat \bTheta_{*j}^{\otimes 2}  \bX_i^{\otimes 2}}_{\Mb}  (\hat \bTheta_{*k} - \bTheta^*_{*k}).
$$
Using Lemma \ref{main:Mbound}, we can handle $\Mb$ in the following way:
\begin{align}
\|\Mb\|_{\max} & \leq \max_{i \in [d]} \|\bX_i^{\otimes 2}\|_{\max} \max_{j\in[d]}  \frac{1}{n} \sum_{i = 1}^n \hat \bTheta_{*j}^T \bX_i^{\otimes 2} \hat \bTheta_{*j} \nonumber \\
& \leq O_p(\log(nd)) \max_{j\in[d]} \|\hat \bTheta_{*j}\|_1 \|\hat \bSigma \hat \bTheta_{*j}\|_{\infty}. \label{main:MboundpropCLIME}
\end{align}
By the definition of $\hat \bTheta_{*j}$ we have:  $\max_{j\in[d]}  \|\hat \bSigma \hat \bTheta_{*j}\|_{\infty} \leq (1 + K\sqrt{\log d/n})$. Hence:
\begin{align*}
\|\Mb\|_{\max} & \leq O_p(\log(nd)) \max_{j\in[d]}  (\|\bTheta^*_{*j}\|_1 + \|\hat \bTheta_{*j} - \bTheta^*_{*j}\|_1) (1 + K\sqrt{\log d/n}) \\
& = O_p(\log(nd))L,
\end{align*}
where we used assumption (\ref{other:norms:prec:ass}). Thus:
$$
|I_{11}| \leq \max_{k\in[d]}  \|\hat \bTheta_{*k} - \bTheta^*_{*k} \|_1^2 L O_p(\log(nd)) =  O_p\left(s^2\log d/n \log(nd) \right) = o_p(1),
$$
with probability at least $1 - 2/d$ uniformly over $\cM(s)$. By a similar argument we can show that $I_{12}$ is of the same order. Putting everything together we conclude: 
\begin{align}\label{square:rate}
I = O_p\left(s^2\log d \log(nd)/n \right) = o_p(1).
\end{align}
Next, we argue that for any fixed $j,k$ we have that the difference 
$$n^{-1}\sum_{i = 1}^n (\bTheta^{*T}_{*j}\bX_i^{\otimes 2}\bTheta^{*}_{*j} - \Theta^*_{jk})^2 - \EE( \bTheta_{*j}^T \bX^{\otimes 2} \bTheta_{*k} -  \Theta^*_{jk})^2$$ 
is small. Since as in Lemma \ref{unif:conv} we have $\|\bTheta^{*T}_{*j} \bX^{\otimes 2} \bTheta^*_{*k}\|_{\psi_1} \leq 2 C^4$, then certainly $\Var((\bTheta^{*T}_{*j} \bX^{\otimes 2} \bTheta^*_{*k} - \Theta^*_{jk})^2)$ is finite for any fixed $j,k\in[d]$. A usage of Chebyshev's inequality gives us that:
$$n^{-1}\sum_{i = 1}^n (\bTheta^{*T}_{*j}\bX_i^{\otimes 2}\bTheta^{*}_{*j} - \Theta^*_{jk})^2- \EE( \bTheta_{*j}^T \bX^{\otimes 2} \bTheta_{*k} -  \Theta^*_{jk})^2 = o_p(1).$$ 
Finally note that by the triangle inequality the following two inequalities hold for any fixed $j,k \in [d]$:
\begin{align*}
\bigg[n^{-1} \sum_{i = 1}^n (\hat \bTheta_{*j}^T \bX_i^{\otimes 2}\hat \bTheta_{*k} - \hat \Theta_{jk} )^2 \bigg]^{1/2}& \leq \bigg[n^{-1} \sum_{i = 1}^n(\bTheta^{*T}_{*j} \bX_i^{\otimes 2} \bTheta^*_{*k} -  \Theta^*_{jk})^2\bigg]^{1/2} + I^{1/2},\\
\bigg[n^{-1} \sum_{i = 1}^n(\bTheta^{*T}_{*j} \bX_i^{\otimes 2} \bTheta^*_{*k} -  \Theta^*_{jk})^2\bigg]^{1/2}& \leq \bigg[n^{-1} \sum_{i = 1}^n (\hat \bTheta_{*j}^T \bX_i^{\otimes 2}\hat \bTheta_{*k} - \hat \Theta_{jk} )^2 \bigg]^{1/2}+ I^{1/2}.
\end{align*}
Observe that $n^{-1}\sum_{i = 1}^n (\bTheta^{*T}_{*j} \bX_i^{\otimes 2} \bTheta^*_{*k} - \Theta^*_{jk})^2 = \EE( \bTheta_{*j}^T \bX^{\otimes 2} \bTheta_{*k} -  \Theta^*_{jk})^2 + o_p(1) = O_p(1)$. This completes the proof.
\end{proof}

\begin{lemma} \label{normallemma} Let $X_n(\br)$ and $\xi_n(\br)$ be two sequences of random variables, depending on a parameter $\br \in \bR$. Suppose that $\lim_{n} \sup_{\br \in \bR} \sup_{t } |\PP(X_n(\br) \leq t) - F(t)| = 0$ where $F$ is a continuous cdf, and $\lim_{n} \inf_{\br \in \bR} \PP(|1 - \xi_n(\br)| \leq \tau(n)) = 1$ for $\tau(n) = o(1)$. Assume in addition that $F$ is Lipschitz, i.e., there exist $\kappa > 0$, such that for any $t, s \in \RR: |F(t) - F(s)| \leq \kappa |t - s|$. Then we have:
$$
\lim_{n}  \sup_{\br \in \bR} \sup_{t} \bigg|\PP\bigg(\frac{X_n(\br)}{\xi_n(\br)} \leq t\bigg) - F(t)\bigg| = 0.
$$
\end{lemma}

\begin{proof} The proof goes through by a direct argument. We omit the details.
\end{proof}

\begin{proof}[Proof of Proposition \ref{multiple:edge:testing:validity}] The proof of the strong error rate control is based upon verifying conditions from Theorem 5.1 of \cite{Chernozhukov2013Gaussian}. The rate $(\log(dn))^6/n = o(1)$ in (\ref{bootstrap:rates}) holds due to recent advances in a Berry-Esseen bound due to \cite{chernozhukov2014central}. Define:
\begin{align*}
\Upsilon_1 & := \sqrt{n} \max_{(j,k) \in E} |(\tilde \Theta_{jk} - \Theta^*_{jk}) + \bTheta^{*T}_{*j}(\hat\bSigma - \bSigma^*) \bTheta^*_{*k} |, \\
\Upsilon_2 & := \max_{(j,k) \in E} n^{-1}\sum_{i = 1}^n (\hat \bTheta^{T}_{*j}\bX_i^{\otimes 2}\hat \bTheta^{}_{*k} - \hat \Theta_{jk} - \bTheta^{*T}_{*j}\bX_i^{\otimes 2}\bTheta^{*}_{*k} + \Theta^*_{jk})^2.
\end{align*}
By Lemma \ref{unif:conv}, we are guaranteed that $\sup_{\bTheta^* \in \cM_{E^{\rm n}, E^{\rm nc}}} \PP(\sqrt{\log |E|}|\Upsilon_1| \geq \sqrt{\log |E|}\xi_1) < \xi_2$, where $\xi_1 = \Xi_1 s \log d/\sqrt{n}$, $\xi_2 = 2/d$ and since under our assumptions $\sqrt{\log |E|}\xi_1 = o(1)$ and $\xi_2 = o(1)$, this satisfies\footnote{In fact assumption (M) requires a stronger control but it is not necessary for our purporse.} the first requirement in condition (M) (i) of Theorem 5.1.  \citep{Chernozhukov2013Gaussian}. 

Next we turn to bounding $|\Upsilon_2|$. Using (\ref{square:rate}) from Lemma \ref{delta:consist} we conclude that there exists a constant $C > 0$ such that $\underset{\bTheta^* \in \cM_{E^{\rm n}, E^{\rm nc}}}{\sup}\PP((\log(|E|n))^2 \Upsilon_2 \geq C s^2 (\log(|E|n))^2 \log(nd)\log d /n ) \leq 2/d$. Clearly since $|E| \leq {d \choose 2} < d^2$ by our assumption we have $s^2(\log(|E|n))^2 \log(nd)\log d/n = o(1)$. Finally, recall that $\EE(\bTheta^{*T}_{*j}\bX_i^{\otimes 2}\bTheta^{*}_{*k} - \Theta^*_{jk})^2 \geq \delta$ and $\|\bTheta^{*T}_{*j}\bX_i^{\otimes 2}\bTheta^{*}_{*k} \|_{\psi_1} \leq 2 C^4$ as we saw in the proof of Lemma \ref{unif:conv}, and in addition by assumption $(\log(dn))^6/n = o(1)$. This completes the verification of conditions needed by Theorem 5.1 of \cite{Chernozhukov2013Gaussian} and shows that (\ref{strong:FWER:control}) holds.

To see the second part it is sufficient to show that on the first iteration of Algorithm \ref{step:down}, all edges will be rejected. We first obtain a high probability bound on the quantile of $\sqrt{n} \max_{(j,k) \in E} |\tilde \Theta_{jk} - \Theta^*_{jk}|$. By Lemma \ref{unif:conv} we get that:
$$
\sqrt{n} \max_{(j,k) \in E} |\tilde \Theta_{jk} - \Theta^*_{jk}| \leq \sqrt{n} \max_{(j,k) \in E} |\bTheta^{*T}_{*j}(\hat\bSigma - \bSigma^*) \bTheta^*_{*k}| + \Xi_1 s\log d/\sqrt{n},
$$
with probability at least $1 - 2/d$, where $\Xi_1$ is an absolute constant. Next recall that by (\ref{bound:theta:sigma}) when $\sqrt{\log d/n} \leq 4 C^4$ on the same event as above we have:
\begin{align*}
\max_{j, k \in [d]} |\bTheta^{*T}_{*j}(\hat \bSigma - \bSigma^*) \bTheta^*_{*k}| \leq 4 \tilde C  C^4 \sqrt{\log d/n}.
\end{align*} 
Putting last two inequalities together we conclude:
\begin{align}\label{max:theta:tilde}
\sqrt{n} \max_{(j,k) \in E} |\tilde \Theta_{jk} - \Theta^*_{jk}| \leq 4 \tilde C  C^4 \sqrt{\log d} + \Xi_1 s\log d/\sqrt{n},
\end{align}
with probability no less than $1 - 2/d$. The last of course implies that any quantile of $\sqrt{n} \max_{(j,k) \in E} |\tilde \Theta_{jk} - \Theta^*_{jk}|$ will be smaller than the value on the RHS with high probability. Next, by Corollary 3.1\footnote{Formally we verified conditions of Theorem 5.1, but in fact they imply Corollary 3.1 holds.} of \cite{Chernozhukov2013Gaussian} we have:
\begin{align}
\sup_{\alpha \in (0,1)} |\PP(\sqrt{n} \max_{(j,k) \in E} |\tilde \Theta_{jk} - \Theta^*_{jk}| \leq c_{1 - \alpha, E}) - (1 - \alpha)| = o(1). \label{quantile}
\end{align}
Combining (\ref{max:theta:tilde}) and (\ref{quantile}) gives us that for large enough $d,n$, for any fixed $\alpha > 0$ we have:
$$
c_{1 - \alpha, E} \leq 4 \tilde C C^4 \sqrt{\log d} + \Xi_1 s\log d/\sqrt{n}.
$$
Now, clearly 
\begin{equation}\label{eq:tjk}
 \sqrt{n} |\tilde \Theta_{jk}| \geq \sqrt{n}|\kappa \sqrt{\log d/n}| - \sqrt{n} \max_{(j,k) \in E} |\tilde \Theta_{jk} - \Theta^*_{jk}| \geq c_{1 - \alpha, E},
\end{equation}
provided that $|\kappa| \geq 9\tilde C C^4$, and hence the power goes to $1$ uniformly over the parameter space $\cM_{E^{\rm n},E^{\rm nc}}$. 
\end{proof} 
\end{supplement}

%
%
%


\end{document}